\documentclass[12pt]{article}

\usepackage{latexsym,amsmath,amscd,amssymb,framed,graphics,mathrsfs}
\usepackage{enumerate}

\usepackage{graphicx}

\usepackage[colorlinks]{hyperref}
\usepackage{url}
\usepackage{colordvi}
\usepackage{color}
\usepackage{bbm}
\usepackage{cancel}

\usepackage{soul}
\setstcolor{blue}

\usepackage[all]{xy}

\makeatletter

\@addtoreset{figure}{section}
\def\thefigure{\thesection.\@arabic\c@figure}
\def\fps@figure{h, t}
\@addtoreset{table}{bsection}
\def\thetable{\thesection.\@arabic\c@table}
\def\fps@table{h, t}
\@addtoreset{equation}{section}

\makeatother

\newenvironment{proof}{\paragraph{Proof:}}{\hfill$\square$}

\newtheorem{theorem}{Theorem}[section]

\newtheorem{lemma}[theorem]{Lemma}
\newtheorem{remark}[theorem]{Remark}
\newtheorem{proposition}[theorem]{Proposition}

\begin{document}

\title{Towards a geometric variational discretization of compressible fluids: the rotating shallow water equations}
\author{Werner Bauer$^{1,2}$ and Fran\c{c}ois Gay-Balmaz$^2$}
\addtocounter{footnote}{1} 
\footnotetext{Imperial College London, Department of Mathematics, London, UK.
\texttt{werner.bauer@lmd.ens.fr}}
\addtocounter{footnote}{1} 
\footnotetext{CNRS and \'Ecole Normale Sup\'erieure, Laboratoire de M\'et\'eorologie Dynamique, Paris, France. 
\texttt{gaybalma@lmd.ens.fr}}

\date{}
\maketitle

\begin{abstract} 
This paper presents a geometric variational discretization of compressible fluid dynamics.
The numerical scheme is obtained by discretizing, in a structure preserving way, the Lie group 
formulation of fluid dynamics on diffeomorphism groups and the associated variational principles. 
Our framework applies to irregular mesh discretizations in 2D and 3D. It systematically extends 
work previously made for incompressible fluids to the compressible case. We consider 
in detail the numerical scheme on 2D irregular simplicial meshes and evaluate the scheme numerically for the rotating 
shallow water equations. In particular, we investigate whether the scheme conserves stationary solutions, represents well 
the nonlinear dynamics, and approximates well the frequency relations of the continuous equations, while preserving conservation 
laws such as mass and total energy.
\end{abstract}

\section{Introduction}\label{intro} 
 
This paper develops a geometric variational discretization for compressible fluid dynamics.
Geometric integrators form a particular class of numerical schemes which aim to preserve the 
intrinsic geometric structures of the equations they discretize. As a consequence, such schemes 
are well-known to correctly reproduce the conservation laws and the global behavior of the 
underlying dynamical system, \cite{HaLuWa2002}.

One efficient way to produce geometric integrators is to exploit the variational formulation of 
the continuous equations and to mimic this formulation at the spatial and/or temporal discrete 
level. For instance, in classical mechanics, a time discretization of the Lagrangian variational 
formulation allows for the derivation of numerical schemes, called variational integrators, that 
are symplectic, exhibit good energy behavior, and inherit a discrete version of Noether's theorem 
which guarantees the exact preservation of momenta arising from symmetries, see \cite{MaWe2001}. 
An extension of this approach to the context of certain partial differential equations may be made 
through an appeal to their spacetime variational formulation resulting in multisymplectic schemes, 
\cite{MaPaSh1998}, \cite{LeMaOrWe2003}, see, e.g.,  \cite{DeGBRa2016}, \cite{DeGBKoRa2016}, \cite{GBPu2016} 
for recent developments in variational multisymplectic integrators.

The development of geometric variational integrators for the partial differential equations of incompressible 
fluid dynamics has been initiated in \cite{PaMuToKaMaDe2011} for the Euler equations of a perfect fluid. 
This approach exploits the geometric formulation due to \cite{Arnold1966} which interprets the motion of 
the ideal fluid as a geodesic curve on the group of volume preserving diffeomorphisms of the fluid domain. 
As a consequence of this interpretation, the fluid equations arise 
in the Lagrangian description from the Hamilton variational principle on the group of volume preserving 
diffeomorphisms, for the Lagrangian given by the kinetic energy of the fluid. By making use of the relabelling 
symmetry of the fluid, this principle naturally induces a variational principle in the Eulerian description on 
the Lie algebra of this group, namely, the space of divergence free vector fields. In \cite{PaMuToKaMaDe2011} this 
variational geometric formulation is implemented on a finite dimensional Lie group discretization of the group 
of volume preserving diffeomorphisms. A main feature of the discrete level is the occurrence of nonholonomic 
constraints which require the use of the Lagrange-d'Alembert principle, a variant of Hamilton's principle 
applicable to nonholonomic systems. The spatially discretized Euler equations emerge from an application 
of this principle on the finite dimensional Lie group approximation. This approach was 
extended in \cite{GaMuPaMaDe2011} to various equations of incompressible fluid dynamics with advected quantities, 
such as MHD and liquid crystals. The development of this geometric method for rotating and stratified fluids for 
atmospheric and oceanic dynamics was given in \cite{DeGaGBZe2014}.
Improvements of this variational method in efficiency, generality, and controllability were achieved 
in \cite{LiMaHoToDe2015}. In \cite{BaGB2016}, the geometric variational 
discretization was extended to anelastic and pseudo-incompressible fluids on 2D irregular simplicial meshes.

In the present paper, we develop this geometric variational discretization towards the treatment of compressible 
fluid dynamics. This extension is based on a suitable Lie group approximation of the group of  (not necessarily 
volume preserving) diffeomorphisms of the fluid domain, accompanied with an appropriate right invariant nonholonomic 
constraint obtained by requiring that Lie algebra elements are approximations of continuous vector fields. 
The spatial discretization of the compressible fluid equations is obtained by an application of the Lagrange-d'Alembert 
principle on the Lie group of discrete diffeomorphisms, for a semidiscrete Lagrangian which is assumed to have the same 
relabelling symmetries as the continuous Lagrangian. From these symmetries, one deduces the Eulerian version of this 
principle and gets the discrete equations by computing the critical curves. This geometric setting is independent of 
the choice of the mesh discretization of the fluid domain.

\color{black}
As we will see later in the paper, extending the discrete diffeomorphism approach from the incompressible to the compressible case requires nontrivial steps.
Simply removing the incompressibility condition in the definition of the group of discrete volume preserving diffeomorphisms defined in \cite{PaMuToKaMaDe2011} is not enough, since this results in a group that is too large as it contains discrete diffeomorphisms that have no physical significance. This difficulty is overcome by the introduction of a nonholonomic constraint that imposes the Lie algebra elements to correspond to a discrete vector field. Note that this extension to the compressible case was not needed for the variational discretization of anelastic and pseudo-incompressible fluids in \cite{BaGB2016}. Indeed, as shown in \cite{BaGB2016}, in the continuous case these equations can be derived from the Hamilton principle written on a group of diffeomorphisms that preserve a modified volume form. Hence, as opposed to the present case, the variational discretization can still be done via a slight modification of the group of volume preserving discrete diffeomorphisms introduced in \cite{PaMuToKaMaDe2011}.
 
\color{black}
\bigskip

\paragraph{Plan of the paper.} 
We end this introduction by giving a quick overview of the geometric variational discretization of incompressible fluids. 
In Sect.~\ref{Section_2}, we describe the geometric setting for the discretization of compressible fluids. In particular, 
for a given mesh on the fluid domain, we introduce the associated group of discrete diffeomorphisms, its Lie algebra, the 
nonholonomic constraint, and we identify the appropriate dual spaces and projections that allow us to apply the 
Lagrange-d'Alembert principle. In Sect.~\ref{section_3}, we derive the discrete compressible fluid equations on 2D irregular 
simplicial grids. In particular, we identify the expression of the discrete Lie derivative of one-form densities on such meshes.
In Sect.~\ref{sec_explizit_scheme}, we present a structure preserving time discretization (of Crank-Nicolson type) 
of the scheme and suggest a solving procedure. Moreover, we express the scheme explicitly in terms of velocity and fluid depth. 
In Sect.~\ref{sec_numerical_analysis}, we perform numerical test and evaluate if the scheme 
is capable of conserving stationary solutions and whether it presents well the nonlinear dynamics, frequency 
relations, and the conservation laws such as mass and total energy.

We conclude this introduction by quickly reviewing the approach of  \cite{PaMuToKaMaDe2011}  based on 
a discretization of the group of volume preserving diffeomorphisms.

\bigskip

\paragraph{Variational discretization of incompressible fluids.} 
Variational discretizations in mechanics always start with a proper understanding of the continuous 
Lagrangian description of the mechanical system, namely, the identification of its configuration manifold 
$Q$ and of its Lagrangian, defined on the tangent bundle of $Q$, to which the 
Euler-Lagrange equations of motion are associated. In the case of the motion of an ideal fluid 
on a manifold $M$, following \cite{Arnold1966}, the configuration space is the infinite dimensional 
Lie group $ \operatorname{Diff}_{\rm vol}(M)$ of volume preserving diffeomorphisms of $M$.
The Lagrangian is defined on its tangent bundle and is given by the $L ^2 $ kinetic energy of 
the Lagrangian velocity of the fluid. The motion of the fluid in the material description thus 
formally corresponds to $L ^2 $ geodesics on $ \operatorname{Diff}_{\rm vol}(M)$. An essential 
property of the configuration manifold $ \operatorname{Diff}_{\rm vol}(M)$ is its group structure, 
which allows one to understand the Eulerian description of fluid dynamics as a ``symmetry reduced" version 
of the material description, associated to the relabelling symmetry of the Lagrangian. 
One of the main features of the approach undertaken in \cite{PaMuToKaMaDe2011} is that it allows one to preserve 
this symmetry at the spatially discretized level.

Let us assume that the fluid domain $M$ is 
discretized as a mesh $\mathbb{M}$ of $N$ cells denoted $C_i$, $i=1,...,N$. The mesh is not assumed to be 
regular. We define the $N \times N$ diagonal matrix $ \Omega $ with diagonal elements 
$ \Omega _{ii}= \operatorname{Vol}( C_i )$, the volume of cell $C_i $. It is shown in \cite{PaMuToKaMaDe2011} 
that an appropriate discrete version of the group $ \operatorname{Diff}_{\rm vol}(M)$ is the matrix group
\begin{equation}\label{DD}
\mathsf{D}_{\rm vol}(\mathbb{M})=\left\{q\in \operatorname{GL}(N)^+\mid q\cdot\mathbf{1}=\mathbf{1}\quad\text{and}\quad q^\mathsf{T}\Omega q=\Omega\right\},
\end{equation}
where $\operatorname{GL}(N)^+$ is the group of invertible $N \times N$ matrices with positive determinant and 
$\mathbf{1}$ denotes the column $(1,...,1)^\mathsf{T}$ so that the first condition reads $\sum_{j=1}^Nq_{ij}=1$, for all $i=1,...,N$. 
The main idea behind this definition is the following, see \cite{PaMuToKaMaDe2011}. Consider the linear action 
of the group $\operatorname{Diff}_{\rm vol}(M)$ on the space $\mathcal{F}(M)$ of functions on $M$ given by composition, i.e.,
\begin{equation}\label{action_diffeo}
f\in\mathcal{F}(M)\mapsto f\circ\varphi^{-1}\in\mathcal{F}(M),\quad \varphi\in\operatorname{Diff}_{\rm vol}(M).
\end{equation}
This linear map has the following two properties: it preserves the $L^2$ inner product of functions and preserves 
constant functions. The discrete diffeomorphism group \eqref{DD} is obtained by imposing that its linear action on 
discrete functions satisfies these two properties. If one chooses a discrete function to be represented by a vector 
$F\in\mathbb{R}^N$, whose value $F_i$ on cell $i$ is regarded as the cell average of the continuous function on 
cell $i$, then \textcolor{black}{a discrete approximation of the $L^2$ inner product of functions} is given by
\begin{equation}\label{pairing_0} 
\left\langle F, G\right\rangle_0=F^\mathsf{T}\Omega G=\sum_{i=1}^NF_i\Omega_{ii} G_i.
\end{equation} 
With this choice, we get the conditions $q\cdot\mathbf{1}=\mathbf{1}$ and $q^\mathsf{T}\Omega q=\Omega$ in \eqref{DD}. 

The spatial discretization of the incompressible Euler equations is then obtained by applying a variational 
principle on the discrete diffeomorphism group $\mathsf{D}_{\rm vol}(\mathbb{M})$ for an appropriate spatially 
discretized right invariant Lagrangian $L=L(q, \dot q)$ and with respect to appropriate nonholonomic constraints. 
This approach directly follows from a variational discretization of the geometric description of the Euler equations 
given in \cite{Arnold1966} that we briefly mentioned above.
Being associated to a right invariant Lagrangian, the variational principle can be equivalently rewritten on the 
Lie algebra of the matrix group $\mathsf{D} _{\rm vol}(\mathbb{M})$, given by the space of $\Omega$-antisymmetric, 
row-null $ N \times N$ matrices
\[
\mathfrak{d}_{\rm vol}(\mathbb{M})=\{A\in\mathfrak{gl}(N)\mid A\cdot \mathbf{1}=0\quad\text{and}\quad A^\mathsf{T}\Omega+\Omega A=0\}.
\]
As shown in \cite{PaMuToKaMaDe2011}, a matrix $A \in \mathfrak{d}_{\rm vol}(\mathbb{M})$ represents the discretization 
of a divergence free vector field $ \mathbf{u} $ through the identification of a 
matrix element with a weighted flux, i.e., 
\[
A _{ij} \simeq - \frac{1}{2 \Omega _{ii}}\int_{ D_{ij} }( \mathbf{u} \cdot \textcolor{black}{\mathbf{n}_{ij}}) \,{\rm d}S,
\]
where $D _{ij} $ is the hyperface common to cells $C_i$ and $C_j$ and $ \mathbf{n} _{ij} $ is the normal vector field 
on $ D _{ij} $ pointing from $C_i$ to $C_j$. This representation shows that only matrix elements associated to neighboring 
cells can be non-zero, which is understood as a nonholonomic constraint $ \mathcal{S} \subset \mathfrak{d}_{vol}(\mathbb{M})$ 
imposed on the Lie algebra elements and appropriately used in the variational principle.

\medskip

For later use, we recall that in the context of the discrete diffeomorphism group approach, a discrete zero-form 
(i.e., a discrete function) on $ \mathbb{M}  $ is a vector $F \in \mathbb{R}  ^N $. The components of such a vector 
are regarded as the cell averages of a continuous scalar field $f \in C^0(M)$, i.e., 
$F_i= \frac{1}{ \Omega _{ii}}\int_{ C_i} f ( \mathbf{x} )\, {\rm d}\mathbf{x} $. The space of discrete zero-forms is denoted 
$ \Omega _d ^0 ( \mathbb{M}  )$. 
A discrete one-form on $ \mathbb{M}  $ is a skew-symmetric matrix $ K \in \mathfrak{so}(N)$. The space of discrete 
one-form is denoted $ \Omega _d ^1 ( \mathbb{M}  )$.

\medskip

The discrete diffeomorphism group approach was developed towards applications to rotating stratified fluids 
in \cite{DeGaGBZe2014}. In \cite{BaGB2016} appropriate discrete diffeomorphism groups were defined to develop 
variational discretization of the equations of anelastic and pseudo-incompressible fluids.

\section{Variational Lie group discretization of compressible fluids}\label{Section_2} 

In this section, we shall appropriately extend the structure preserving spatial discretization of \cite{PaMuToKaMaDe2011} 
to the compressible case. As we shall see, the treatment of compressible fluids requires the inclusion of an additional 
nonholonomic constraint.

\bigskip

\paragraph{Group of discrete diffeomorphisms.} Given a mesh $ \mathbb{M}  $ on the fluid domain $M$, an evident candidate 
for a discretization of the group of all (i.e., not necessarily volume preserving) diffeomorphisms is obtained by removing 
the volume preserving condition $q^\mathsf{T} \Omega q= \Omega $ in \eqref{DD}, thereby obtaining the matrix Lie group 
\begin{equation}\label{DD_all}
\mathsf{D}(\mathbb{M})=\left\{q\in \operatorname{GL}(N)^+\mid q\cdot\mathbf{1}=\mathbf{1}\right\}
\end{equation}
of dimension $ N ^2 -N$.
Its Lie algebra is the space of row-null $N \times N$ matrices
\begin{equation}\label{Lie_algebra} 
\mathfrak{d}(\mathbb{M})=\{A\in\mathfrak{gl}(N)\mid A\cdot \mathbf{1}=0\}.
\end{equation} 
The following lemma characterizes the dual space to $\mathfrak{d}(\mathbb{M})$ relative to the duality pairing on 
$ \mathfrak{gl}(N)$ given by
\begin{equation}\label{duality_pairing}
\left\langle L, A \right\rangle = \operatorname{Tr}(L^\mathsf{T} \Omega A).
\end{equation}

\begin{lemma}\label{lemma_1}  The dual space to $\mathfrak{d}(\mathbb{M})$ with respect to the pairing \eqref{duality_pairing} 
can be identified with the space of $N \times N$ matrices with zero diagonal:
\begin{equation}\label{dual_space} 
\mathfrak{d}(\mathbb{M}) ^\ast = \{ L  \in \mathfrak{gl}(N)\mid L_{ii}=0,\;\; \text{for all $i$} \}.
\end{equation}
In particular, given $L \in \mathfrak{gl}(N)$, we have $ \left\langle L,A \right\rangle =0$, for all $A \in \mathfrak{d}( \mathbb{M}  )$ 
if and only if $\mathbf{Q} (L)=0$, for the projector
\[
\mathbf{Q} : \mathfrak{gl}(N) \rightarrow \mathfrak{d}(\mathbb{M}) ^\ast , \quad \mathbf{Q} (L):= L- \widehat{L},
\]
with $\widehat{L}_{ij}:= L_{ii}$. 
\end{lemma}
\begin{proof} The annihilator of $\mathfrak{d}(\mathbb{M})$ in $ \mathfrak{gl}(N)$, relative to the pairing \eqref{duality_pairing} 
is $\mathfrak{d}(\mathbb{M})^ \circ=\{L \in  \mathfrak{gl}(N)\mid L_{ij}= k_i,\;\; \text{for all $i,j$}\}$. The dual space is 
therefore identified with the quotient space $\mathfrak{gl}(N)/\mathfrak{d}(\mathbb{M})^ \circ$, which is clearly isomorphic 
to the space \eqref{dual_space}. 
\end{proof}
\medskip

From the preceding Lemma, the coadjoint operator 
$ \operatorname{ad}^*_A: \mathfrak{d} ( \mathbb{M}  ) ^\ast  \rightarrow \mathfrak{d} ( \mathbb{M}  ) ^\ast $ defined by 
$ \left\langle \operatorname{ad}^*_AL, B \right\rangle := \left\langle L, [A,B] \right\rangle $, for all 
$A,B \in \mathfrak{d}( \mathbb{M}  )$ and $L \in \mathfrak{d} ( \mathbb{M}  ) ^\ast $ is given by
\begin{equation}\label{ad_star} 
\operatorname{ad}^*_AL= \mathbf{Q} \left(  \Omega ^{-1} [A^\mathsf{T},\Omega L] \right),
\end{equation} 
where $[\,,]$ is the commutator of matrices, i.e., $[A,B]= AB- BA$. 

As we shall see below, an element in $\mathfrak{d}(\mathbb{M}) ^\ast$ does not represent necessarily a discrete momentum, 
since $A$ does not necessarily represent a discrete velocity.

\medskip

\textcolor{black}{We shall denote by $\{ \mathbb{M}  _h\}_{h>0}$ a family of meshes on the fluid domain $M$, indexed by $h=\max\{ h_{C_i} | C_i\in\mathbb{M}_h\}$, where $h_{C_i}$ is the diameter of cell $C_i$.}
Exactly as in the incompressible case considered in \cite{PaMuToKaMaDe2011}, given a family $\{ \mathbb{M}  _h\}_{h>0}$ of meshes on $M$, we say that a family $\{q_h\}_{h>0}$ of matrices $q_h \in \mathsf{D}(\mathbb{M}_h)$ approximates a 
diffeomorphism $ \varphi \in\operatorname{Diff}(M)$ if
\[
\|S _{ \mathbb{M}  _h} \left( q_h P_{ \mathbb{M}  _h} (f) \right)  - f \circ \varphi ^{-1} \|_{\textcolor{black}{L^\infty(M)}}\rightarrow 0,\quad \text{for all $f \in C^0( M)$}, 
\]
as $h \rightarrow 0$, where
\[
\left( P_{ \mathbb{M}  _h} (f)\right) _i:= \frac{1}{\Omega_{ii}}\int_{C _i }f( \mathbf{x} )\,{\rm d} \mathbf{x}  
\quad\text{and}\quad \left( S_{ \mathbb{M}  _h}(f _h ) \right) ( \mathbf{x} ):= (f_h)_i, \quad \text{if $ \mathbf{x}  \in C_i$},
\]
see Def. 1 of \cite{PaMuToKaMaDe2011}.
%


Consider a time dependent diffeomorphism $ \varphi (t) \in \operatorname{Diff}(M)$, fix a function $f_0$ on $M$ and 
define the time dependent function $f(t):= f_0 \circ  \varphi (t) ^{-1} $. The time derivative of $f$ is 
$\dot f(t)=- \mathbf{d} f(t) \cdot \mathbf{u}(t)$, the derivative of $f(t)$ in the direction $ \mathbf{u} (t)$. 
Suppose that a family $q_h(t)$ of discrete flows approximates $ \varphi (t) \in \operatorname{Diff}(M)$. 
From the definition of the discrete diffeomorphism group, the discrete version of $f(t) = f_0 \circ  \varphi (t)^{-1} $ 
is given by $ F_h(t) = q_h(t) F_h^0 $, where $F_h ^0 $ is a discrete function on $ \mathbb{M}  _h $. 
The time derivative of $F_h$ reads $\dot F_h(t)= A_h(t)F_h(t)$, where $ A_h(t):= \dot q_h(t) q_h(t) ^{-1} $.
Left multiplication by the matrix $A_h(t)$ thus corresponds to (minus) the discrete derivative along the discrete vector field $A_h(t)$.
This motivates the following definition, see Def. 3. of \cite{PaMuToKaMaDe2011}. Given a family $\{ \mathbb{M}  _h\}_{h>0}$ of meshes on the fluid domain, we say that a family $\{A_h\}_{h>0}$ of matrices $A_h \in \mathfrak{d}(\mathbb{M})$ approximates a vector field $ \mathbf{u} $ on $M$ if 
\begin{equation}\label{def_approx_velocity}
\|S _{ \mathbb{M}  _h} \left( A_h P_{ \mathbb{M}  _h} (f) \right)- (- \mathbf{d} f \cdot \mathbf{u})\|_{\textcolor{black}{L^\infty(M)}}\rightarrow 0 ,\quad 
\text{for all $f \in C^\infty( M)$}.
\end{equation}

The element $A _{ij} $ of the matrix $A(t)=\dot q(t) q(t) ^{-1} $ describes the infinitesimal exchange of fluid 
particles between cells $C_i$ and $C_j$. We thus assume the same nonholonomic constraint as in the incompressible 
case, namely that $A_{ij}$ is non-zero only if cells $C_i$ and $C_j$ share a common boundary. This leads to the constraint
\begin{equation}\label{constraint_1}
\mathcal{S} = \left\{ A \in \mathfrak{d} ( \mathbb{M}  )\mid A _{ij} =0, \;\; \text{for all $j \notin N(i)$} \right \}, 
\end{equation}
where $N(i)$ denotes the set of all indices (including $i$) of cells sharing a hyperface with cell $C_i$.

\color{black}

We shall now use \eqref{def_approx_velocity} to show explicitly how the elements of the matrix $ A \in \mathcal{S} $ approximate the continuous vector field $\mathbf{u}$ as $h\rightarrow 0$.
To do this, we shall assume some standard conditions on the family of meshes $\{\mathbb{M}_h\}_{h>0}$, see, e.g., \cite{ErGu2004}.
Recall that the family $\{\mathbb{M}_h\}_{h >0}$ is \textit{shape-regular} if there exists a constant $\sigma$ independent of $h$ such that
\begin{equation}\label{S_R}
\max_{C_i\in\mathbb{M}_h} \frac{h_{C_i}}{\rho_{C_i}}\leq \sigma,\quad\text{for every $h$},
\end{equation}
where $\rho_{C_i}$ is the diameter of the largest ball that can be inscribed in $C_i$. The family $\{\mathbb{M}_h\}_{h >0}$ is \textit{quasi-uniform} if it is shape-regular and there exists a constant $\gamma$ independent of $h$ such that
\begin{equation}\label{Q_U}
\max_{C_i\in\mathbb{M}_h} \frac{h}{h_{C_i}}\leq \gamma,\quad\text{for every $h$}.
\end{equation}
We shall also consider the following condition on the family of meshes 
\begin{equation}\label{additional_hypothesis}
\max_{C_i,C_j\in\mathbb{M}_h,\;i\in N(j)}\left|\frac{\mathbf{x}_i+\mathbf{x}_j}{2}- \mathbf{x}_{ij} \right|\leq \lambda h^{1+\alpha},
\end{equation}
for some $\alpha\geq 1$, where $\lambda$ is independent of $h$, $\mathbf{x}_k$ is the barycenter of cell $C_k$, and $\mathbf{x}_{ij}$ is the barycenter of hyperface $D_{ij}$.
In two dimensions, \eqref{additional_hypothesis} is equivalent to the following condition: Every pair of adjacent triangles $C_i\cup C_j$ forms an $O(h^{1+\alpha})$ approximate parallelogram.
That is, the lengths of any two opposite edges of $C_i\cup C_j$ differ by $O(h^{1+\alpha})$. This assumption is sometimes used in the finite element literature; see, for instance \cite{BaXu2003}, \cite{HuXu2008}, \cite{Ca2015}.

\begin{lemma}\label{lemma_matixA} Consider a family $\{\mathbb{M}_h\}_{h >0}$ of meshes on $M$. Assume that this family is quasi-uniform and satisfies \eqref{additional_hypothesis}. Given $\mathbf{u}\in W^{2,\infty}(M)$, we define for each $h>0$ the matrix $A_h^\mathbf{u}\in\mathcal{S}$ by
\begin{equation}\label{Aii}
(A^\mathbf{u}_h)_{ii}= \frac{1}{2\Omega_{ii}}\int_{C_i}\operatorname{div}\mathbf{u} \,{\rm d}\mathbf{x},\quad\text{for every $i=1,...,N_h$}
\end{equation}
and
\begin{equation}\label{Aij}
(A_h^\mathbf{u})_{ij}=-\frac{1}{2\Omega_{ii}}\int_{D_{ij}}(\mathbf{u}\cdot\mathbf{n}_{ij} ){\rm d}S,\quad\text{for every $i,j=1,...,N_h$, $j\in N(i)$},
\end{equation}
where $N_h$ is the number of cells in $\mathbb{M}_h$ and $\mathbf{n}_{ij}$ is the normal vector field on $D_{ij}$ pointing from $C_i$ to $C_j$.

Then
\begin{equation}\label{result_approx_velocity}
\left\|S _{ \mathbb{M}  _h} \left( A_h^\mathbf{u} P_{ \mathbb{M}  _h} (f) \right)- (- \mathbf{d} f \cdot \mathbf{u})\right\|_{{L^\infty(M)}}\rightarrow 0
\end{equation}
holds for every $f\in C^2(M)$.
\end{lemma}
\begin{proof} By using \eqref{Aii}, \eqref{Aij}, and Gauss' Theorem, we have, for $\mathbf{x}\in C_i$
\begin{align*}
-(A^\mathbf{u}_hF)_i - \mathbf{d}f\cdot \mathbf{u}(\mathbf{x})&=  \frac{1}{\Omega_{ii}} \sum_{j\neq i} \int_{D_{ij}} \left(\frac{F_j+F_i}{2}- f\right)(\mathbf{u}\cdot\mathbf{n} _{ij}) {\rm d}S\\
&\qquad -\frac{1}{\Omega_{ii}}\int_{C_i}(F_i-f)\operatorname{div}\mathbf{u}\,{\rm  d}\mathbf{x}\\
&\qquad + \frac{1}{\Omega_{ii}} \int_{C_i}(\mathbf{d}f\cdot\mathbf{u}) {\rm d}\mathbf{x}- \mathbf{d}f\cdot \mathbf{u}(\mathbf{x}) \\
&=: e_1+ e_2 + e_3,\quad\text{for all $\mathbf{x}\in C_i$},
\end{align*}
where $F_i:=P_{ \mathbb{M}  _h} (f)= \frac{1}{\Omega_{ii}}\int_{C_i}f(\mathbf{x})\,{\rm d}\mathbf{x}$ and where we drop the index $h$ on $F$ for simplicity.

By the Poincar\'e-Wirtinger inequality, we have
\begin{equation}\label{e_2}
\|e_2\|_{L^\infty(C_i)}\leq \| f-F_i\|_{L^\infty(C_i)}\|\operatorname{div}\mathbf{u}\|_{L^\infty(C_i)}\leq C h_{C_i} |f|_{W^{1,\infty}(C_i)}\|\operatorname{div}\mathbf{u}\|_{L^\infty(C_i)}
\end{equation}
and
\begin{equation}\label{e_3}
\|e_3\|_{L^\infty(C_i)}\leq C h_{C_i} |\mathbf{d}f\cdot \mathbf{u}|_{W^{1,\infty}(C_i)},
\end{equation}
where $|\cdot|_{W^{k,\infty}(C_i)}$ denotes the $W^{k,\infty}$ semi-norm on $C_i$.
We write $e_1$ as
\[
e_1=  \frac{1}{\Omega_{ii}} \sum_{j\neq i} \int_{D_{ij}}\! \!\left(\frac{F_j+F_i}{2}- f\right)(\mathbf{u}-\bar{\mathbf{u}})\cdot\mathbf{n} _{ij} {\rm d}S +  \frac{1}{\Omega_{ii}} \sum_{j\neq i} \int_{D_{ij}} \!\!\left(\frac{F_j+F_i}{2}- f\right)(\bar{\mathbf{u}}\cdot\mathbf{n} _{ij}) {\rm d}S,
\]
where $\bar{\mathbf{u}}=\frac{1}{|D_{ij}|}\int _{D_{ij}}( \mathbf{u} \cdot\mathbf{n} _{ij})\,{\rm d}S$. By the Poincar\'e-Wirtinger inequality, the first term is bounded by
\[
\sum_{j\neq i}C\frac{|D_{ij}|}{\Omega_{ii}}\left(h_{C_i} |f| _{W^{1,\infty}(C_i)}+ h_{C_j} |f| _{W^{1,\infty}(C_j)}\right) h_{D_{ij}}|\mathbf{u}\cdot\mathbf{n}_{ij}|_{W^{1,\infty}(D_{ij})}.
\]
By the shape-regularity assumption, we have $\frac{|D_{ij}|}{\Omega_{ii}}\leq C \frac{1}{h_{C_i}}$ and since $h_{D_{ij}}\leq h_{C_i}$ this term 
is bounded by
\[
\sum_{j\neq i}C h\left( |f| _{W^{1,\infty}(C_i)}+  |f| _{W^{1,\infty}(C_j)}\right)|\mathbf{u}|_{W^{1,\infty}(C_i)}.
\]
The second term is bounded by
\begin{equation}\label{intermediate_term}
\sum_{j\neq i}|\bar{\mathbf{u}}\cdot\mathbf{n}_{ij}|\frac{ |D_{ij}|}{\Omega_{ii}}  \left|\frac{F_j+F_i}{2}-\frac{1}{|D_{ij}|} \int_{D_{ij}}f\, {\rm d}S\right|.
\end{equation}
The last factor in \eqref{intermediate_term} is then written as
\begin{equation}\label{four_terms}
\begin{aligned}
&\left[\frac{F_j+F_i}{2} - \frac{f(\mathbf{x}_j)+f(\mathbf{x}_i)}{2} \right] +  \left[\frac{f(\mathbf{x}_j)+f(\mathbf{x}_i)}{2} -  f\left( \frac{\mathbf{x}_j+\mathbf{x}_i}{2}\right)\right]\\
&\qquad\qquad + \left[ f\left( \frac{\mathbf{x}_j+\mathbf{x}_i}{2}\right)-f(\mathbf{x}_{ij})\right] + \left[ f(\mathbf{x}_{ij}) -\frac{1}{|D_{ij}|} \int_{D_{ij}}f \,{\rm d}S\right] ,
\end{aligned}
\end{equation}
where $\mathbf{x}_i$ is the barycenter of cell $C_i$ and $\mathbf{x}_{ij}$ is the barycenter of hyperface $D_{ij}$. Each of these terms is then estimated via Taylor expansion,  
giving the bound $C( |f|_{W^{2,\infty}(D)}h^2+ |f|_{W^{1,\infty}(D)}h^{1+\alpha})$, where the hypothesis \eqref{additional_hypothesis} is used in the treatment of the third term in \eqref{four_terms}. When used in \eqref{intermediate_term} this estimation has to be combined with the shape-regular assumption \eqref{S_R} and the quasi-uniform assumption \eqref{Q_U} to finally give the bound $C\|\mathbf{u}\|_{L^\infty(D)}( |f|_{W^{2,\infty}(D)}h+ |f|_{W^{1,\infty}(D)}h^{\alpha})$ for \eqref{intermediate_term} and hence
\begin{equation}\label{e_1}
\begin{aligned}
\|e_1\|_{L^\infty(C_i)}\leq &C\left(  h |\mathbf{u}|_{W^{1,\infty}(D)} |f|_{W^{1,\infty}(D)} + h  \|\mathbf{u}\|_{L^\infty(D)} |f|_{W^{2,\infty}(D)}\right. \\
&\quad\quad\quad \left.+ h^\alpha \|\mathbf{u}\|_{L^\infty(D)} |f|_{W^{1,\infty}(D)}\right).
\end{aligned}
\end{equation}
The combination of all these estimations gives
\[
\|e_1+e_2+e_3\|_{L^\infty(D)}\leq C\left( h\|\mathbf{u}\|_{W^{1,\infty}(D)}\|f\|_{W^{2,\infty}(D)}+ h^\alpha \|\mathbf{u}\|_{L^\infty(D)}\|f\|_{W^{1,\infty}(D)}\right)
\]
which proves the result.
\end{proof}

From this Lemma, it follows that if $A_h\in \mathfrak{d}(\mathbb{M}_h)$ is an approximation of the vector field $\mathbf{u}$, i.e., it satisfies \eqref{def_approx_velocity}, then $\|A_h-A^\mathbf{u}_h\|\rightarrow 0$, as $h\rightarrow 0$. Formulas \eqref{Aii}--\eqref{Aij} thus give the relation between the Lie algebra element $A$ and the vector field $\mathbf{u}$ it approximates. In particular $(A_h)_{ij}\rightarrow 0$ for $j\notin N(i)$ as $h\rightarrow 0$, i.e., $A_h$ satisfies the sparsity constraint $\mathcal{S}$ in an approximate sense.

\color{black}

Note that the expressions \eqref{Aii}--\eqref{Aij} are consistent with the condition $ A \cdot \mathbf{1} =0$ 
in \eqref{Lie_algebra}. From these expressions, we also deduce that the matrices $A \in \mathfrak{d}(\mathbb{M})$ 
have to satisfy, in addition to the constraint $A\in \mathcal{S} $ imposed earlier, the constraint $ \Omega _{ii} A _{ij} =- \Omega _{jj} A _{ji}$, 
for all $j\neq i$, i.e., $A^\mathsf{T}\Omega + \Omega A$ is a diagonal matrix. We include this in an additional 
nonholonomic constraint given by
\begin{equation}\label{constraint_2} 
\mathcal{R} = \left \{ A \in \mathfrak{d} ( \mathbb{M}  )\mid A^\mathsf{T} \Omega + \Omega A\;\; \text{is diagonal} \right   \}.
\end{equation} 
This constraint is equivalently described by saying that $A$ decomposes as $A=A^a+A^d$, where $A^a \in \mathfrak{d}_{\rm vol}( \mathbb{M}  )$ and $A^d$ is diagonal. For $ A \in \mathcal{R} $, the diagonal part is found from the equality $ A^\mathsf{T} \Omega + \Omega A=2 \Omega  A ^d$.
From the condition $A \cdot \mathbf{1} =0$, we get $A_{ii}^d=- \sum_{j} A _{ij} ^a $.

\bigskip

\paragraph{\textcolor{black}{Semidiscrete} variational equations for compressible fluids.} 
The derivation of the spatially discretized equations for compressible fluids is based on the following proposition. 

\begin{proposition}[Discrete momenta]\label{DM} The dual space to the constraint space $ \mathcal{R} \subset \mathfrak{d} ( \mathbb{M}  )$ 
can be identified with the space of discrete one-forms, i.e., 
\[
\mathfrak{d} ( \mathbb{M}  ) ^\ast / \mathcal{R} ^\circ= \Omega_d^1 ( \mathbb{M}  ).
\]
In particular, given  $ L \in \mathfrak{gl}(N) $, we have $\left\langle L, A \right\rangle =0$, for all $A \in \mathcal{R} $ 
if and only if $ \mathbf{P} (L)=0$, where
\begin{equation}\label{projection} 
\mathbf{P} : \mathfrak{gl}(N) \rightarrow  \Omega_d^1 ( \mathbb{M}  ), \quad  \mathbf{P} (L):=( L- \widehat{L} )^{(A)},
\end{equation} 
with $\widehat{L}_{ij}:= L_{ii}$ and $L^{(A)}:= \frac{1}{2} (L- L^\mathsf{T})$.
\end{proposition} 
\begin{proof} Recall from Sect.~\ref{intro} that $ \Omega _d ^1 ( \mathbb{M}  )$ is identified with the space 
of $N \times N$ skew-symmetric matrices. From Lemma \ref{lemma_1}, we have 
$\mathfrak{d}(\mathbb{M}) ^\ast = \{ L \in \mathfrak{gl}(N)\mid L_{ii}=0,\;\; \text{for all $i$} \}$.
The annihilator of $ \mathcal{R} $ in  $\mathfrak{d} ( \mathbb{M}  ) ^\ast$ with respect to the 
pairing \eqref{duality_pairing} is given by $\mathcal{R} ^\circ =\{L \in  \mathfrak{d} ( \mathbb{M}  ) ^\ast\mid L=L^\mathsf{T}\}$. 
Indeed, one checks that for $L \in \mathcal{R} ^\circ $, we have $ \left\langle L, A \right\rangle =0$, for all $A \in \mathcal{R} $. 
The result follows since $ \operatorname{dim} \mathcal{R} ^ \circ=  \operatorname{dim} \mathfrak{d} ( \mathbb{M}  ) ^\ast- \operatorname{dim} \mathcal{R} $.

The second part of the proposition easily follows from the first part. 
\end{proof}


The right action of a discrete diffeomorphism $q$ on a discrete function $F \in\Omega _d^0( \mathbb{M}  )$ is given by matrix multiplication, i.e., $F \cdot q=  q ^{-1} F$. The space of discrete densities $ \operatorname{Den}_d( \mathbb{M}  )$ is defined as the dual space to  $\Omega _d^0( \mathbb{M}  )$ with respect to the pairing \eqref{pairing_0}. The right action on a density $D \in \operatorname{Den}_d( \mathbb{M}  )$ is defined by duality as $ \left\langle D \cdot q, F \right\rangle _0:=\left\langle D, F\cdot q^{-1}  \right\rangle _0 $, for all $ F \in\Omega ^0 _d(\mathbb{M}  )$. It is explicitly given by
\begin{equation}\label{action_density} 
D \in \operatorname{Den}_d( \mathbb{M}  ) \longmapsto D \cdot q=  \Omega ^{-1} q^{\mathsf{T}} \Omega D\in \operatorname{Den}_d( \mathbb{M}  ), \quad q \in \mathsf{D}( \mathbb{M}  ).
\end{equation} 
The associated infinitesimal action of a Lie algebra element $A \in \mathfrak{d} ( \mathbb{M}  )$ on a discrete density $D \in \operatorname{Den}_d( \mathbb{M}  )$ is found by taking the derivative of the action with respect to $q$ at the neutral element. We get
\begin{equation}\label{infinitesimal_action} 
D \in \operatorname{Den}_d( \mathbb{M}  ) \longmapsto D \cdot A:= \Omega ^{-1} A^{\mathsf{T}} \Omega D \in  \operatorname{Den}_d( \mathbb{M}  ), \quad A \in \mathfrak{d} ( \mathbb{M}  ).
\end{equation} 

\medskip

Let us assume that a spatially discretized Lagrangian $L_{D_0}: T \mathsf{D} ( \mathbb{M}  ) \rightarrow \mathbb{R}  $ is defined on the tangent bundle $T \mathsf{D} ( \mathbb{M}  )$ of the discrete diffeomorphism group. This Lagrangian depends on the initial density $D_0$ of the fluid in Lagrangian description, hence the notation $L_{D_0}$.
We assume that this Lagrangian preserves the symmetry of the continuous Lagrangian, namely, there exists a Lagrangian $\ell=\ell(A, D): \mathfrak{d} ( \mathbb{M}  ) \times \mathbb{R}  ^N \rightarrow \mathbb{R} $ in Eulerian coordinates such that for any given initial density $D_0$, we can write
\begin{equation}\label{Lagrange_to_Euler} 
L_{D_0}(q, \dot q)= \ell(A, D), \quad A=\dot q q ^{-1} , \;\; D= D _0\cdot  q^{-1} .
\end{equation}

The spatially discretized equations follow from the Lagrange-d'Alembert variational principle (see, e.g., \cite{Bloch2003}) applied to the 
Lagrangian $L_{D_0}$ and with respect to the nonholonomic constraint $ \mathcal{S} \cap \mathcal{R} \subset \mathfrak{d} ( \mathbb{M}  )$. 
This variational principle reads
\begin{equation}\label{VP_Lagr} 
\delta \int_0^T L_{D_0}( q, \dot q) {\rm d}t=0,
\end{equation} 
where $\dot q q ^{-1} \in \mathcal{S} \cap \mathcal{R} $ and with respect to variations $ \delta q$ such that $ \delta q q ^{-1} \in \mathcal{S} \cap \mathcal{R} $ 
and with $ \delta q(0)=\delta q(T)=0$. In the Eulerian description, using \eqref{Lagrange_to_Euler}, this variational principle can be rewritten as
\begin{equation}\label{VP_Euler} 
\delta \int_0^T \ell( A, D) {\rm d}t=0,
\end{equation} 
where $A \in \mathcal{S} \cap \mathcal{R} $ and for variations $ \delta A$ and $ \delta D$ such that
\begin{equation}\label{VP_constraint} 
\delta A=\partial_tB+[B,A] ,\ \  \delta D=- \Omega ^{-1} B^\mathsf{T} \Omega D \ \  \text{with} \ \  B\in\mathcal{S}\cap \mathcal{R} ,\; B(0)=B(T)=0.
\end{equation} 
We call the principle \eqref{VP_Euler}--\eqref{VP_constraint} an \textit{Euler-Poincar\'e-d'Alembert principle}. The expression of the variation $ \delta A$ is 
obtained from the definition $A=\dot q q ^{-1} $, where $B$ is defined as $B= \delta q q ^{-1} \in \mathcal{S} \cap \mathcal{R} $.
The expression of $ \delta D$ is obtained from the definition $D= D_0 \cdot q ^{-1} $ and from \eqref{action_density}. The passage from the 
Lagrange-d'Alembert principle \eqref{VP_Lagr} to the Euler-Poincar\'e-d'Alembert principle \eqref{VP_Euler}--\eqref{VP_constraint} is rigorously 
justified by employing the Euler-Poincar\'e reduction theory, see  \cite{HoMaRa1998}, extended to include the nonholonomic constraint $ \mathcal{S} \cap \mathcal{R} $.

\begin{theorem}[\textcolor{black}{Semidiscrete} variational equations]\label{theorem_1}  Consider a semidiscrete Lagrangian $\ell= \ell (A,D): \mathfrak{d} ( \mathbb{M}  ) \times \operatorname{Den}_d( \mathbb{M}  )  \rightarrow \mathbb{R}  $. Then, the curves $A(t), D(t)$ are critical for the variational principle \eqref{VP_Euler} if and only if it  satisfies the equations
\begin{equation}\label{discrete_EP_compressible} 
\mathbf{P} \left( \frac{d}{dt} \frac{\delta  \ell}{\delta  A}+  \Omega ^{-1} \left[ A^\mathsf{T}, \Omega  \frac{\delta  \ell}{\delta  A}\right] + D \frac{\delta \ell}{\delta D}^\mathsf{T} \right)_{ij}  =0, \quad \text{for all $i \in N(j)$},
\end{equation}
where $ \mathbf{P} : \mathfrak{g}(N) \rightarrow \Omega ^1 _d ( \mathbb{M}  )$ is the projection \eqref{projection}. \textcolor{black}{This is the semidiscrete balance of momenta for compressible fluids.} The \textcolor{black}{semidiscrete} continuity equation reads
\begin{equation}\label{continuity} 
\frac{d}{dt} D+\Omega ^{-1} A^\mathsf{T}\Omega D=0.
\end{equation} 
\end{theorem}

\begin{proof} Given $\ell: \mathfrak{d} ( \mathbb{M}  ) \times \operatorname{Den}_d( \mathbb{M}  ) \rightarrow \mathbb{R}$, the functional derivatives 
$ \frac{\delta \ell}{\delta A} \in \textcolor{black}{\mathfrak{d} ( \mathbb{M}  )^*}$ and $ \frac{\delta \ell}{\delta D} \in \Omega _d ^0 ( \mathbb{M}  )$   are defined with 
respect to the appropriate pairings as
\[
\left\langle \frac{\delta \ell}{\delta A}, \delta A \right\rangle = \left.\frac{d}{d\varepsilon}\right|_{\varepsilon=0} \ell( A+ \varepsilon \delta A, D), 
\quad \left\langle \frac{\delta \ell}{\delta D}, \delta D \right\rangle_0 = \left.\frac{d}{d\varepsilon}\right|_{\varepsilon=0} \ell( A, D+ \varepsilon \delta D),
\]
for all $ \delta A \in \mathfrak{d} ( \mathbb{M}  )$ and $ \delta D\in\mathbb{R}  ^N $.

Application of the variational principle \eqref{VP_Euler} yields
\[
\delta \int_0^T \ell( A, D) {\rm d}t= \int_0^T \left\langle \frac{\delta \ell}{\delta A},  \partial_tB+[B,A] 
\right\rangle  {\rm d}t+\int_0^T \left\langle \frac{\delta \ell}{\delta D},   - \Omega ^{-1} B^\mathsf{T} \Omega D \right\rangle_0  {\rm d}t.
\]
Isolating the matrix $B$, integrating by parts, and using $ B(0)=B(T)=0$, we get
\[
\int_0^T \left\langle \frac{d}{dt} \frac{\delta  \ell}{\delta  A}+  \Omega ^{-1} \left[ A^\mathsf{T},  
\Omega \frac{\delta  \ell}{\delta  A}\right] + D \frac{\delta \ell}{\delta D}^\mathsf{T} , B \right\rangle  {\rm d}t=0,
\]
for all $B \in \mathcal{S} \cap \mathcal{R} $. The result then follows from Proposition \ref{DM} and by noting 
that since $B \in \mathcal{S} $, the equations are verified only for neighboring cells, i.e., $ j \in N(i)$.

The \textcolor{black}{semidiscrete} continuity equation follows from the definitions
\[
D(t)= D_0 \cdot q(t)^{-1} =\Omega ^{-1} q(t)^{-\mathsf{T}} \Omega D_0 \quad\text{and}\quad A(t)= \dot q(t) q(t) ^{-1}. 
\]
\end{proof}

\begin{remark}[Coadjoint operator and discrete Lie derivative]{\rm We have seen that the coadjoint operator for the group 
$\mathsf{D}( \mathbb{M}  )$ with respect to the pairing \eqref{duality_pairing} has the expression
\[
\operatorname{ad}^*_AL= \mathbf{Q} \left(  \Omega ^{-1} [A^\mathsf{T},\Omega L] \right), \quad A \in\mathfrak{d}( \mathbb{M}  ),\quad L\in \mathfrak{d}( \mathbb{M}  ) ^\ast .
\]
The discrete advection term in \eqref{discrete_EP_compressible} is thus related to the coadjoint operator as follows
\[
\mathbf{P} \left(  \Omega ^{-1} \left[ A^\mathsf{T}, \Omega  L\right] \right) 
= \left(  \mathbf{Q} \left(  \Omega ^{-1} [A^\mathsf{T},\Omega L] \right) \right) ^{(A)}= \left( \operatorname{ad}^*_AL \right) ^{(A)},
\]
where we recall that $L^{(A)}:= \frac{1}{2} (L- L^\mathsf{T})$ denotes the skew-symmetric part of a matrix.
We shall compute explicitly this advection operator on simplicial grids below and obtain a discretization of the Lie 
derivative operator on one-form densities.
}
\end{remark}

\paragraph{Discrete Lagrangian for barotropic fluids.} The Lagrangian of a compressible rotating barotropic fluid on a manifold $M$ with Riemannian metric $g$ has the general form
\begin{equation}\label{Lagrangian_barotropic}
\ell( \mathbf{u} , \rho )= \int _M  \Big[ \frac{1}{2}  \rho  \,\mathbf{u}^\flat \cdot \mathbf{u}  + \rho\,\mathbf{R} ^\flat \cdot  \mathbf{u}  - \varepsilon ( \rho )\Big]   {\rm d}\mathbf{x},
\end{equation} 
with $ \mathbf{u} $ the fluid Eulerian velocity, $ \rho $ the mass density, and $ \varepsilon (\rho )$ the internal energy density. 
The vector field $ \mathbf{R} $ is 
the vector potential of the angular velocity of the Earth. The first and second term of the Lagrangian involve the flat operator $ \flat $ associated to the Riemannian metric. 
The intrinsic expression of the equations for the compressible fluid on Riemannian manifolds is
\begin{equation}\label{compressible_fluid_Riemannian} 
\partial _t \mathbf{u} +\mathbf{u} \cdot \nabla \mathbf{u} + 2( \mathbf{i} _ \mathbf{u}\boldsymbol{\Omega} )^\sharp = - \frac{1}{ \rho } \operatorname{grad} p, \quad \partial _t \rho 
+ \operatorname{div}(\rho \mathbf{u} )=0 ,
\end{equation} 
where $ \nabla $, $\operatorname{grad}$, and $ \operatorname{div}$ are, respectively, the covariant derivative, the gradient, and the divergence associated to the 
Riemannian metric $g$, $ \boldsymbol{\Omega}$ is the $2$-form defined by $2 \boldsymbol{\Omega}  := \mathbf{d} \mathbf{R} ^\flat $, 
and $ \mathbf{i} _ \mathbf{u} \boldsymbol{\Omega} =\boldsymbol{\Omega} (\mathbf{u} ,\_\,)$. The pressure is given by 
$ p= \frac{\partial \varepsilon }{\partial \rho }\rho -\varepsilon$. On $M= \mathbb{R}  ^3 $, we recover the classical Coriolis term 
$2( \mathbf{i} _ \mathbf{u}\boldsymbol{\Omega} )^\sharp= 2 \boldsymbol{\Omega} \times \mathbf{u} $ (see, e.g., \cite{Bauer2016}). If, in addition to $ \rho $, the 
internal energy depends on other non dynamical fields, like the bottom topography, then \eqref{compressible_fluid_Riemannian} has the more general expression
\begin{equation}\label{compressible_fluid_Riemannian_more_general} 
\partial _t \mathbf{u} +\mathbf{u} \cdot \nabla \mathbf{u} + 2( \mathbf{i} _ \mathbf{u}\boldsymbol{\Omega} )^\sharp 
= -  \operatorname{grad}  \frac{\partial \varepsilon }{\partial \rho } , \quad \partial _t \rho + \operatorname{div}(\rho \mathbf{u} )=0.
\end{equation} 

A main ingredient in the definition of a discretization of the Lagrangian \eqref{Lagrangian_barotropic} on $ \mathfrak{d}( \mathbb{M}  ) \times\operatorname{Den}_d( \mathbb{M}  )$ 
is a consistent discretization of the Riemannian metric or, equivalently, of the flat operator, on a given mesh $ \mathbb{M}  $. Since only the skew symmetric part of the discrete 
momenta counts (i.e., the discrete momenta are in $ \mathfrak{so}(N)$), we can still use the same flat operators as in \cite{PaMuToKaMaDe2011}. Given such a discrete flat operator 
$\flat $, the discrete Lagrangian associated to \eqref{Lagrangian_barotropic} is
\begin{equation}\label{discrete_Lagrangian} 
\ell(A, D) = \frac{1}{2} \sum_{i,j=1}^ND_i A ^\flat_{ij} A _{ij} \Omega _{ii}+ \sum_{i,j=1}^ND_i R ^\flat_{ij} A _{ij} \Omega _{ii} - \sum_{i=1}^N \epsilon (D_i) \Omega _{ii}.
\end{equation}

\section{Compressible rotating fluids on simplicial grids}\label{section_3}

In this section we shall use Theorem \ref{theorem_1} to derive a structure preserving variational discretization of 
2 dimensional rotating compressible fluids on irregular simplicial grids. We shall then consider the special case of 
the rotating shallow water equations.

\begin{figure}[t]
\centering
{\includegraphics[width=3.5in]{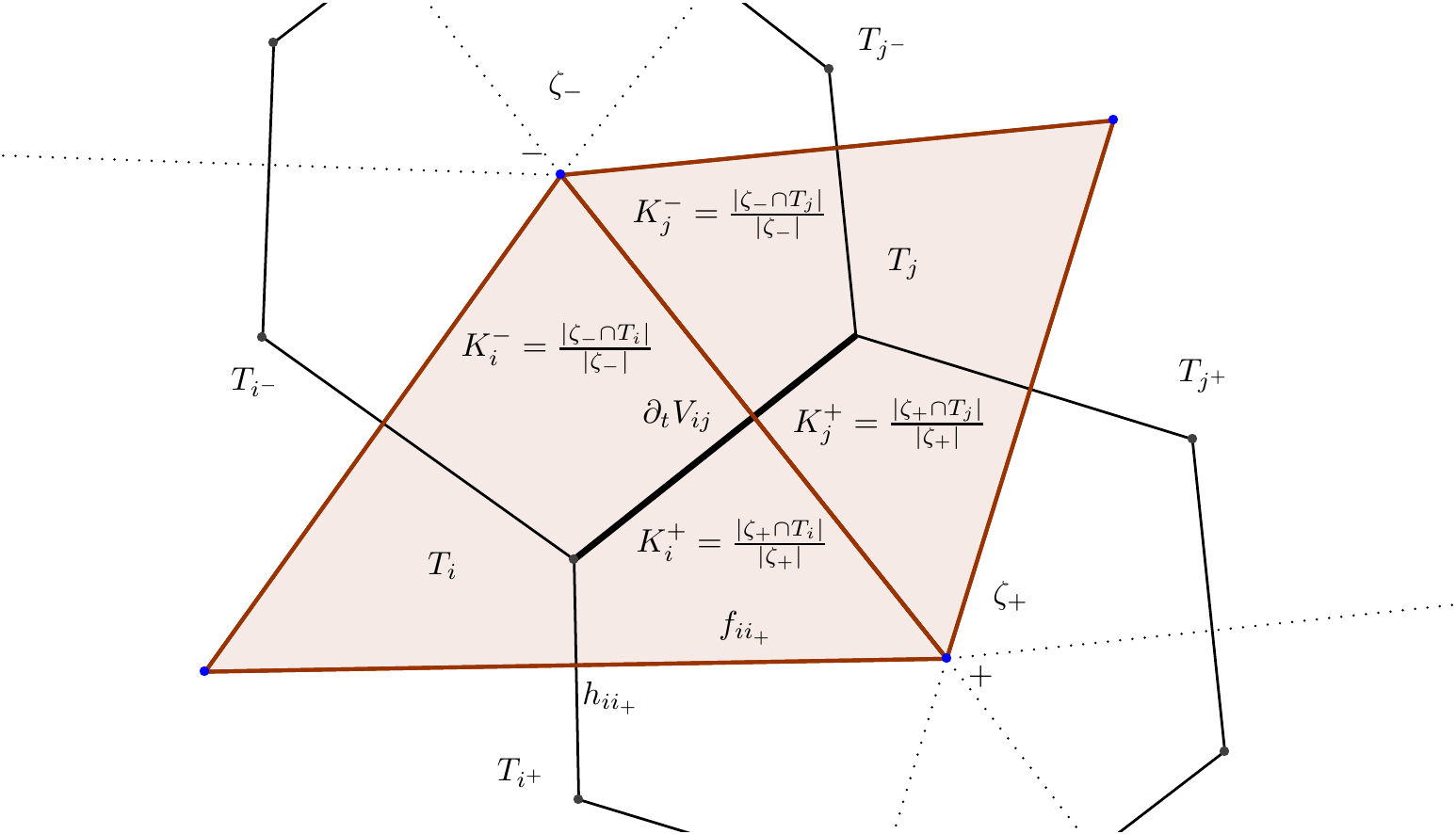}}
\caption{Notation and indexing conventions for the 2D simplicial mesh.}
\label{fig_Notation}
\end{figure}

\medskip

\paragraph{Simplicial grid.} We consider a 2D simplicial mesh on the fluid domain, as described in Fig.~\ref{fig_Notation}. We adopt the following notations:
\begin{small}
\begin{align*} 
f_{ij}:&=  \;\text{length of a primal edge, triangle edge located between triangle $i$ and triangle $j$};\\
h_{ij}:&=  \;\text{length of a dual edge that connect the circumcenters of triangle $i$ and triangle $j$};\\
\Omega_{ii}:&= \;\text{area of a primal simplex (triangle) $T_i$}.
\end{align*} 
\end{small}
The flat operator on a 2D simplicial mesh is defined by the following two conditions, see \cite{PaMuToKaMaDe2011},
\begin{equation}\label{def_flat}
\begin{aligned} 
&A ^\flat _{ij}= 2 \Omega_{ii}\frac{h_{ij} }{f _{ij} }A _{ij}, \quad \text{for $j \in N(i)$},\\
&A ^\flat _{ij}+ A ^\flat _{jk}+ A^\flat _{ki}= K_j ^{e} \left\langle\omega (A ^\flat ),\zeta_{e}\right\rangle, \quad \text{for $i,k \in N(j)$, $k \notin N(i)$},
\end{aligned}
\end{equation}  
where $e$ denotes the node common to triangles $T_i, T_j, T_k$ and $ \zeta _e $ denotes the dual cell to $e$. The vorticity $ \omega $ of a discrete 
one-form ${L} \in \Omega ^1 _d ( \mathbb{M}  )$ is defined by
\[
\left\langle\omega ({L}),\zeta_e\right\rangle:=  \sum_{ h _{mn}\in \partial \zeta _e } L _{mn}
\]
and the constant $K _i ^e $ is defined as
\begin{equation}\label{equ_K}
K^e_k:=\frac{|\zeta_{e}\cap T_k|}{|\zeta_{e}|},
\end{equation}
where $|\zeta_e\cap T_k|$ denotes the area of the intersection of $ \zeta _e$ and $T_k$.
\bigskip

\paragraph{Variational discretization of compressible fluids.} 
A main ingredient in the variational discretization \eqref{discrete_EP_compressible} is the expression of the 
discrete advection term \textcolor{black}{in the momentum equation}. We shall compute it separately in the next lemma, before giving the semidiscrete 
compressible fluid equations.

\begin{lemma}[Discrete Lie derivative on simplicial grids]\label{lemma}  On a 2D simplicial grid the discrete advection term for 
$A,B \in \mathfrak{d} ( \mathbb{M}  )$ and $D\in \operatorname{Den}_d(\mathbb{M}  )$ is given by
\begin{equation}\label{discrete_LD}
\begin{aligned} 
&\mathbf{P} \left(  \Omega ^{-1} \left[ A^\mathsf{T}, \Omega D B ^\flat \right] \right)_{ij}\\
&\vspace{0.2cm}= \omega(B^\flat ) _- \Big(K_i^-\overline{D}_{ji_-} A_{ii_-}
+ K_j^-\overline{D}_{ij_-}A_{jj_-}\Big)-\omega (B^\flat )_+ \Big(K_i^+\overline{D}_{ji_+} A_{ii_+}+ K_j^+\overline{D}_{ij_+}A_{jj_+}\Big)\\
& \vspace{0.4cm}\qquad + \overline{D} _{ij} \Big(\sum_{ k \in N(i)} A_{ik} B_{ik} ^\flat - \sum_{k \in N(j)} A_{jk} B_{jk} ^\flat \Big)  
- (\overline{ D \cdot A}) _{ij} B ^\flat  _{ij}\\
&= :\left( \mathcal{L}^d _A (D B ^\flat ) \right) _{ij} .
\end{aligned}
\end{equation}
In this formula, the indices $i,j$ correspond to two neighboring cells $T_i$ and $T_j$, the indices $i_\pm$ and $j_\pm$ refer to 
cells as indicated on Fig.~\ref{fig_Notation}. The discrete vorticity associated to $B ^\flat $ at the nodes $\pm$ is defined by
\[
\omega (B^\flat )_\pm= \sum_{ h _{mn}\in \partial \zeta _\pm } B ^\flat  _{mn}
\]
and $ \overline{D} _{ij} := \frac{1}{2} (D_i+ D_j)$. In the last term of the third line of \eqref{discrete_LD}, $D\!\cdot\! A$ denotes the infinitesimal 
action of $A \in \mathfrak{d} ( \mathbb{M}  )$ on $D \in \operatorname{Den}_d(\mathbb{M}  )$, see \eqref{infinitesimal_action}.

In the last line we introduce the notation $\mathcal{L}^d _A (D B ^\flat )$ to indicate that this expression is a discrete version of the Lie derivative, see Remark \ref{discrete_L}.
\end{lemma}
\begin{proof} Using the second equation in \eqref{def_flat} and $A_{ki}\Omega_{kk}=- \Omega_{ii}A_{ik}$, we compute the matrix element $i\neq j$ for $i \in N(j)$,
\begin{align*}
\Big( \Omega^{-1} &[ A^\mathsf{T},\Omega DB ^\flat ] \Big) _{ij} = \\ 
& \omega (B)_- \left( K _i ^- D_{i_-}A_{ii_-}+ K _j ^- D_iA_{jj_-} \right) 
- \omega (B)_+  \left( K _i ^+ D_{i_+}A_{ii_+}+ K _j ^+ D_iA_{jj_+} \right)\\
& \quad - A_{ii_-} D_{i_-}(B ^\flat_{i_-i}+B ^\flat_{ij} )-A_{ii_+} D_{i_+}(B ^\flat_{i_+i}+B ^\flat _{ij} )- A_{jj_-} D_{i}(B ^\flat_{jj_-}+B ^\flat _{ij} )\\
& \quad - A_{jj_+} D_i(B ^\flat_{jj_+}+B ^\flat _{ij} )+ D_i B ^\flat_{ij}(A_{ii}-A_{jj}).
\end{align*}
We also compute the diagonal matrix elements
\[
\left( \Omega^{-1} [ A^\mathsf{T},\Omega DB ^\flat ] \right) _{ii}= \sum_{k\in N(i)} (D_k-D_i) A_{ik}B^\flat _{ik}
\]
and the expression
\[
(D \cdot A) _i= ( \Omega ^{-1} A^\mathsf{T} \Omega D)_i= A_{ii}D_i - \sum_{k \in N(i), k\neq i}A_{ik}D_k.
\]
The final formula \eqref{discrete_LD} is obtained by using the three expressions above, the formula $ \mathbf{P} (X)_{ij} = \frac{1}{2} (X_{ij}- X_{ji}-X_{ii}+X_{jj})$, 
and noting several cancellations and rearrangements. 
\end{proof}

\begin{remark}[Continuous and discrete Lie derivative]\label{discrete_L}{\rm  Given two vector fields $ \mathbf{u} , \mathbf{v} $ on a manifold and a density $ \rho $, the Lie derivative $\mathbf{L}$ of the one-form density $ \rho \mathbf{v} ^\flat $ is given by
\begin{equation}\label{continuous_LD} 
\mathbf{L}_{ \mathbf{u} } ( \rho \,\mathbf{v} ^\flat )=  \rho\pounds_\mathbf{u}\mathbf{v}^\flat+\operatorname{div}(\rho \mathbf{u})\mathbf{v}^\flat=  \rho\, \mathbf{i} _{\mathbf{u} } \mathbf{d} \mathbf{v} ^\flat +\rho  \, 
\mathbf{d} ( \mathbf{v} ^\flat \cdot \mathbf{u} )+\operatorname{div}(\rho\,\mathbf{u} ) \mathbf{v} ^\flat ,
\end{equation} 
where $ \mathbf{d} $ is the exterior derivative and \textcolor{black}{$\pounds$ is the Lie derivative of a one-form}. Denoting by $ \omega := \mathbf{d} \mathbf{v} ^\flat $ the vorticity 
of $ \mathbf{v} $, we note the strict analogy between this formula and its discrete counterpart \eqref{discrete_LD}. 
The link between the discrete and continuous objects is established by
$ \mathbf{v} \sim B$, $ \mathbf{u} \sim A$, $ \rho \sim D$. The first two terms in \eqref{discrete_LD} correspond to the term  
$ \rho\, \mathbf{i} _{\mathbf{u} } \mathbf{d} \mathbf{v} ^\flat= \mathbf{i} _{ \rho \mathbf{u} } \omega $, the third term 
in \eqref{discrete_LD} corresponds to  $\rho  \, \mathbf{d} ( \mathbf{v} ^\flat \cdot \mathbf{u} )$ and the fourth term 
corresponds to $\operatorname{div}(\rho\,\mathbf{u} ) \mathbf{v} ^\flat $.}
\end{remark}

\begin{theorem} The \textcolor{black}{semidiscrete} equations \eqref{discrete_EP_compressible} for the discrete Lagrangian \eqref{discrete_Lagrangian} on a simplicial grid are given by 
{\fontsize{9pt}{10pt}\selectfont
\begin{equation}\label{simplicial_scheme}
\hspace{-0.4cm}\left\{
\begin{array}{l}
\displaystyle \vspace{0.2cm}\!\!\overline{D}_{ij} \frac{d}{dt} A^\flat_{ij} \\
\displaystyle \vspace{0.2cm} \qquad+  \omega _+ \left(K_i^+\overline{D}_{ji_+} A_{ii_+}+ K_j^+\overline{D}_{ij_+}A_{jj_+}\right) -\omega _- \left(K_i^-\overline{D}_{ji_-} A_{ii_-}
+ K_j^-\overline{D}_{ij_-}A_{jj_-}\right) \\
\displaystyle \vspace{0.2cm} \qquad  + \overline{D}_{ij} \frac{1}{2} \left( A_{ii_-} ^\flat A_{ii_-} +  A_{ii_+} ^\flat A_{ii_+}  + A_{ij} ^\flat A_{ij} - A_{ji} ^\flat A_{ji} 
 -  A_{jj_-} ^\flat A_{jj_-}- A_{jj_+} ^\flat A_{jj_+} \right) \\
\displaystyle \vspace{0.2cm} \qquad + \overline{D}_{ij} \left(  \frac{\partial \epsilon }{\partial D_i}- \frac{\partial \epsilon }{\partial D_j}  \right) =0\\
\displaystyle\!\!\frac{d}{dt} D_i = A_{ii_-}D_{i_-}+ A_{ii_+}D_{i_+} + A_{ij}D_{j}  - A_{ii}D_i,
\end{array} \right.
\end{equation}
}where $\overline{D}_{ij}:=  \frac{1}{2} (D_i+D_j)$ and $ \omega _{\pm}:=\sum_{ h_{mn}\in \partial \zeta _{\pm} } (A ^\flat  _{mn}+  R^\flat _{mn})$ is the discrete absolute vorticity at the node $\pm$, see Fig.~\ref{fig_Notation}.
\end{theorem}

\begin{proof} 
The functional derivatives of the discrete Lagrangian \eqref{discrete_Lagrangian} with respect to the pairing $ \left\langle \,, \right\rangle _0$ and $ \left\langle \,, \right\rangle _1$ are
\[
\frac{\delta \ell}{\delta A _{ij} }= D _i (A ^\flat _{ij} + R ^\flat _{ij} ) \quad\text{and}\quad\frac{\delta \ell}{\delta D_i}= \frac{1}{2} \sum_ j A ^\flat_{ij} A _{ij} + \sum_ j R ^\flat_{ij} A _{ij} - \frac{\partial \epsilon }{\partial D_i }.
\]
The system \eqref{simplicial_scheme} results from a series of computations. For the first term in \eqref{discrete_EP_compressible}, we have
\begin{align*} 
\frac{d}{dt} \mathbf{P} \left( \frac{\delta \ell}{\delta A}\right) _{ij} &= \frac{d}{dt} \left( \overline{D}_{ij}(A ^\flat _{ij} + R ^\flat _{ij} ) \right) = \frac{d}{dt} \overline{D}_{ij}  (A ^\flat _{ij} + R ^\flat _{ij} )+ \overline{D}_{ij}\frac{d}{dt} A ^\flat _{ij}\\
&= -(\overline{D\cdot A} )_{ij}(A ^\flat _{ij} + R ^\flat _{ij} )+ \overline{D}_{ij}\frac{d}{dt} A ^\flat _{ij},
\end{align*}
where in the last equality we use the discrete continuity equation \eqref{continuity}.
For the second term we use Lemma \ref{lemma} with $B ^\flat = A^\flat + R^\flat $. For the last term in \eqref{discrete_EP_compressible}, we compute
\begin{align*}
\mathbf{P} \left( D \frac{\delta \ell}{\delta D}^\mathsf{T}\right) _{ij} &=  \overline{D} _{ij}\Big( \sum_{k\in N(j)} \Big( \frac{1}{2} A^\flat _{jk}+R ^\flat _{jk} \Big)A_{jk}- \frac{\partial\varepsilon }{\partial D_j}  \Big)\\
&\quad \quad \quad \quad - \overline{D} _{ij} \Big(  \sum_{k\in N(i)} \Big( \frac{1}{2} A^\flat _{ik}+R ^\flat _{ik} \Big)A_{ik} - \frac{\partial\varepsilon }{\partial D_i}\Big).
\end{align*}
Adding these three terms and noting some cancellations, we get \eqref{simplicial_scheme}. 
\end{proof}


\begin{remark}{\rm The momentum equations \eqref{compressible_fluid_Riemannian} can equivalently be written in the space of one-forms as
\begin{equation}\label{one_form} 
\rho \,\partial _t \mathbf{u} ^\flat + \mathbf{i} _{ \rho \mathbf{u} } \mathbf{d} ( \mathbf{u}^\flat + \mathbf{R}^\flat )=- \rho\,\mathbf{d} \Big(  \frac{1}{2} | \mathbf{u} | ^2 + \frac{\partial \varepsilon }{\partial\rho} \Big).
\end{equation} 
It is this expression of the compressible equations that appears in a discretized form in the variational discretization \eqref{simplicial_scheme}.}
\end{remark}

\begin{remark}{\rm
The advection term in \eqref{simplicial_scheme} consists of two parts associated to either node $+$ or node $-$: 
the weighted sums of the absolute vorticity $\omega_+$ with those elements $A_{ii_+}$ and $A_{jj_+}$ that
contribute to $\omega_+$ and another weighted sum of $\omega_-$ with $A_{ii_-}$ and $A_{jj_-}$. This 
form follows naturally from the variational principle. In particular, it does not require the reconstruction 
of tangential velocities out of the prognostic normal ones, in contrast to standard C-grid schemes.

In standard C-grid discretizations of the shallow water equations (see, e.g., \cite{BauerPHD2013,Ringler2010,Thuburn2009}), 
the advection term is different. Usually, it is a product of a reconstructed tangential velocity value 
and an averaged absolute vorticity value evaluated at the triangles' edge midpoints. 
However, it is a nontrivial task to develop reconstructions and suitable average methods that provide well-behaving, 
conservative discretizations, as the resulting schemes might be inconsistent \cite{ThuburnCotter2015}.
}
\end{remark}

\medskip

\paragraph{The case of the rotating shallow water equations.} For the rotating shallow water equations, the density variable is 
the fluid depth denoted $h$. The Lagrangian is
\[
\ell( \mathbf{u} , h)= \int_M \Big[ \frac{1}{2} h  \,\mathbf{u} ^\flat \cdot \mathbf{u} + h \,\mathbf{R}^\flat  \cdot \mathbf{u}
- \frac{1}{2} g(h + B) ^2 \Big]  {\rm d} \mathbf{x},
\]
where $B$ is the bottom topography and where $h+B$ describes the free surface elevation of the fluid.
The above developments directly apply to this case by choosing the internal energy $ \varepsilon (h)= \frac{1}{2} g(h+B) ^2 $. 
Equations \eqref{compressible_fluid_Riemannian_more_general} becomes
\[
\partial _t \mathbf{u} +\mathbf{u} \cdot \nabla \mathbf{u} + 2( \mathbf{i} _ \mathbf{u}\boldsymbol{\Omega} )^\sharp = 
- g\operatorname{grad} (h+B), \quad \partial _t h + \operatorname{div}(h \mathbf{u} )=0 ,
\]
i.e., the rotating shallow water equations. The formulation \eqref{one_form} becomes
\begin{equation}\label{one_form_RSW} 
h \,\partial _t \mathbf{u} ^\flat + \mathbf{i} _{ h \mathbf{u} } \mathbf{d} ( \mathbf{u}^\flat 
+ \mathbf{R}^\flat )=- h\,\mathbf{d} \Big(  \frac{1}{2} | \mathbf{u} | ^2 + g(h+B) \Big).
\end{equation} 
The discrete rotating shallow water equations are obtained from the first equation in \eqref{simplicial_scheme} by replacing the last term with
\[
\overline{D}_{ij}g \big( (D_i +  B_i) - (D_j +  B_j)\big),
\]
where $D _i $ is the discretization of the fluid depth $h$. In this case \eqref{simplicial_scheme} becomes the discrete form of the formulation \eqref{one_form_RSW} of the rotating shallow water momentum equation.

\medskip

\begin{remark}{\rm Note that the first equation in \eqref{simplicial_scheme} follows from \eqref{discrete_EP_compressible} by using the advection equation \eqref{continuity}. 
Without making use of this, we get
{\fontsize{10pt}{9pt}\selectfont
\begin{equation}\label{discrete_EP_simplicial} 
\begin{aligned} 
&\frac{d}{dt} \big(\overline{D} _{ij} (A _{ij}^\flat + R_{ij}^\flat )\big) \\
&\;+ \omega _+ \left(K_i^+\overline{D}_{ji_+} A_{ii_+}+ K_j^+\overline{D}_{ij_+}A_{jj_+}\right) -\omega _- \left(K_i^-\overline{D}_{ji_-} A_{ii_-}
+ K_j^-\overline{D}_{ij_-}A_{jj_-}\right)\\
&\;  +\overline{D}_{ij} \frac{1}{2} \left( A_{ii_-} ^\flat A_{ii_-} +  A_{ii_+} ^\flat A_{ii_+}  + A_{ij} ^\flat A_{ij} - A_{ji} ^\flat A_{ji} 
 -  A_{jj_-} ^\flat A_{jj_-}- A_{jj_+} ^\flat A_{jj_+} \right)\\
&\; -(\overline{D \cdot A}) _{ij} (A _{ij} ^\flat + R _{ij}^\flat )+\overline{D}_{ij} \left(  \frac{\partial \epsilon }{\partial D_i}- \frac{\partial \epsilon }{\partial D_j}  \right)=0.
\end{aligned}
\end{equation}}}
\end{remark}

We end this section by giving in Table \ref{table} a summary that enlightens the correspondence between the continuous and discrete objects.

\begin{small}
\begin{table}[h!]\centering
\centering
\begin{tabular}{|c | c |} 
\hline
Continuous diffeomorphism & Discrete diffeomorphisms\\ [0.25ex] 
\hline
$\operatorname{Diff}(M)\ni\varphi $ & $\mathsf{D}( \mathbb{M}  )\ni q$ \\[0.25ex] 
\hline\hline
Lie algebra & Discrete diffeomorphisms\\ [0.25ex] 
\hline
$ \mathfrak{X}  (M)\ni\mathbf{u} $ & $ \mathfrak{d} ( \mathbb{M}  ) \ni A$ \\[0.25ex] 
\hline\hline
Group action on functions & Group action on discrete functions\\ [0.25ex] 
\hline
$f \mapsto f \circ \varphi $ & $F\mapsto q^{-1} F$  \\ [0.25ex] 
\hline\hline
Lie algebra action on functions & Lie algebra action on discrete functions\\ [0.25ex] 
\hline
$ f\mapsto \mathbf{d} f \cdot \mathbf{u}$ & $ F\mapsto -A F$  \\ [0.25ex] 
\hline\hline
Group action on densities & Group action on discrete densities\\ [0.25ex] 
\hline
$ \rho \mapsto ( \rho  \circ \varphi)J \varphi$ & $ D\mapsto \Omega^{-1} q^\mathsf{T}\Omega D$  \\ [0.25ex] 
\hline\hline
Lie algebra action on densities & Lie algebra action on discrete densities\\ [0.25ex] 
\hline
$ \rho \mapsto \operatorname{div}(\rho\mathbf{u} )$ & $D \mapsto \Omega^{-1} A^\mathsf{T}\Omega D$  \\ [0.25ex] 
\hline\hline
Hamilton's principle & Lagrange-d'Alembert principle  \\ [0.25ex] 
\hline
$\delta \int_0^T L_{ \rho _0}( \varphi , \dot{\varphi}) dt=0,$ & $ \delta \int_0^T L_{ D _0}( q , \dot q ) dt=0$, $\dot q q ^{-1} \in \mathcal{S} \cap \mathcal{R} $,\\
for arbitrary variations $\delta \varphi $
& for variations $\delta q q ^{-1} \in \mathcal{S} \cap \mathcal{R} $\\ [0.25ex]
\hline\hline
Eulerian velocity and density& Eulerian discrete velocity and discrete density  \\ [0.25ex] 
\hline
$ \mathbf{u} = \dot{ \varphi } \circ \varphi^{-1}$, $ \rho =(\rho _0 \circ \varphi ^{-1} ) J \varphi^{-1} $
& $ A= \dot{ q} q^{-1}$, $ D =\Omega^{-1} q^{-\mathsf{T}}\Omega D_0$\\ [0.25ex]
\hline\hline
Euler-Poincar\'e principle & Euler-Poincar\'e-d'Alembert principle\\ [0.25ex] 
\hline
$\delta \int_0^T \ell( \mathbf{u}  ,  \rho ) dt=0$, \;\; $ \delta \mathbf{u} =\partial _t \boldsymbol{\zeta} + [\boldsymbol{\zeta}, \mathbf{u} ]$, 
&$ \delta \int_0^T \ell( A , D ) dt=0$, \quad $ \delta A= \partial_t B+[B,A]$,  \\ [0.25ex]
\hspace{2.63cm}$ \delta \rho = - \operatorname{div}( \rho \boldsymbol{\zeta} )$
& \hspace{3.05cm}  $\delta D=- \Omega ^{-1} B^\mathsf{T} \Omega D$,\\[0.25ex]
\qquad  & \hspace{2.5cm}  $A, B \in \mathcal{S} \cap \mathcal{R}$ \\
\hline\hline
Compressible Euler equations & Discrete compressible Euler equations\\ [0.25ex] 
\hline
Form I: & Form I: on 2D simplicial grid \\
{\fontsize{9pt}{9pt}\selectfont$ \partial _t ( \rho ( \mathbf{u} ^\flat + \mathbf{R} ^\flat ))
+ \mathbf{i} _{\rho \mathbf{u} } \omega + \operatorname{div}( \rho\mathbf{u} ) ( \mathbf{u} ^\flat + \mathbf{R} ^\flat )$} 
& Equation \eqref{discrete_EP_simplicial}  \\
\hspace{3.2cm}{\fontsize{9pt}{9pt}\selectfont 
$=- \rho \,\mathbf{d} \big(  \frac{1}{2} | \mathbf{u} | ^2 + \frac{\partial \varepsilon }{\partial \rho } \big)$ } 
&  \\ [0.25ex]
Form II : & Form II: on 2D simplicial grid\\
{\fontsize{9pt}{9pt}\selectfont$  \rho\partial _t  \mathbf{u} ^\flat + \mathbf{i} _{\rho \mathbf{u} } \omega 
= - \rho \,\mathbf{d} \big( \frac{1}{2} | \mathbf{u} | ^2 +  \frac{\partial \varepsilon }{\partial \rho } \big) $}
&  Equation \eqref{simplicial_scheme} \\
\hline
\end{tabular}
\vspace{1em}
\caption{Continuous and discrete objects} \label{table} 
\end{table}
\end{small}

\section{Time discretization and solving algorithm}
\label{sec_explizit_scheme}

  We write the variational scheme of equations \eqref{simplicial_scheme} for the rotating shallow water equations explicitly in terms of 
  the normal velocities $V_{ij}$ and the cell densities (i.e. fluid depth) $D_i$. In particular,
  the elements of $A$ are given by:
  \begin{equation}\label{equ_explicit_matrixA}
\begin{aligned}  A_{ij} &= - \frac{1}{2 \Omega _{ii} } f_{ij} V _{ij}, \quad  \text{for all} \ j \in N(i), \ j\neq i, \\
   A_{ii} &= - A_{ij} -  A_{ii_-} - A_{ii_+}   = \sum_{k\in N(i),k\neq i}\frac{1}{2 \Omega _{ii} } f_{ik} V _{ik} =:\textcolor{black}{\frac{1}{2}} \operatorname{div} (V)_{i} .
  \end{aligned}
  \end{equation}
These are explicit representations of $A$ introduced in Lemma~\ref{lemma_matixA} on the simplicial mesh $\mathbb{M}$. 
Here, $\operatorname{div}(V)$ is a standard FV divergence operator on a triangular mesh, see \cite{BauerPHD2013} for instance.

\medskip

\paragraph{Semi-discrete equations.}
In terms of $V_{ij}$, the semidiscrete shallow water momentum equation \eqref{simplicial_scheme} becomes
\begin{equation}\label{equ_swe_discrete_expl}
\partial _t    V_{ij} +  \operatorname{Adv}_{\rm}(V,D)_{ij}= \operatorname{K}_{\rm}( V )_{ij} - \operatorname{G}(D)_{ij},
\end{equation} 
where we defined
  \begin{align*} 
  & \operatorname{Adv}_{\rm }(V,D)_{ij} : =  \\ 
  &   -  \frac{1}{\overline{D}_{ij}  h_{ij} } \Big( \frac{1}{ |\zeta_- | } 
  \sum_{ h _{mn}\in \partial \zeta _-} 
	h_{mn} { (V_{mn}  + \bar R_{mn}) } \Big) \left(   \frac{|\zeta_{-} \cap T_i|}{2 \Omega_{ii}  } \overline{D}_{ji_-}  f_{ii_-}   V_{ii_-} 
				    +  \frac{|\zeta_{-}  \cap T_j|}{2 \Omega_{jj}  } \overline{D}_{ij_-}  f_{jj_-}   V_{jj_-} \right) \\
  & +  \frac{1}{\overline{D}_{ij}  h_{ij} } \Big( \frac{1}{ |\zeta_ + | }\sum_{ h _{mn}\in \partial \zeta _+} 
	h_{mn} { (V_{mn}  + \bar R_{mn}) } \Big) \left(  \frac{|\zeta_{+}  \cap T_i|}{2 \Omega_{ii} }\overline{D}_{ji_+} f_{ii_+}   V_{ii_+}
				    + \frac{|\zeta_{+}  \cap T_j|}{2 \Omega_{jj}   } \overline{D}_{ij_+}   f_{jj_+}   V_{jj_+} \right),
  \end{align*} 
  \begin{align*} 
  \operatorname{K}_{\rm }( V) _{ij} &:= -  \frac{1}{2 h_{ij} }  \Big(  \frac{h_{jj_-} f_{jj_-} (V_{jj_-})^2  }{2\Omega_{jj}}  
								      + \frac{h_{jj_+} f_{jj_+} (V_{jj_+})^2  }{2\Omega_{jj}} 
								      + \frac{h_{ij  } f_{ij  } (V_{ji  })^2  }{2\Omega_{jj}} \\
				  &  \qquad \qquad  \qquad      - \frac{h_{ii_-} f_{ii_-} (V_{ii_-})^2  }{2\Omega_{ii}}
								      - \frac{h_{ii_+} f_{ii_+} (V_{ii_+})^2  }{2\Omega_{ii}} 
								      - \frac{h_{ij  } f_{ij  } (V_{ij  })^2  }{2\Omega_{ii}} \Big)   , \\
  \operatorname{G}(D)_{ij} &:=  \frac{g}{ h_{ij} }  \Big( D_j + B_j - (D_i +  B_i)   \Big)   ,
  \end{align*}
  
  \!\!\!\!
  
  \noindent
  for values $\bar R_{mn}$ related to $ R_{mn}$ by $ R_{ij} = - \frac{1}{2 \Omega _{ii} } f_{ij} \bar R_{ij}$,
  analogously to the relation between $V_{mn}$ and $A_{mn}$ given by \eqref{equ_explicit_matrixA}.

  
  For later use we define the curl-operator and a discrete tangential gradient operator
 \begin{equation}\label{curl_grad} 
  \operatorname{curl} (V)|_{\zeta _\pm} := \frac{1}{ |\zeta_\pm | } \sum_{  h _{mn}\in \partial \zeta _\pm} h_{mn}  V_{mn} \ , 
 \ \  \operatorname{G}^\perp(D)_{ij}:= \frac{1}{f_{ij} } \left( D_-+B_- - (D_+ + B_+)\right),
 \end{equation} 
 where $D_\pm$ is the fluid depth at the node $\pm$, see Fig.~\ref{fig_Notation}.
 Moreover, we define 
 \begin{equation}
  f|_{ \zeta _\pm}:= \frac{1}{ |\zeta_\pm | } \sum_{  h _{mn}\in \partial \zeta _\pm} h_{mn}  \bar R_{mn}.
 \end{equation}
 Later on, we will apply the $f$-plane approximation, i.e., for a Coriolis parameter $f$ (defined further below)
 whose values $f|_{ \zeta _\pm}= f$ are identical on each node. \textcolor{black}{The application of the variational scheme to the rotating shallow water system on the sphere is given in \cite{BrBaBiGBML2018}.}

%

Here, we include an explicit representation of the continuity equation \eqref{simplicial_scheme} by means of $V$ and $D$ to enable comparisons to standard methods. Using the fact that $A_{ii} = - A_{ij} - A_{ii_-} - A_{ii_+}$, there follows
\begin{equation}\label{equ_disc_cont}
\begin{array}{l}
\vspace{0.23cm}\partial _t D _i + A _{ii} D_i-  A _{ij} D_j- A _{ii_-} D_{i_-}- A _{ii_+} D_{i_+}  = 0 ,\\
\partial _t D _i  - A_{ij} (D_i + D_j ) -  A_{ii_-} (D_i + D_{i_-}) -  A_{ii_+} (D_i + D_{i_+})  = 0 ,\\
\partial _t D _i  + \underbrace{\frac{1}{\Omega _{ii}} f _{ij}  V _{ij} \overline{D}_{ij}   + \frac{1}{\Omega _{ii}} f _{ii_-}  V _{ii_-} \overline{D}_{ii_-}
  + \frac{1}{\Omega _{ii}} f _{ii_+}  V _{ii_+} \overline{D}_{ii_+}}_{:= {\rm div} (V, {D})_i}  = 0 ,\\ 
\end{array}
\end{equation}
where $\overline{D}_{ij} = \frac{D_i + D_j }{2}$.
Hence, the continuity equation \eqref{simplicial_scheme} can be written by means of a 
standard FV divergence operator ${\rm div}$ defined on triangles (cf. e.g. \cite{BauerPHD2013}).

\medskip

\paragraph{Time discretization.} The explicit representations \eqref{equ_swe_discrete_expl} and \eqref{equ_disc_cont}
permit one to apply standard time discretization methods, such as Runga-Kutta or Crank-Nicolson schemes while 
  applying operator splitting methods. 
  Here we proceed differently. Since the spatial discretization has been realized by variational principles in a structure preserving way, a temporal variational discretization can be implemented by following the discrete (in time) Euler-Poincar\'e-d'Alembert approach, analogously to what has been done in
 \cite{GaMuPaMaDe2011} and \cite{DeGaGBZe2014}, to which we refer for a detailed treatment. Following \cite{BRMa2009}, the discrete Euler-Poincar\'e approach is based on the introduction of a local approximant to the exponential map of the Lie group, chosen here as the Cayley transform $\tau$. Note that here the Cayley transform is only defined on an open subset of  $\mathfrak{d}(\mathbb{M})$ containing $\mathfrak{d}_{\rm vol}(\mathbb{M})$.

The temporal scheme consists of the following two steps. First, we compute the advected
 quantities, here the fluid density $D$ applying the Cayley transform $\tau$. 
 The update equation is then given by $D^{t+1}= \tau (\Delta t A^t)D^t$
 for the time $t$ and a time step size $\Delta t$. This equation, in particular $\tau$, can be represented as
\begin{equation}\label{equ_D} 
\big(I - \frac{1}{2} \Delta t A^t\big) D^{t+1} = \big(I + \frac{1}{2} \Delta t A^t\big) D^{t},
\end{equation}
with $I$ the identity matrix (cf. \cite{DeGaGBZe2014} for more details).
 The elements of the matrix $A$ in terms of $V_{ij}$ for the simplicial mesh $\mathbb{M}$ 
 are given in \eqref{equ_explicit_matrixA}.  
 
 Second, the following update equations for the momentum equation has to be solved:
 \begin{equation}\label{equ_fulldisc_momentum}
 \begin{aligned}
  \frac{V^{t+1}_{ij} -V^{t}_{ij} }{\Delta t} = &  -\frac{\operatorname{Adv}_{\rm }(V^{t+1},D^{t+1})_{ij} +\operatorname{Adv}_{\rm }(V^{t  },D^{t  })_{ij}}{2}\\
                                               &+\frac{\operatorname{K}_{\rm }(V^{t+1} )_{ij}  +\operatorname{K}_{\rm }(V^{t  } )_{ij}}{2} 
                                                 - \operatorname{G}(D^{t+1})_{ij} ,\;\; \text{for all} \ j \in N(i), \ j \neq i.
 \end{aligned}
 \end{equation}
 We solve this implicit nonlinear momentum equation by fixed point iteration 
 for all edges $ij$. To enhance readability, we skip the corresponding subindices in the following.
 The solving algorithm reads:
  \begin{enumerate}
  \item Start loop over $k=0$ with initial guess at $t$: $ V^{*}_{k=0} = V^{t}$; 
\item Calculate updated velocity $V_{k+1}^{*}$ from the explicit equation:
\[
\textcolor{black}{\frac{V^{*}_{k+1} -V^{t}}{\Delta t}  =-\frac{\operatorname{Adv}_{\rm }(V^*_k,D^{t+1})+\operatorname{Adv}_{\rm }(V^{t  },D^{t  })}{2} +\frac{\operatorname{K}_{\rm }(V^*_k )+\operatorname{K}_{\rm }(V^{t} )}{2}-\operatorname{G}(D^{t+1}),}
\]
then set $k+1 = k$;
  
\item Stop loop over $k$ if $||V^{*}_{k+1} - V^{*}_k|| < \epsilon$ for a small positive $\epsilon$.
\end{enumerate}
In case of convergence $V^{*}_{k+1} \rightarrow V^{t+1}$, this algorithm solves the momentum equation \eqref{equ_fulldisc_momentum}.

\begin{remark}{\rm 
In general, a structure preserving time discretization for the equations of the Euler-Poincar\'e 
type \eqref{discrete_EP_compressible} is obtained by applying a discrete analogue of the variational 
principle \eqref{VP_Euler}. This is the point of view followed in \cite{GaMuPaMaDe2011} and \cite{DeGaGBZe2014}, 
to which we refer for the complete treatment. As explained in these papers, by an appropriate choice of the Cayley 
transform and by dropping cubic terms, it results in a Crank-Nicolson type time update for the momentum 
equation \eqref{discrete_EP_compressible}, i.e., \eqref{discrete_EP_simplicial} for 2D simplicial grids. \textcolor{black}{While in absence of these cubic terms the resulting temporal scheme is no more variational, it was checked that this simplification does not significantly affect the behavior and properties of the solution on the tested configuration.}
We have used above this time update directly on the momentum equation as reformulated in \eqref{simplicial_scheme} 
which slightly differs from what would have been obtained by applying it to the \eqref{discrete_EP_simplicial}. 
This considerably simplifies the solving procedure without altering the behavior of the scheme. \textcolor{black}{We postpone to a future work the treatment of the fully variational time integrator for 2D and 3D compressible fluids.}}
\end{remark}

 \section{Numerical experiments}
 \label{sec_numerical_analysis}

 In this section we evaluate numerically the discrete variational shallow water equations derived in the previous chapters. 
 To this end we investigate whether (i) the scheme conserves stationary solutions such as a lake at rest or a steady isolated 
 vortex, whether (ii) it represents well the nonlinear dynamics, and whether (iii) the scheme 
 approximates well the frequency relations of the continuous equations.
 
 \medskip
 \paragraph{Computational meshes.}
 We perform all simulations in the $f$-plane approximation on a rectangular domain 
 $[0,L_x] \times [0,L_y]$ with $L_x = 5000\,$km and $L_y = 4330\,$km denoting the 
 domain's length in $x$- and $y$-directions, respectively.   
 The $f$-plane approximation corresponds to describing in 
 \eqref{Lagrangian_barotropic} the Earth rotation by the 
 vector potential $ \mathbf{R}  ( \mathbf{x} )=  \frac{1}{2}(-fy , f x) $.
 We apply double periodic 
 boundary conditions and a constant Coriolis parameter of 
 $f = 5.3108$ $\rm days^{-1}$ ($6.147\cdot 10^{-5} \rm s^{-1}$) which corresponds to 
 a latitude of $25^\circ$. 
 To account for the test cases' dimensions and enhance 
 readability, we use in the following units of km and days instead of m and s.

 The simulations are performed on regular and irregular triangular meshes, as shown in 
 Fig.~\ref{fig_meshes}. 
 We denote their resolutions by $N = 2(N_{1D})^2$, the number of triangular cells,
 in which $N_{\rm 1D}$ is the number of subintervals in $x$- or $y$-direction. 
 This number is defined with respect to a regular mesh 
 by $N_{\rm 1D}:  = L_x / f_{\rm uni} =  2 L_y / \sqrt{3} f_{\rm uni}$
 for the uniform triangle edge length $f_{\rm uni} = f_{ij}$ for all edges $ij$.

 The irregular mesh exhibit a refined region in the center of the domain. 
 The mean edge lengths of the cells in the center are about twice as small as of those
 in the outer regions. The refinement procedure starts from a regular mesh and 
 is controlled by a monitor function such that the topology of the mesh is conserved and 
 entanglement of the mesh is avoided (r-adaptivity). For more details on how to construct 
 such irregular meshes, we refer the reader to \cite{Baueretal2014} and references therein.
 
 These irregular meshes mimic realistic situations in operational forecasting, in which locally 
 refined meshes are employed to better resolve local small scale features. 
 It is important that these local refinement areas do not impact on the quality of the 
 global large scale flows. We employ these irregular meshes hence to illustrate that 
 the variational RSW scheme is capable to provide excellent results even on meshes with 
 locally refined areas consisting of cells that are possibly strongly non-regular.  
 
 \begin{figure}[t]\centering
 \begin{tabular}{cc}
  \hspace*{-0.0cm}{\includegraphics[scale=0.40]{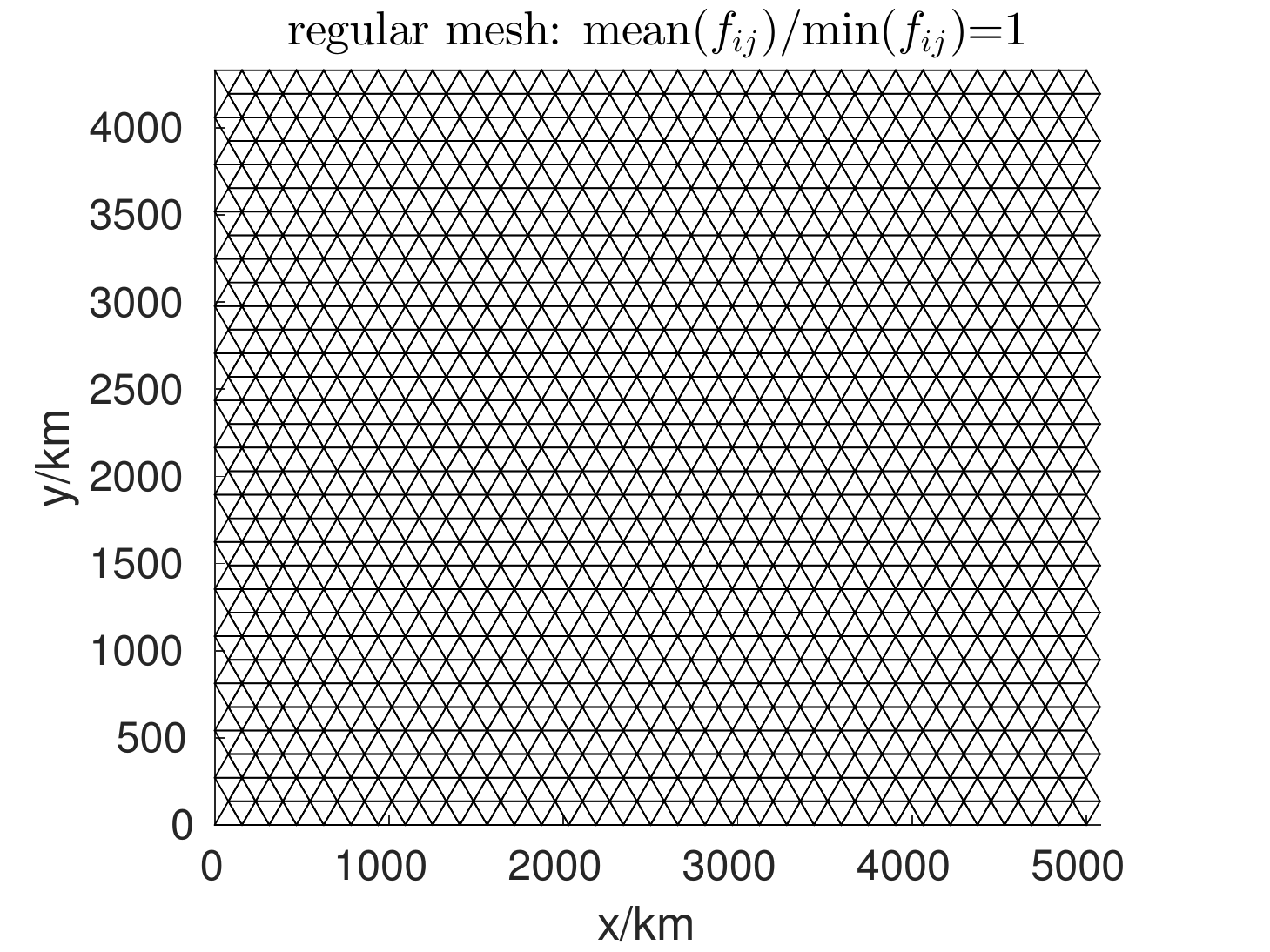}} &
  \hspace*{-0.0cm}{\includegraphics[scale=0.40]{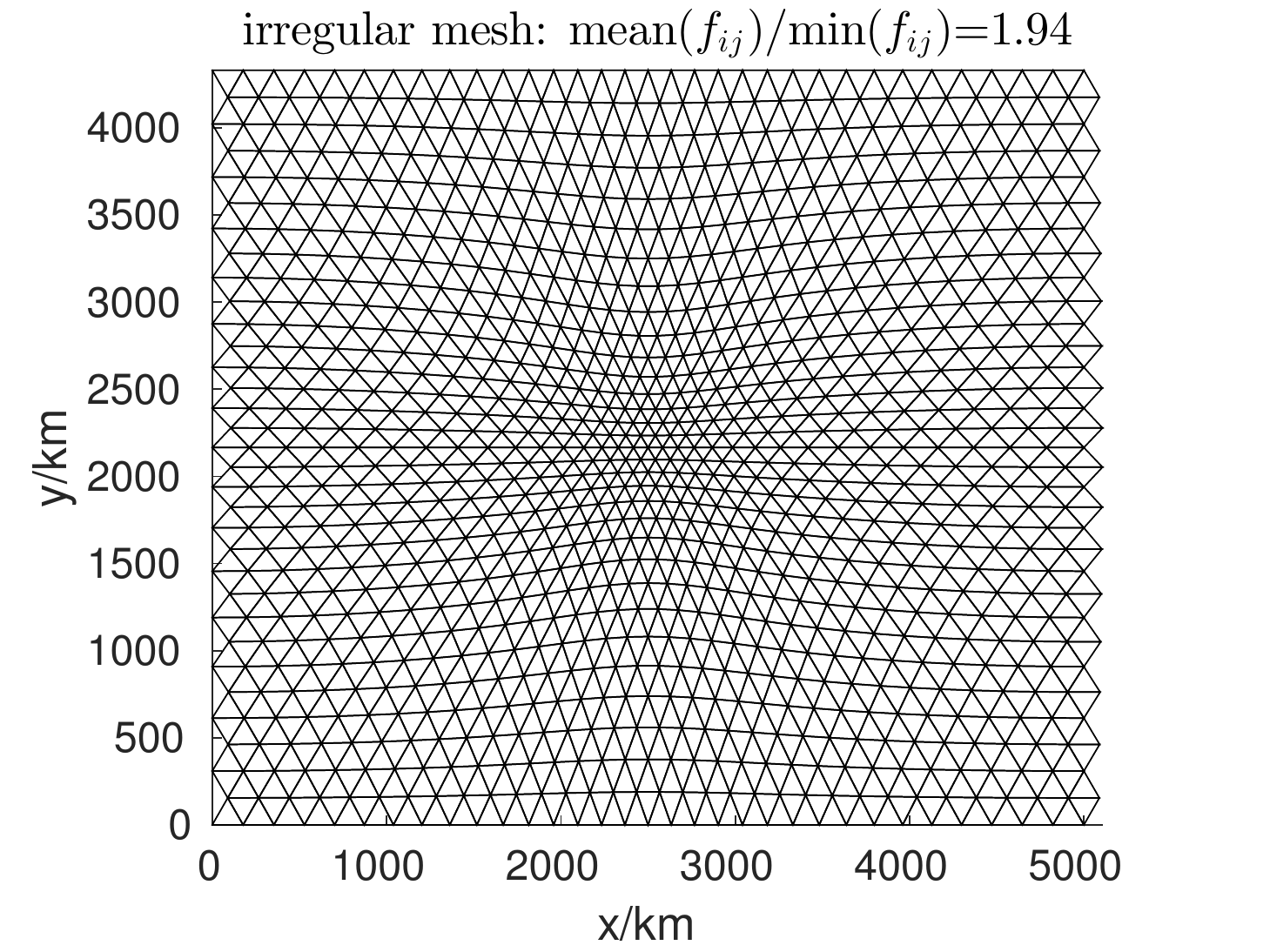}}
 \end{tabular}
 \caption{Regular mesh with equilateral triangles and irregular mesh with central refinement region,
          both with $2 \cdot 32^2$ triangular grid cells.
          }
 \label{fig_meshes}
  \end{figure}

\color{black}
\medskip

 \paragraph{Choice of spatial and temporal resolution.}
 We use test cases that are in the geostro\-phic regime in which the flow is dominated by the geostropic balance. 
 In this context, the Rossby deformation radius $L_D$ \eqref{equ_RandB} describes the length scale at which effects caused by 
 rotation are as important as those by gravity. 
 For the test cases studied, $L_D$ is at the order of $10^{3}\,$km (cf. Section~\ref{sec_tc_single}). 
 Our choice of domain size and spatial resolutions of $2\cdot 32^2$, $2\cdot 64^2$, $2\cdot 128^2$, and $2\cdot 256^2$
 throughout all simulations guarantees that $L_D$ is well resolved and geostrophic effects equally well 
 represented as gravitational ones. 

 Despite Crank Nicolson is an implicit time scheme and is, as such theoretically unconditionally stable, 
 in practice the condition number of the implicit system decreases with larger time steps 
 until the iterative solver fails to converge. This imposes an upper bound on the time step also for implicit schemes. To evaluate the ability of the scheme to handle large  time 
 steps, we use the gravity Courant number (or CFL number)
 \begin{equation}\label{CFLcond}
  C = \sqrt{g H_0} \frac{\Delta t}{\Delta x_{min}} < C_{max} 
 \end{equation}
 with gravitational constant $g = 7.32\cdot10^7\,$ km $\rm days^{-2}$ and water depth 
 $H_0\, [km]$, where $c = \sqrt{g H_0}$ 
 is the speed of the fastest traveling wave,
 $\Delta x_{min}:= \min_{ij}(h_{ij})$ is the shortest dual edge length of the mesh,
 and $\Delta t$ is the time step size. In contrast to explicit schemes with necessarily $C_{max} \leq 1$,
 implicit schemes might reach some multiples of this. 
 For the test cases studied below, our implicit time integrator achieves a maximal 
 Courant number of about $C_{max} = 3$. 
 
 Individually for each test case, we choose one fixed time step $\Delta t$ (unless indicated otherwise).
 It is bounded via \eqref{CFLcond} by $C_{max} = 3$, the largest water depth $H_0$ applied, and $\Delta x_{min}$ 
 of the irregular mesh with highest resolution used. Besides guaranteeing that the iterative solver converges for all meshes applied and all flow regimes studied, the fixed time step 
 per test case permits us to distinguish between error sources related to the spatial and to the 
 temporal discretizations.
 
 \color{black}

 \medskip
 \paragraph{Quantities of interest.}
 For all test cases we are particularly interested in studying the time evolution of the relative errors in the 
 following quantities of interest (QOI): mass, total energy, mass-weighted potential vorticity, and potential enstrophy. 
 As these values are conserved quantities in time of the continuous shallow water equations, we study if the 
 corresponding discrete values are conserved too. These quantities can be calculated as follows. The total 
 mass $m(t)$ follows as an integral of the fluid depth $h({\bf x},t)$ over the domain $M$ 
 and is approximated by
 \begin{equation}
  m= \int_M h \, {\rm d}{\bf x} \; \approx \; \sum_{i=1}^N D_i \Omega_{ii} .
 \end{equation}
 
 The total energy $E_{\rm tot}(t)$ is the sum of kinetic $E_{\rm kin}(t)$ and potential energy $E_{\rm pot}(t)$ 
 which are given, respectively, by
 \begin{equation}
 \begin{aligned}
  E_{\rm kin}& = \int_M \frac{1}{2} h{\bf u}^2  {\rm d}{\bf x}     
            \; \approx \; \sum_{i=1}^N \frac{1}{2} D_i \Omega_{ii}\!\!\!\!\!  \sum_{k\in N(i), k\neq i}\frac{h_{ik  } f_{ik  } V_{ik}^2  }{2\Omega_{ii}} \, , \\
  E_{\rm pot}& = \int_M  \frac{1}{2} g  (h+B)^2  {\rm d}\mathbf{x}  \; \approx \; 
  \sum_{i=1}^N \frac{1}{2}  g (D_i+ B_i)^2\Omega_{ii} \, .
  \end{aligned}
 \end{equation}

 Defining the absolute vorticity $ \omega_a := \operatorname{curl}  \mathbf{u} +f$,
 the potential vorticity $q:= \frac{ \omega_a }{h} $,
 and the relative potential vorticity 
 $q_{\rm rel}:= \frac{\operatorname{curl}  \mathbf{u}}{h}$, 
 the conserved quantities of \textit{mass-weighted potential vorticity} $PV$ and {\em potential enstrophy} 
 $PE$ are given by
 \begin{equation}\label{conserved_quantities} 
 \begin{aligned}
  PV &=  \int_M qh\,  {\rm d}{\bf x}    
            \; \approx \; \sum_{e =1}^{N_e}  \left ( \operatorname{curl}(V)|_{\zeta_e} + f \right)  |\zeta_e|     \, ,\\
  PE &= \frac{1}{2}\int_M q^2 h\,  {\rm d}{\bf x} \; \approx \; \frac{1}{2}\sum_{e =1}^{N_e} \frac{(\operatorname{curl}(V)|_{\zeta_e} + f )^2}{D_e} |\zeta_e| \, ,
 \end{aligned}
 \end{equation}
 using \eqref{curl_grad}, where $N_e$ denotes the total number of nodes.
 The depth $D_e$ associated to dual cells $e$ is obtained by an area 
 weighted average of neighboring cell values $D_i$ using the coefficients $K^e_i$ of \eqref{equ_K}, i.e., 
 $D_e= \sum_{i = 1}^6 K^e_i D_i$. 
 The functions in \eqref{conserved_quantities} are examples of Casimir functions for the rotating 
 shallow water equations, whose general form is $\int_M \Phi (q)h  \,{\rm d} \mathbf{x} $, where $ \Phi $ is an 
 arbitrary function of the potential vorticity.

 \subsection{Well-balancedness and frequency representation}

By means of two test cases, a lake at rest and a lake at rest with small 
disturbance in the surface elevation but trivial bottom topography,
 we investigate whether the variational integrator is capable to preserve stationary
 solutions of the shallow water equations without generating spurious oscillations 
 and without generating or destroying mass.  In particular we study if the scheme 
 is well-balanced with respect to the lake at rest steady state. In addition, we numerically determine the 
 frequency spectrum of occurring surface waves and compare it with the theory.

 \subsubsection{Lake at rest}
 
 This test case serves us to illustrate that our implementation can perfectly 
 handle nontrivial bottom topography. Initializing with constant surface elevation 
 $h + B = \rm const.$ and zero velocity $u = 0$ (cf. \cite{LeVeque}),
 we expect the scheme to conserve this steady solution in case of flat and nontrivial
 bottom topography without exciting spurious modes. We expect further that 
 the QOI are preserved too.

 \medskip
 \paragraph{Initialization.} 
 The setup for the lake at rest consists in a fluid in rest with zero initial velocity, 
 $V_{ij} = 0$ for all edges, and with a surface elevation that coincides with the 
 background depth, here $H_0 = 750\,$m, such that $D_i + B_i = H_0$.
 The function for bottom topography $B({\bf x})$
 describes an underwater island, positioned next to the domain center, that is given by 
  \begin{equation}\label{equ_init_lake}
    B(x,y) =  B' e^{-\frac{1}{2}\big(  \frac{1}{\sigma_x ^2 }(x - x_{c_1}) ^2 
                                   +   \frac{1}{\sigma_y ^2 }(y - y_{c_1})^2  \big)}   \ , 
  \end{equation}
 for $B' = 100\,$m, $x_{c_1} = (\frac{1}{2}-o) L_x$, $y_{c_1} = (\frac{1}{2}-o)L_y$, $o = 0.1$, 
 $\sigma_x = \frac{3}{40}L_x$, and $\sigma_y = \frac{3}{40}L_y$.

 We obtain the discrete function for $B_i$ by sampling \eqref{equ_init_lake} 
 at the centers of the triangles $T_i$. Here and henceforth, we denote the discrete fields such as $B_i$, 
 $D_i$ and $(q_{\rm rel})_i$ with $B({x,y})$, $D({x,y})$, and $q_{\rm rel}({x,y})$,
 respectively, when considering them as 2D fields that depend on the $x$- and $y$-directions
 (cf. Fig.~\ref{fig_D_diff_max}, for instance).

\color{black}
Here and for the frequency test case below, we use a time step of $\Delta t = 60\,$s. 
 According to \eqref{CFLcond} with $\Delta x_{min}= 5.183\,$km 
 of the irregular mesh with $2\cdot 64^2$ cells and $H_0 = 750\,$m, the Courant number is $C = 0.99$
 which is below $C_{max} =3$. 
 \color{black}
%
%

 \begin{figure}[t]\centering
 \begin{minipage}{0.45\linewidth}
 \begin{tabular}{c}
  \hspace*{-0cm}{\includegraphics[scale=0.5]{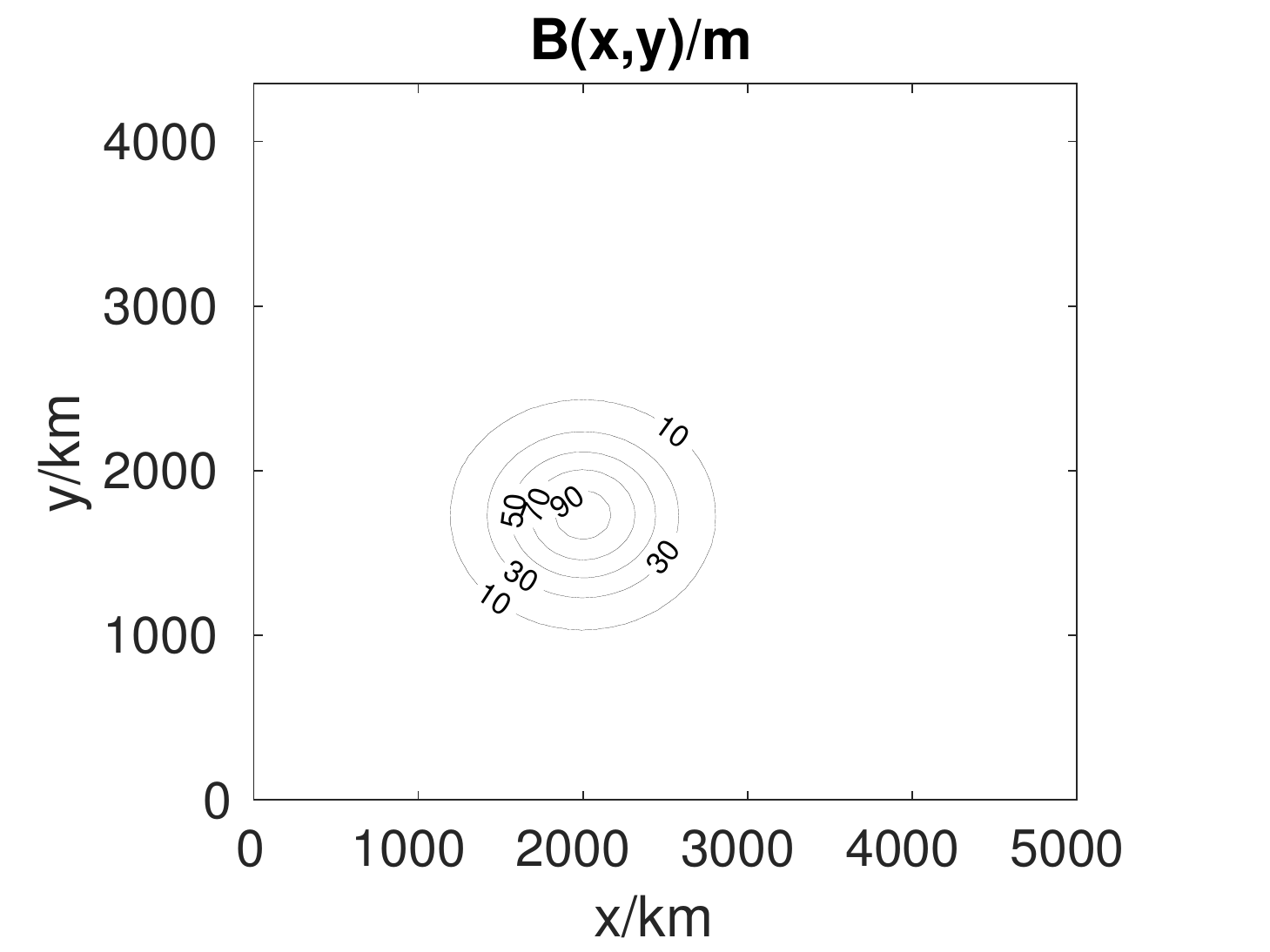}} 
 \end{tabular}
 \end{minipage}\hfill
 \begin{minipage}{0.55\linewidth}
   \begin{tabular}{c}
  \hspace*{-0cm}{\includegraphics[scale=0.25]{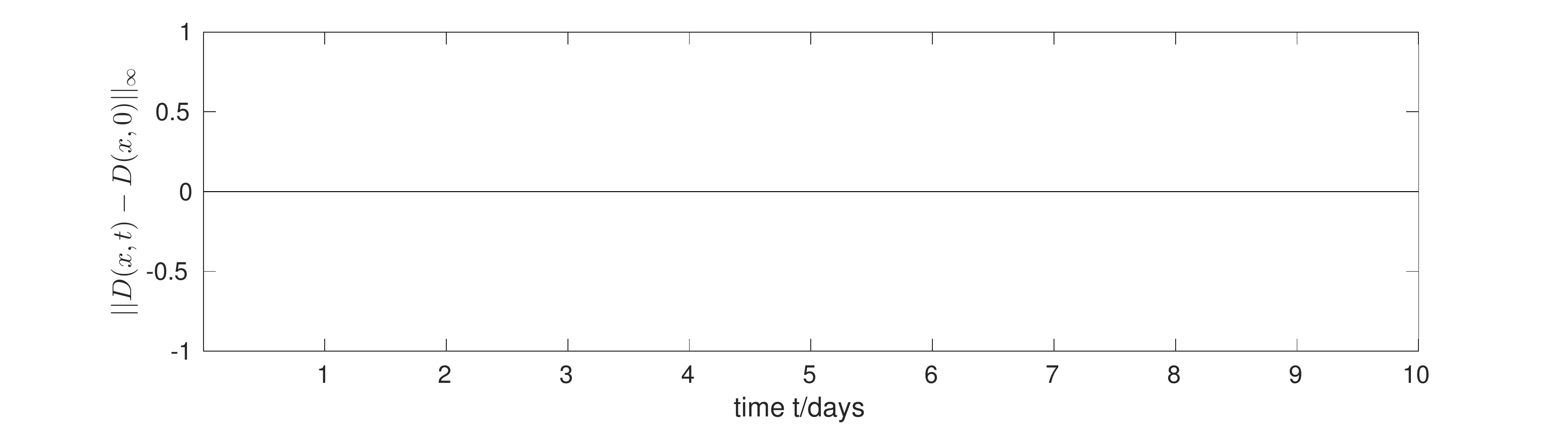}} \\
  \hspace*{-0cm}{\includegraphics[scale=0.25]{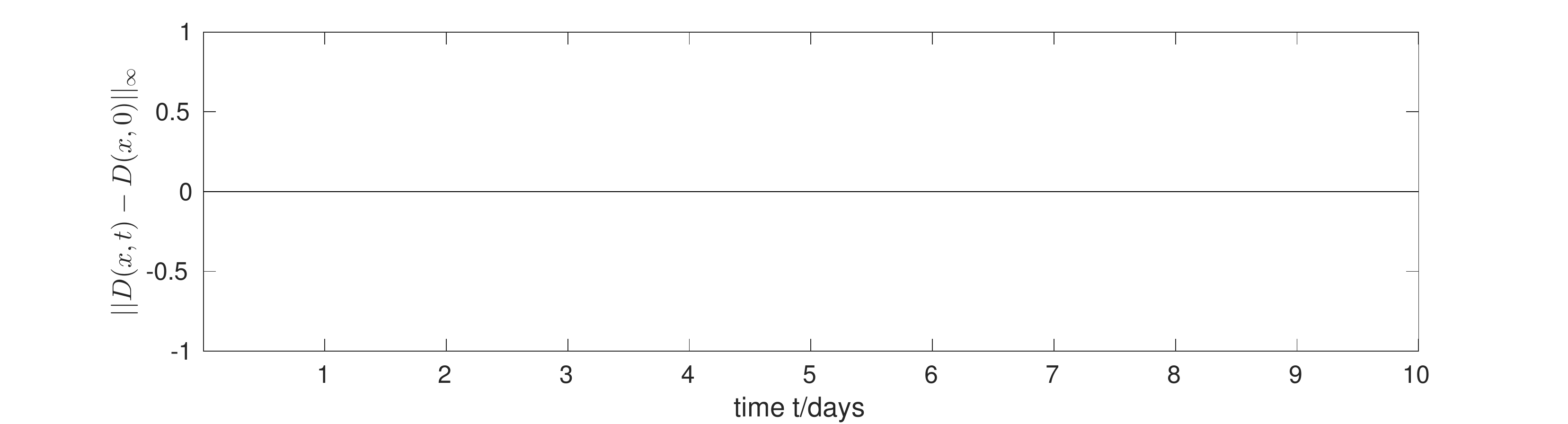}} 
 \end{tabular}
 \end{minipage} 
 \caption{Left: contour lines of the bottom topography $B(x,y)$ on the computational domain. 
          Right: maximum errors in surface elevation at rest relative to $H_0 = 750\,$m for 
          regular (upper right) and irregular (lower right) meshes.}  
   \label{fig_D_diff_max}
  \end{figure}

 \bigskip 

 \paragraph{Results.}
 We consider the maximum error over all cells $i$ of the surface elevation $D_i(t) + B_i$ minus 
 the initial value $D_i(0) + B_i = H_0$, denoted by $|| D(t) - D(0) ||_{\infty}$
 when solving the nonlinear equations \eqref{equ_D}-\eqref{equ_fulldisc_momentum}.
 The corresponding error values for regular and irregular meshes are shown in the upper-right and 
 lower-right panels of Fig.~\ref{fig_D_diff_max}, respectively. Clearly, even in case of a 
 nontrivial bottom topography (Fig.~\ref{fig_D_diff_max}, left), the surface elevation is preserved at
 machine precision for both mesh types.
 In addition, all QOI are preserved at machine precision too (not shown).
 This allows us to conclude that the scheme perfectly satisfy the well-balanced property and 
 does not generate spurious modes in case of nontrivial bottom topography.

\subsubsection{Frequency spectrum of linearized shallow water equations}

 Here we check if the occurring wave frequencies agree with the theory. We assume trivial bottom topography. Linearizing the equations around the undisturbed fluid depth $H_0$ and zero velocity and inserting 
 plane wave solutions of the form $h(x,y,t) = H_0 \exp(i(kx + ly - \omega t))$, there follow the solutions
 \begin{equation}\label{equ_disp_rel_swe}
  \omega= 0 \qquad \text{or} \qquad \omega^2 = f^2 + c^2(k^2 + l^2)
 \end{equation}
 for $c = \sqrt{gH_0}$, wave frequency $\omega$, and wave numbers $k,l$ in $x,y$-direction, respectively. 
 From this relation we thus have either a stationary solution $(\omega =0)$ or waves with frequencies greater 
 than the Coriolis frequency $f$, i.e., $ \omega \geq f$. The case $ \omega  = f$, i.e., $k=l=0$, corresponds 
 to \textit{inertial oscillations} which do not propagate.  
 Because of the double periodic boundary conditions, these waves are not excited here. 
 The case $\omega >f$, corresponds to inertia-gravity (or Poincar\'e) waves, cf. \cite{Zeitlin2007}. 
 Since we have a bounded, double periodic computational domain, the permitted wave numbers for inertia-gravity 
 waves are $k = \frac{n_x 2\pi}{L_x}$ and $l = \frac{n_y 2\pi}{L_y}$ for $n_x,n_y = 0,1,2,\dots$ with $n_x + n_y > 0$. 
 This gives a minimum wavenumber in each direction. Hence, all frequencies are greater than $f$ but there is no maximal wavenumber.

 \bigskip

 \paragraph{Initialization of surface disturbance.}
 To obtain an initialization that is close to the values $\bar h(\mathbf{x},t) = H_0$ and 
 $\mathbf{\bar u}(\mathbf{x},t)=0$ around which we linearized the shallow water equations, we superimpose 
 on the lake at rest a small disturbance of magnitude $H' = 7.5\,$m in the surface elevation.
 Hence, we apply the fluid depth
 \begin{equation}\label{equ_init_single_vortex_h}
  h(x,y,0) =  H_0 - H' \left[   e^{-\frac{1}{2}({x_1'}^2 + {y_1'}^2  )}  
  - \frac{4\pi \sigma_x \sigma_y}{L_x L_y}     \right] \, , 
 \end{equation}
 using the periodic extensions
 \begin{equation}\label{equ_define_center}
 \begin{split}
  x'_{1} = \frac{L_x}{\pi \sigma_x}\sin \left( \frac{\pi}{L_x}(x - x_{c_{1}}) \right)\, , \quad 
     y'_{1} = \frac{L_y}{\pi \sigma_y}\sin \left( \frac{\pi}{L_y}(y - y_{c_{1}}) \right)\, .
 \end{split}
 \end{equation}
 The center of the perturbation is positioned at $x_{c_{1}} = \frac{1}{2}L_x, y_{c_{1}} = \frac{1}{2} L_y$.
 To obtain a circular initial (negative) surface evaluation, we use only one value for sigma, 
 i.e. $\sigma_x = \sigma_y =  \frac{3}{40}L_y$. 
 Note that in terms of implementation, we initialize the fluid depth $D_i$ by 
 sampling \eqref{equ_init_single_vortex_h} at each cell center 
 and we set all velocity values $V_{ij}$ to zero. 
 We set all $B_i =0$ to apply trivial bottom topography.
 
\color{black}
Using the same time step $\Delta t = 60\,$s and the same meshes (i.e. $\Delta x_{min} = 5.183\,$km) as for the lake at rest,
 the Courant number for case (i) with $H_0 = 750\,$m is again $C = 0.99$ but for case (ii) with $H_0= 1267.5\,$m it is $C = 1.29$, both well below $C_{max}$. 
 \color{black}

%

 \bigskip
 
 \paragraph{Results of the frequency spectrum study.}
 Recall that we use a gravitational constant of $g = 7.32\cdot10^7\,$ km $\rm days^{-2}$
 and a double periodic domain with wave numbers $k = \frac{n_x 2\pi}{L_x}$ 
 and $l = \frac{n_y 2\pi}{L_y}$ for $n_x,n_y = 0,1,2,\dots$. 
 According to the dispersion relation in Eqn.~\eqref{equ_disp_rel_swe} (right),
 we find for two sets of parameters, namely 
 case (i) $f = 5.31\,$ $\rm days^{-1}$, $H_0 = 750\,$m, and 
 case (ii) $f= 6.903\,$ $\rm days^{-1}$, $H_0= 1267.5\,$m, 
 the following frequencies $ \omega (n_x,n_y)$ in units of $\rm rad \ days^{-1}$ 
 for some combinations of $n_x$ and $n_y$: 

 \medskip
 \noindent
 \begin{small}
 \begin{tabular}{c|ccccccccc}
                  & \hspace{-0.5em} ${ \omega}(0,0)$ & \hspace{-0.5em} ${ \omega}(1,0)$ & \hspace{-0.5em} ${ \omega}(0,1)$ & \hspace{-0.5em} ${ \omega}(1,1)$ 
                  & \hspace{-0.5em} ${ \omega}(2,0)$ & \hspace{-0.5em} ${ \omega}(2,1)$ & \hspace{-0.5em} ${ \omega}(0,2)$ & \hspace{-0.5em} ${ \omega}(1,2)$ & \hspace{-0.5em} ${ \omega}(2,2)$ \\ \hline   
  case (i) &  5.3          &   10.7        &  12.0         &   15.2        & 19.4          & 22.1          & 22.2          & 24.0         & 28.9         \\ \hline 
  case (ii) &  6.9          &   13.9        &  15.6         &   19.7        & 25.2          & 28.8          & 28.8          & 31.2         & 37.6         \\ \hline   
  \end{tabular}
  \end{small}

  \medskip
  \noindent
  The parameters for $f$ and $H_0$ have been chosen such that the flow 
  remains for both case (i) and case (ii) in the quasi-geostrophic regime 
  (i.e. ${\rm Bu} \approx 1$, cf. \eqref{equ_RandB}).
  
  To verify if these theoretical values are well represented by the variational shallow water 
  scheme, we determine the frequencies occurring during the simulations.
  To this end, we numerically calculate at the center of the domain 
  the Fourier transforms of a time series of the fluid depth 
  $D(x,z,t)$ for the time interval $t \in [0,10\,{\rm days}]$ with a sample frequency of $0.01\,$days. 
  The resulting spectra for the two choices of parameters are shown in 
  Fig.~\ref{fig_freq_singlevortex_irreg_center},
  left for case (i) and right for case (ii).
  
  Besides some small background noise of waves occupying all possible wave numbers,
  we clearly distinguishes sharp peaks in the spectra exactly at the predicted wave 
  numbers for both parameter sets. 
  For the illustrated combinations of $k,l\leq2$, this perfect match of expected
  and numerically determined values can easily be seen when comparing the values from the 
  table with those from Fig.~\ref{fig_freq_singlevortex_irreg_center}, while for combinations with 
  larger wave numbers the associated peaks might overlap (not shown). The overlap of 
  the frequencies ${ \omega}(2,1)$ and ${ \omega}(0,2)$ reflects itself in a nearly doubled magnitude
  of the associated peak. Neither at case (i) nor (ii) we observe unphysical solutions 
  at the frequencies 
  $f = 5.31\,$ $\rm days^{-1}$ or $f= 6.903\,$ $\rm days^{-1}$, respectively. 
  In addition, we notice the dependency of the spectra 
  on the parameter $f$ when comparing the left and the right spectra. In the latter, the 
  peaks are shifted slightly to higher wave numbers because of the greater Coriolis 
  parameter, in agreement with \eqref{equ_disp_rel_swe}.

 \begin{figure}[t!]\centering
 \begin{tabular}{cc}
    \hspace*{-0cm}{\includegraphics[scale=0.45]{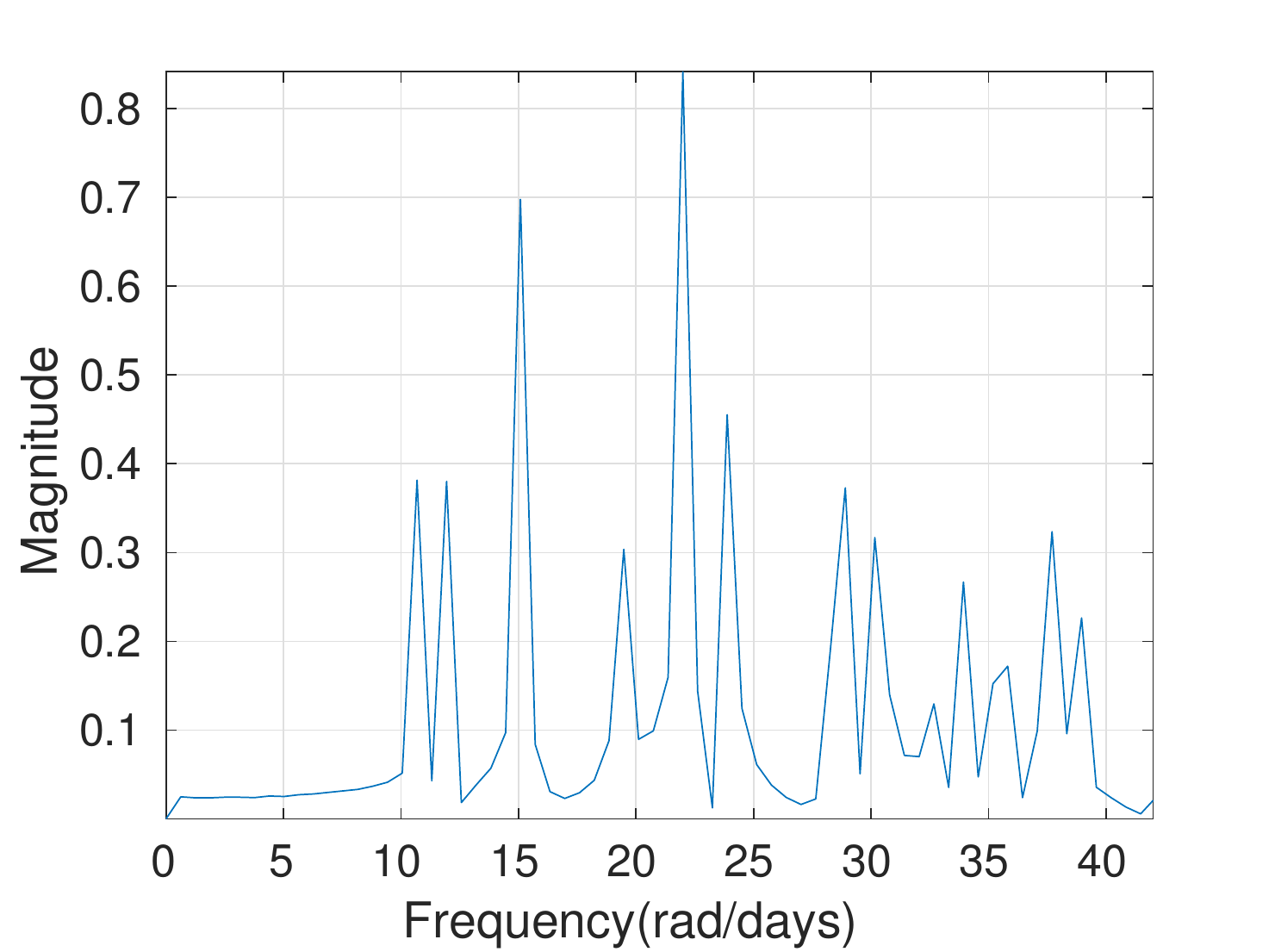}}&
    \hspace*{-0cm}{\includegraphics[scale=0.45]{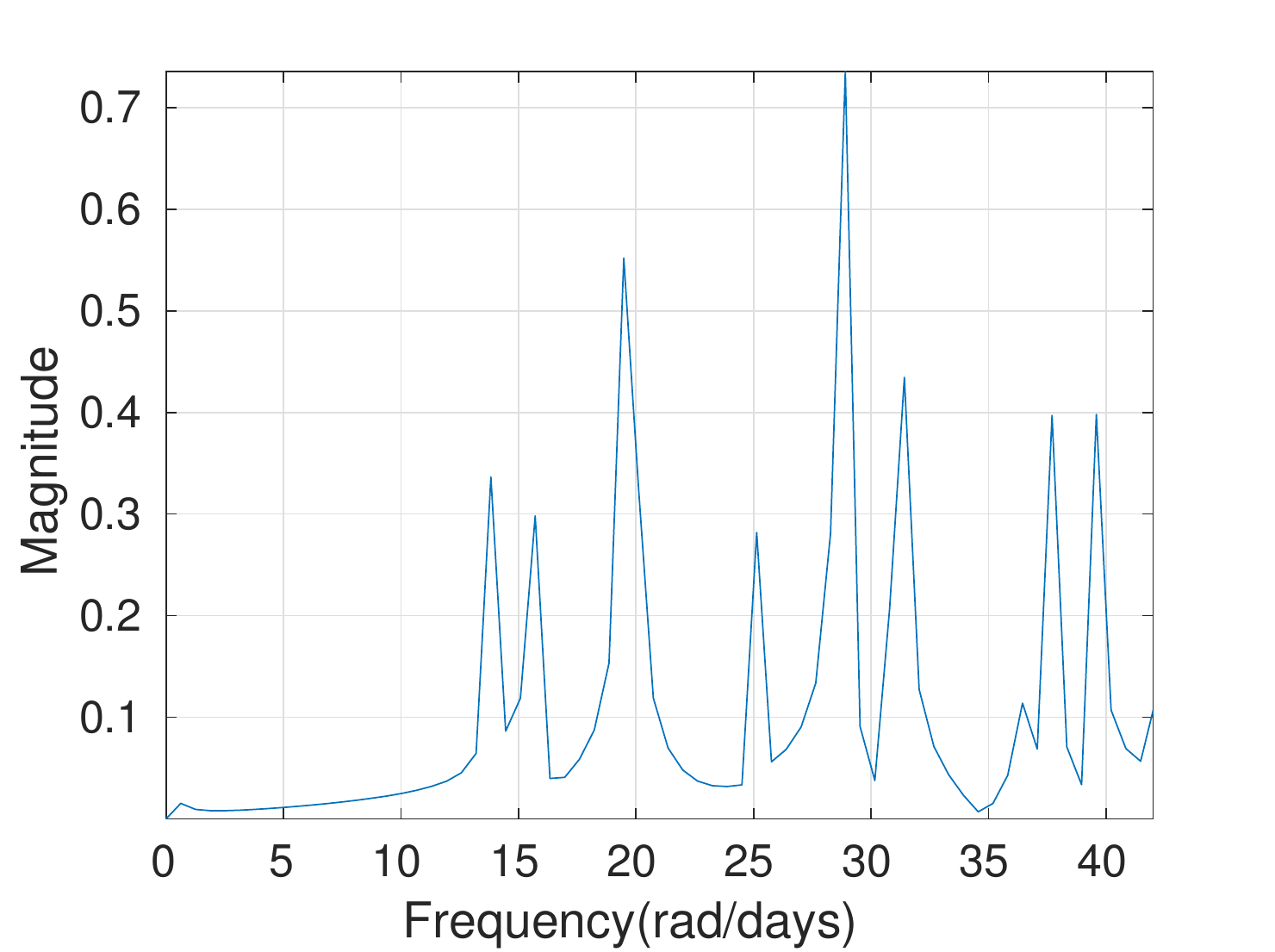}}  
  \end{tabular}
 \caption{Frequency spectra of the disturbed lake at rest after 10 days 
         for parameters 
         $f = 5.31\,$ $\rm days^{-1}$, $H_0 = 750\,$m (left) or 
               $f= 6.903\,$ $\rm days^{-1}$, $H_0= 1267.5\,$m (right)
         determined on an irregular mesh with $2 \cdot 64^2$ cells. 
  The frequency spectra determined on regular meshes looks very similar (not shown).}                                                                                             
  \label{fig_freq_singlevortex_irreg_center}
 \end{figure}


 \subsection{Conservation of exact, steady solution of an isolated vortex}
 \label{sec_tc_single}

 Here we test if the variational shallow water scheme preserves the stationary solution 
 of an isolated vortex. We perform long-term simulations 
 up to 100 days and evaluate alongside the conservation properties of 
 mass, total energy, mass-weighted potential vorticity, and potential enstrophy.
 For the \textcolor{black}{long-term} simulation we apply either regular or irregular computational meshes 
 with only $2 \cdot 64^2$ triangular cells because potential instabilities usually occur earlier
 for coarser mesh resolutions.
 
  
 Comparing the numerical solutions at day 1, for instance, with the initial state allows us further
 to determine the solutions spatial convergence behavior since, being stationary, every deviation 
 is due to numerical errors.
 
 \bigskip

 \paragraph{Initialization.}  
 A stationary vortex solution of the rotating shallow water equations  with trivial bottom topography has the velocity $ \mathbf{u} ( \mathbf{x} )$ and 
 fluid depth $h( \mathbf{x} )$ given by
 \begin{equation}\label{equ_vel_calc}
 \mathbf{u} ( \mathbf{x} )= V(r) \frac{(-y, x)}{r} \quad\text{and}\quad h( \mathbf{x} ) = H(r),
 \quad r= \sqrt{ {x_1'}^2 + {y_1'} ^2 },
 \end{equation}
 where $V(r)$ and $H(r)$ verify the gradient wind balance
 \begin{equation}\label{relation_V_h1}
  \frac{V(r)^2}{r}+ f V(r)   = g \frac{\partial H(r)}{\partial r} ,
 \end{equation} 
 (cf.~\cite{StDr2000}, for instance). This condition 
 complies with the construction method for steady state solutions of the 
 RSW equations suggested by \cite{Staniford_White2007}. 
 Here, we apply relation \eqref{relation_V_h1} to construct a stationary solutions 
 and consider a test case with trivial bottom topography $B({\bf x}) = 0$.

 We consider the following radial function to describe the velocity (resp. streamfunction):
 \begin{equation}\label{equ_steady_V}
   V(r)= u _{0} \frac{r}{r_0} e^{- \frac{1}{2}(\frac{r}{r_0})^2}, \quad \text{resp.} \quad
   \Psi(r)= - u _{0} {r_0} e^{- \frac{1}{2}(\frac{r}{r_0})^2}.
 \end{equation}
 The function $V(r)$ results from choosing the exponential vorticity profile suggested by \cite{StDr2000} for $\alpha =2$ 
 combined with the geophysically relevant scaling discussed in \cite{Staniford_White2007}.
 Its integration with respect to $r$ gives the streamfunction $\Psi(r)$. 
 The corresponding radial function for the fluid depth, which follows from \eqref{relation_V_h1},
 reads
 \begin{equation}\label{equ_steady_H}
  H(r) = H_0 - \frac{u _{0} ^2 }{2g}e^{- (\frac{r}{r_0})^2 }- \frac{f u _{0} r_0}{g} e^{- \frac{1}{2}(\frac{r}{r_0})^2 },
 \end{equation}
 where $u_0$ describes the maximal velocity and $r_0$ is a scaling constant, both to be determined further below. 
 
 \begin{remark}\rm
   Equations \eqref{equ_steady_V} and \eqref{equ_steady_H} propose an exact, stationary 
   solution of the rotating shallow water equations \eqref{one_form_RSW} 
   on the plane for trivial bottom topography.  
   Topography can be easily included, as this steady 
   solution is consistent with the construction method of \cite{Staniford_White2007} which allows for such modification. 
   Moreover, our steady solution is an alternative example to that suggested by \cite{Staniford_White2007} of a steady isolated 
   vortex. It provides a Cartesian analog for the famous 
   test case 2 of \cite{Williamson1992} of a steady solution of the 
   shallow water equations in spherical geometry which is frequently applied to 
   measure a scheme's ability to preserve large-scale geostrophic balance 
   (see \cite{Ringler2010}, for instance).
 \end{remark}

 To position the vortex in the center of the domain, we use the definitions
 $x'_{1} = \left( x - x_{c_{1}} \right)$ and
 $y'_{1} = \left( y - y_{c_{1}} \right)$ for 
 $x_{c_{1}} = \frac{1}{2}L_x, y_{c_{1}} = \frac{1}{2} L_y$, similarly to 
 \eqref{equ_define_center} but omitting the periodic extension. 
 We take a sufficiently large domain and a corresponding scaling parameter $r_0$ 
 so that the fluid is at rest at the boundaries.
 We initialize the fluid depth $D_i$ by sampling \eqref{equ_steady_H} at each cell center.
 For the velocity, we have two options 
 to map the analytical initial conditions to the mesh. 
 We either sample (i) the velocity field $\bf u$ of \eqref{equ_vel_calc} and (\ref{equ_steady_V}, left)
 at each triangle edge midpoint before we project it onto the edge's normal direction to obtain $V_{ij}$. 
 Alternatively, we sample (ii) the streamfunction $\Psi$ of (\ref{equ_steady_V}, right)
 at the triangles' vertices and calculate the normal velocities as 
 $V_{ij} = {\bf k} \times \operatorname{G}^\perp(\Psi)_{ij}$
 for ${\bf k} = (0,0,1)$, using the tangential gradient operator \eqref{curl_grad}. 
 Both options lead to very similar results, in particular when comparing the
 numerical solutions visually. However, a comparison in terms of $L_2$ and $L_\infty$ error norms 
 reveals that initialization (i) leads to slightly smaller error values, 
 in particular on coarse meshes (not shown). 
 Hence, in the following we only present results obtained using the velocities initialization 
 (\ref{equ_steady_V}, left).

 \bigskip
 
 \paragraph{Parameter choice and flow regimes.}
 We consider a set of dimensionless parameters to characterise the flow resulting from 
 \eqref{equ_steady_V} and \eqref{equ_steady_H}. The characteristic velocity is described by
 \begin{equation}
  U = 2\frac{g H'}{fd}\, ,
 \end{equation}
 with characteristic length scale $d = 4 r_0$  and $H'$ as maximal deviation of the surface 
 elevation from the background depth $H_0$. We consider further 
 the Rossby number $\rm Ro$, Froude number $\rm Fr$, and Burger number $\rm Bu$:
 \begin{equation}\label{equ_RandB}
  {\rm Ro} = \frac{U}{fd} = 2\frac{gH'}{f^2 d^2}, \quad {\rm Fr} = \frac{U}{\sqrt{gH_0}} ,  
  \quad {\rm Bu} = \frac{{\rm Ro}^2}{{\rm Fr}^2}= \frac{L_D^2}{d^2}  =  \frac{g H_0}{f^2 d^2}, 
 \end{equation}
 with Rossby deformation radius $L_D = \frac{\sqrt{g H_0}}{f}$.
 
 For this study, we want to consider fluids in geostrophic regime in which the flow is dominated by the geostrophic balance. 
 This requires ${\rm Ro} \ll 1$. The geostrophic regime can further be classified in: 
 (i) semi-geostrophic regime for ${\rm Bu} \ll 1$, (ii) quasi-geostrophic regime for ${\rm Bu} \approx 1$,
 and (iii) incompressible regime for ${\rm Bu} \gg 1$ (cf. \cite{Cullen2006,Pedlosky_GFL1979}, for instance).
 Because ${\rm Fr}$ describes the stratification of the fluid -- with strong stratification in case of small ${\rm Fr}$
 -- the choice of ${\rm Bu}$ allows us to describe shallow water flows with different degree of 
 compressibility: with (i) and (ii) for compressible and (iii) for almost divergence free flows.
 As suggested by \cite{RestelliHundermark2009} for the vortex pair interaction, 
 we fix $H'= 75\,$m which gives a Rossy number of ${\rm Ro} \approx 0.199$. Then, the choice of the 
 background depth $H_0$ allows us to model flows in the different geostrophic regimes:
 (i) $H_0 = 450\,$m for semi-geostrophic \textcolor{black}{($L_D = 1080\,$km)}, (ii) $H_0 = 750\,$m for quasi-geostrophic \textcolor{black}{($L_D = 1400\,$km)}, and (iii) $H_0 = 10\,$km for incompressible flows \textcolor{black}{($L_D = 5100\,$km)}.

 We study both the isolated vortex test case and the dual vortex interaction of Sect.~\ref{sec_tc_dual}
 in these three different flow regimes. In particular, we apply the same characteristic values as suggested by 
 \cite{RestelliHundermark2009}, which allows us to compare qualitatively as well as quantitatively the different 
 numerical solutions. Hence, we assume for both test cases the same characteristic length of $d = 4r_0 = 4 \sigma$ 
 with $\sigma = \frac{1}{2}(\sigma_x + \sigma_y)$ for 
 $\sigma_x = \frac{3}{40}L_x$ and $\sigma_y = \frac{3}{40}L_y$ (cf. \eqref{equ_dualvortexoffset}).
 As maximum velocity we use the characteristic velocity, i.e. we use $u_0 = U = 2\frac{g H'}{fd}$.
 Given this parameter choice, the isolated vortex is stable, cf. \cite{StDr2000}.

\textcolor{black}{To assure the convergence of the iterative solver (i.e. $C \leq 3$), we use a time step of $\Delta t = 48\,$s for the long-term simulations, 
 which yields the Courant number $C = 2.90$ for $\Delta x_{min} = 5.183\,$km of the irregular $2\cdot 64^2$ mesh 
 and $H_0= 10\,$km, the largest water depth studied.
 Because we use irregular meshes with resolutions up to $2 \cdot 256^2$ cells 
 with $\Delta x_{min} = 1.313\,$km, the time step for the convergence study with water depth up to $H_0= 10\,$km is 
 only $\Delta t = 12\,$s, which corresponds to a Courant number of $C = 2.86$. 
 }
 

 \medskip

 \paragraph{Results \textcolor{black}{of the long-term simulations}.} 
 Being a stationary solution of the rotating shallow water equations, we expect the 
 variational integrator to exactly preserve the initial distributions of fluid depth 
 $D$ and relative potential vorticity $q_{\rm rel}$ of the isolated vortex even for 
 long integration times. 
 Here and consistently throughout the manuscript, we illustrate the quantity $q_{\rm rel}$ 
 rather than, e.g., the absolute vorticity $\omega_a$ or the conserved potential vorticity $q$.
 This is because $q_{\rm rel}$ highlights the positive and negative regions of 
 the vorticity distribution and it allows us to compare further below our results with those obtained 
 by \cite{RestelliHundermark2009}.

 As it can be inferred from Fig.~\ref{fig_dyn_D_singlevortex} and \ref{fig_dyn_Z_singlevortex} 
 in which we compare solutions after 100 days of integration on a 
 regular (middle) and an irregular (right) mesh with the initial conditions (left),
 our variational scheme performs very well because it
 preserves in fact both fields for both mesh types very well 
 without generation spurious modes. In particular for the regular mesh, 
 the position, extend, and magnitude of the vortex at initial and end states very much agree 
 while we realize for the irregular case a slight oval shape of the initially round vortex 
 in both $D$ and $q_{\rm rel}$. This deviation is due to numerical errors that are 
 caused by strongly deformed mesh cells, which is particularly apparent on coarse mesh 
 resolution as here for a grid with only $2\cdot 64^2$ cells.
 However, as it can be inferred from the convergence study below 
 (cf. Fig.~\ref{fig_isolated_vtx_convergence}), this deviation reduces with, at least, 
 $1^{st}$-order with increasing mesh resolution.
 
 \textcolor{black}{Despite the isolated vortex is a steady state solution of the RWS equations, 
 in which any quantity is conserved because of no time dependence, for the numerical
 scheme such solutions are stationary only up to numerical errors. 
 As such, it is interesting to discuss also here the conservation properties of the QOI.} 
 In Fig.~\ref{fig_Z_single_Hall_qoi} we show the time evolutions of the relative errors
 (determined as ratio of current values at time $t$
 over initial value at $t = 0$) of the total energy $E$, determined on a 
 mesh with $2\cdot 64^2$ cells (upper row) and with $2\cdot 32^2$ cells (lower row), 
 for fluids in semi-geostrophic (first and second column), in quasi-geostrophic (third and forth column),
 and in incompressible (fifth and sixth column) regimes. 
 As above, we compare results for 
 regular (first, third, fifth column) with irregular (second, forth, sixth column) meshes.
 Here, and for all other cases studied, mass $m$ is preserved up 
 to machine precision (not shown). 
 
 Comparing the three different flow regimes, we note that the errors in total energy 
 depend on the fluid depth, with smallest values of $10^{-10}$ in the incompressible case. 
 In the other regimes that allow compressibility, total energy is  also very well preserved
 at the order of $10^{-8}$. 
 \textcolor{black}{All energy error values relate to the time step $\Delta t = 48\,$s and reduce further at $1^{\rm st}$-order convergence with smaller $\Delta t$.}
 But they are more or less independent from the spatial resolution (cf. Fig.~\ref{fig_Z_single_Hall_qoi}).
 The indicated marginal trends of loss in total energy, visible in the energy plots of the 
 irregular mesh cases, diminish with higher spatial resolution for all three flow regimes
 (compare lower and upper rows of Fig.~\ref{fig_Z_single_Hall_qoi}). 

 \textcolor{black}{The relative errors in $PE$ show a similar dependency on the flow regime as the 
 energy errors (hence not shown). Compared to the values presented in Fig.~\ref{fig_Z_single_Hall_qoi} for $\Delta t = 48\,$s, 
 the $PE$ error values are about two orders of magnitude larger while they are more or less independent from the time
 step size. $PV$ is conserved at machine precision for all cases studied (analogously to the time series presented further below for the 
 nonstationary cases).
 }
 
%
 \begin{figure}[t!]\centering
 \begin{tabular}{ccc}
  \hspace*{-0.3cm}{\includegraphics[scale=0.38]{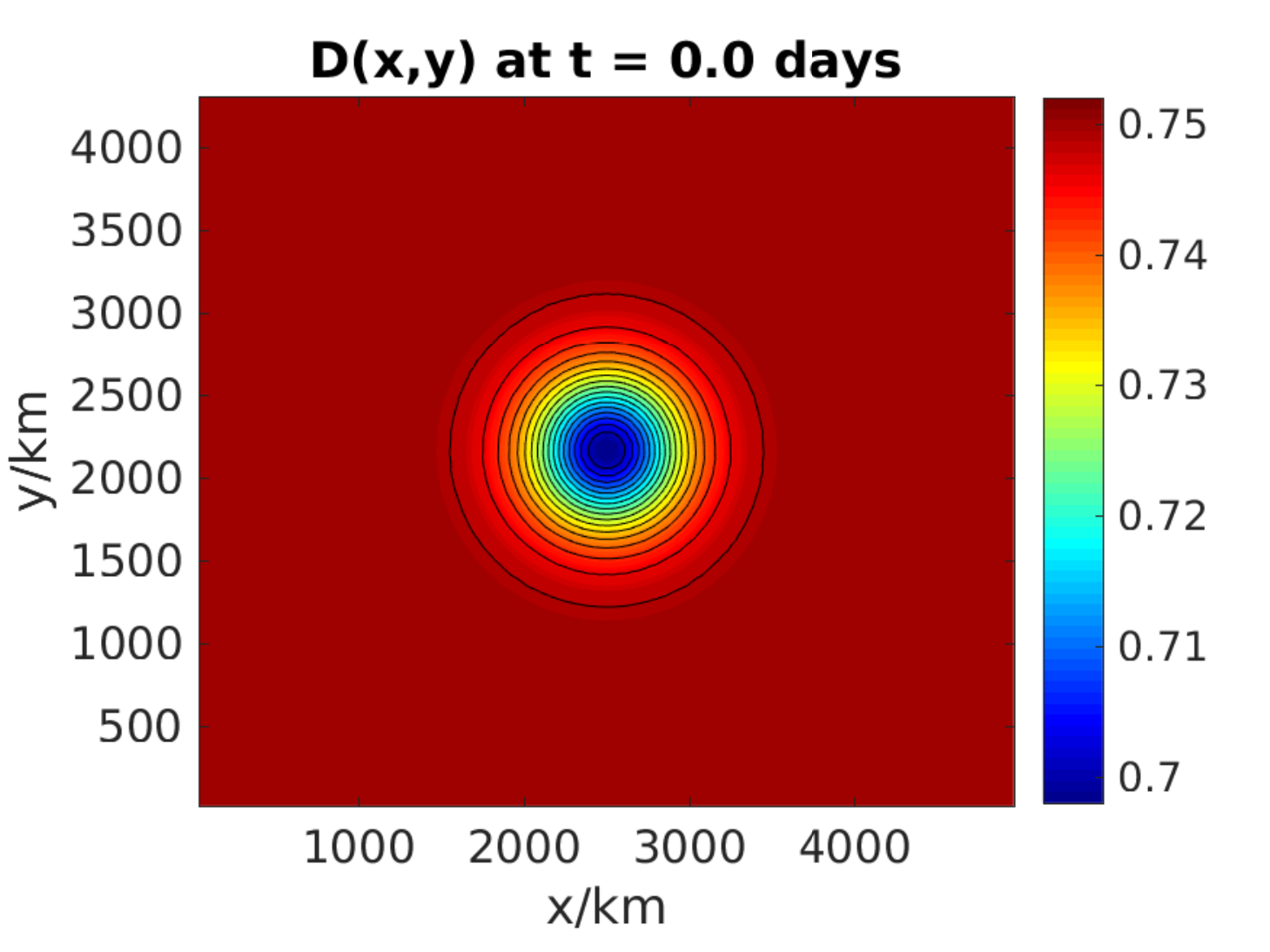}} &
  \hspace*{-0.5cm}{\includegraphics[scale=0.38]{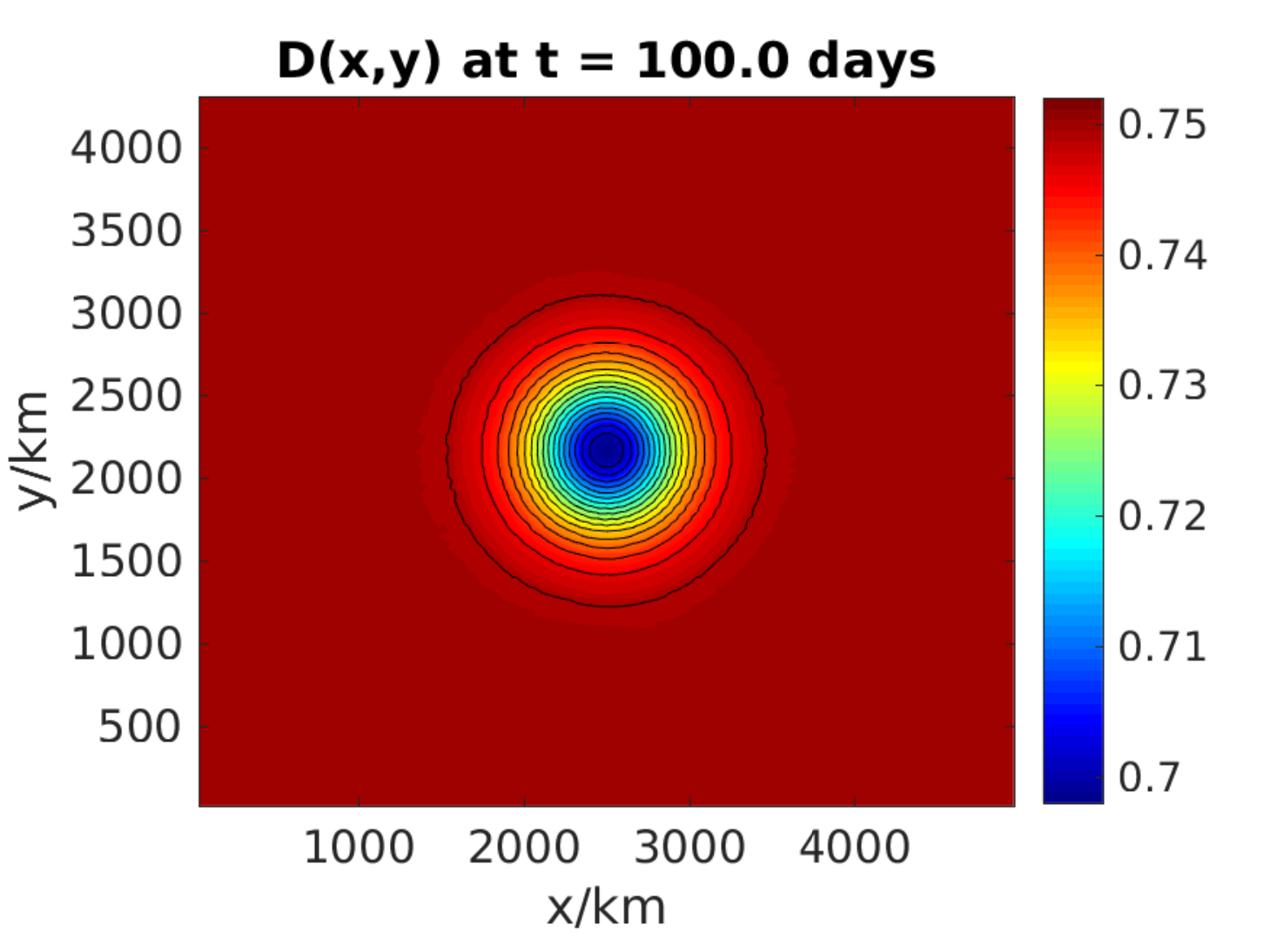}} & 
  \hspace*{-0.5cm}{\includegraphics[scale=0.38]{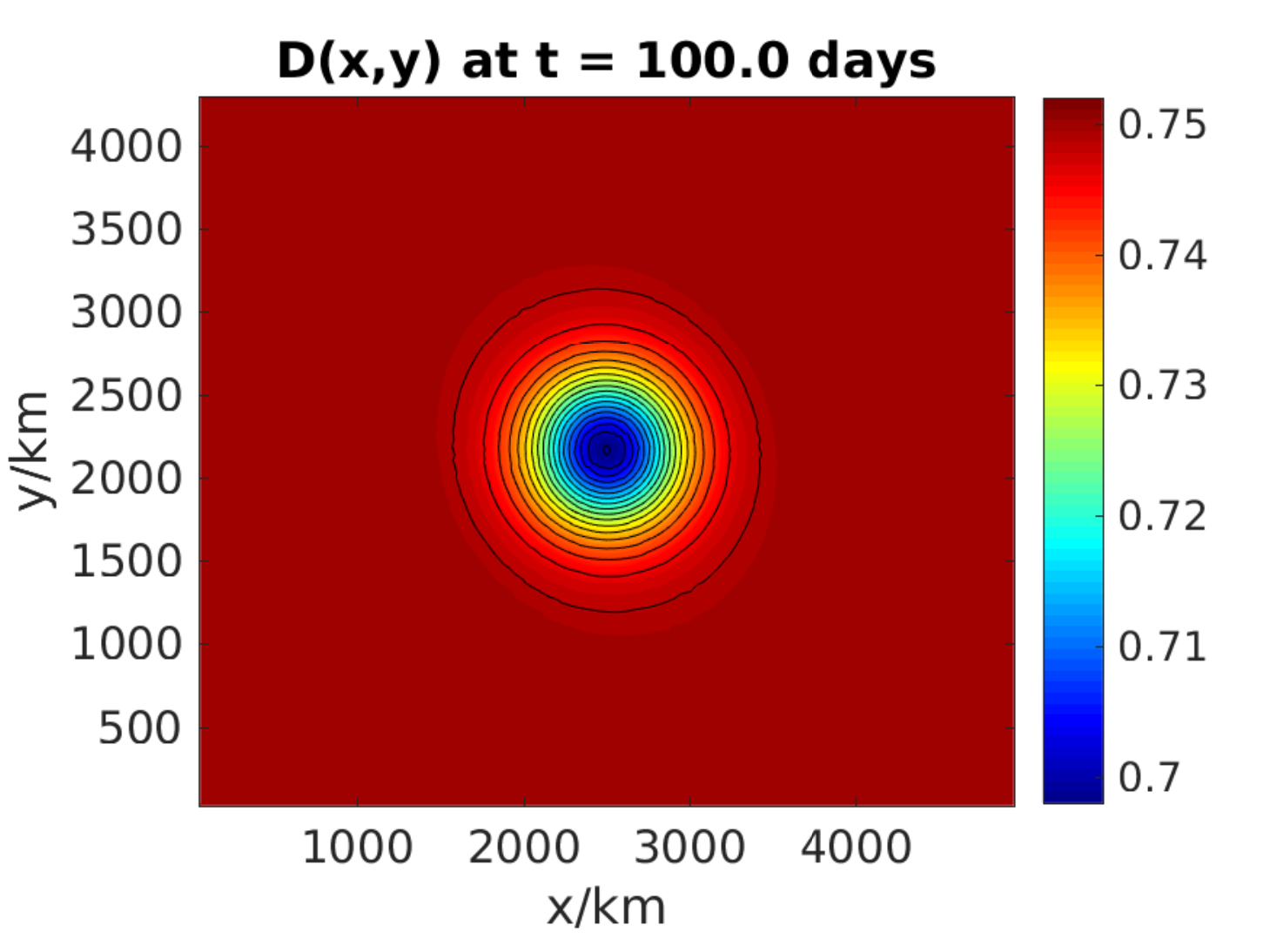}} 
 \end{tabular}
 \caption{Isolated vortex test case: fluid depth $D(x,y)$ at initial time $t=0$ (left) and 
 at $t=100\,$days on a regular (center) and an irregular (right) mesh with $2 \cdot 64^2$ triangular 
 cells. Contours between $0.698\,{\rm km}$ and $0.752\,{\rm km}$ with interval of 
 $0.003\,{\rm km}$.
 }                                                                                             
 \label{fig_dyn_D_singlevortex}
 \end{figure}

 \begin{figure}[t!]\centering
 \begin{tabular}{ccc}
  \hspace*{-0.3cm}{\includegraphics[scale=0.38]{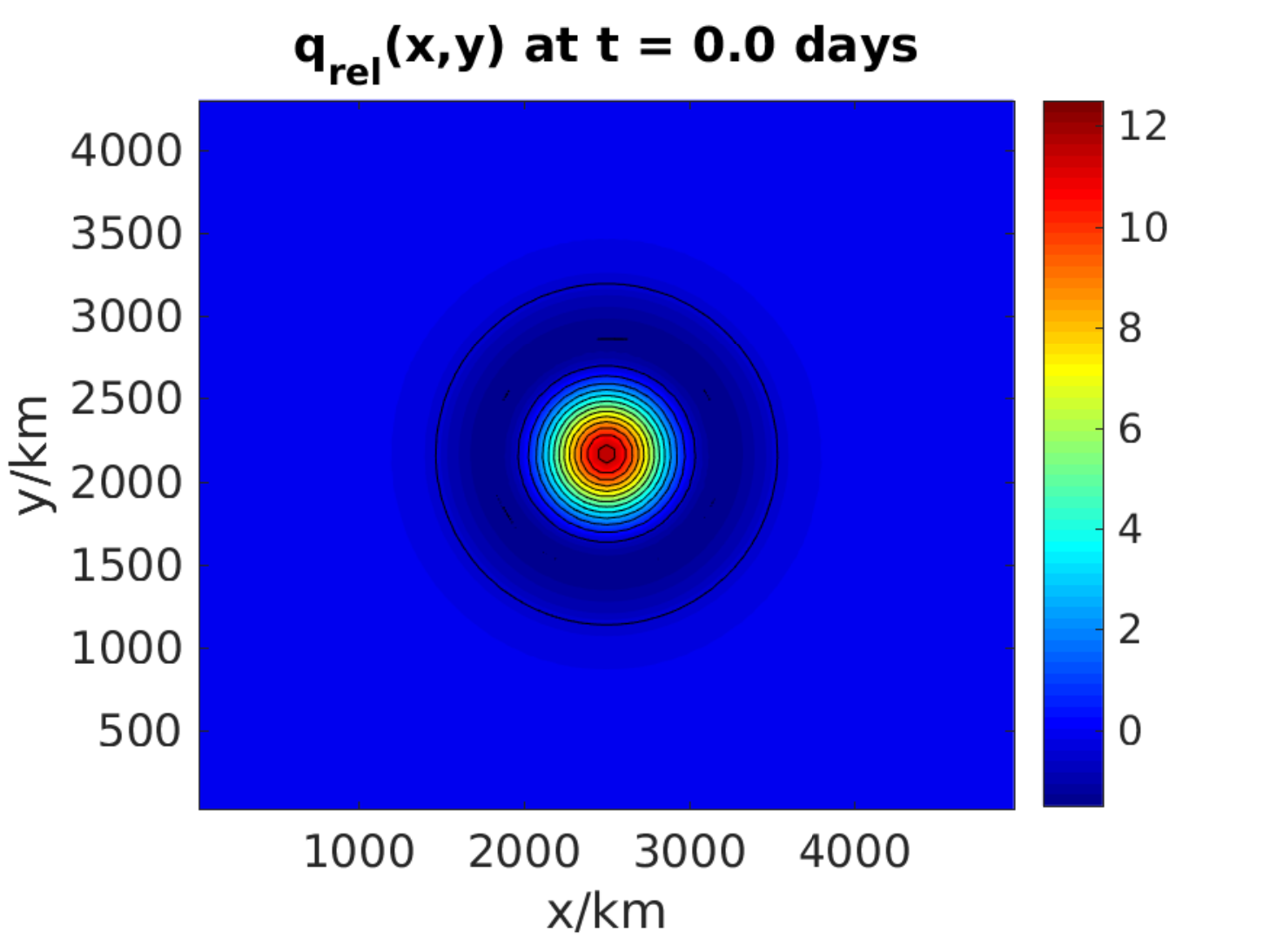}} &
  \hspace*{-0.5cm}{\includegraphics[scale=0.38]{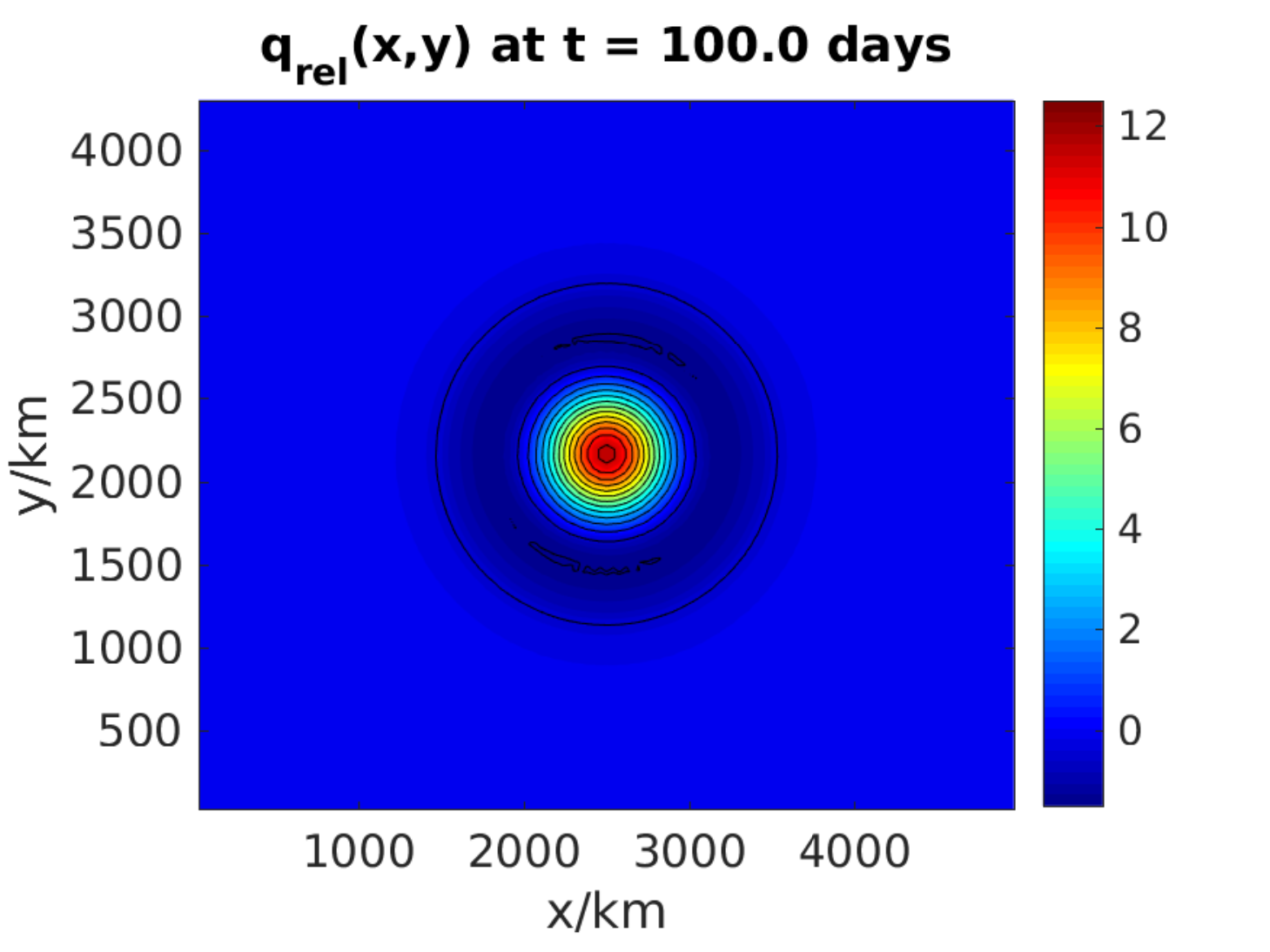}} &
  \hspace*{-0.5cm}{\includegraphics[scale=0.38]{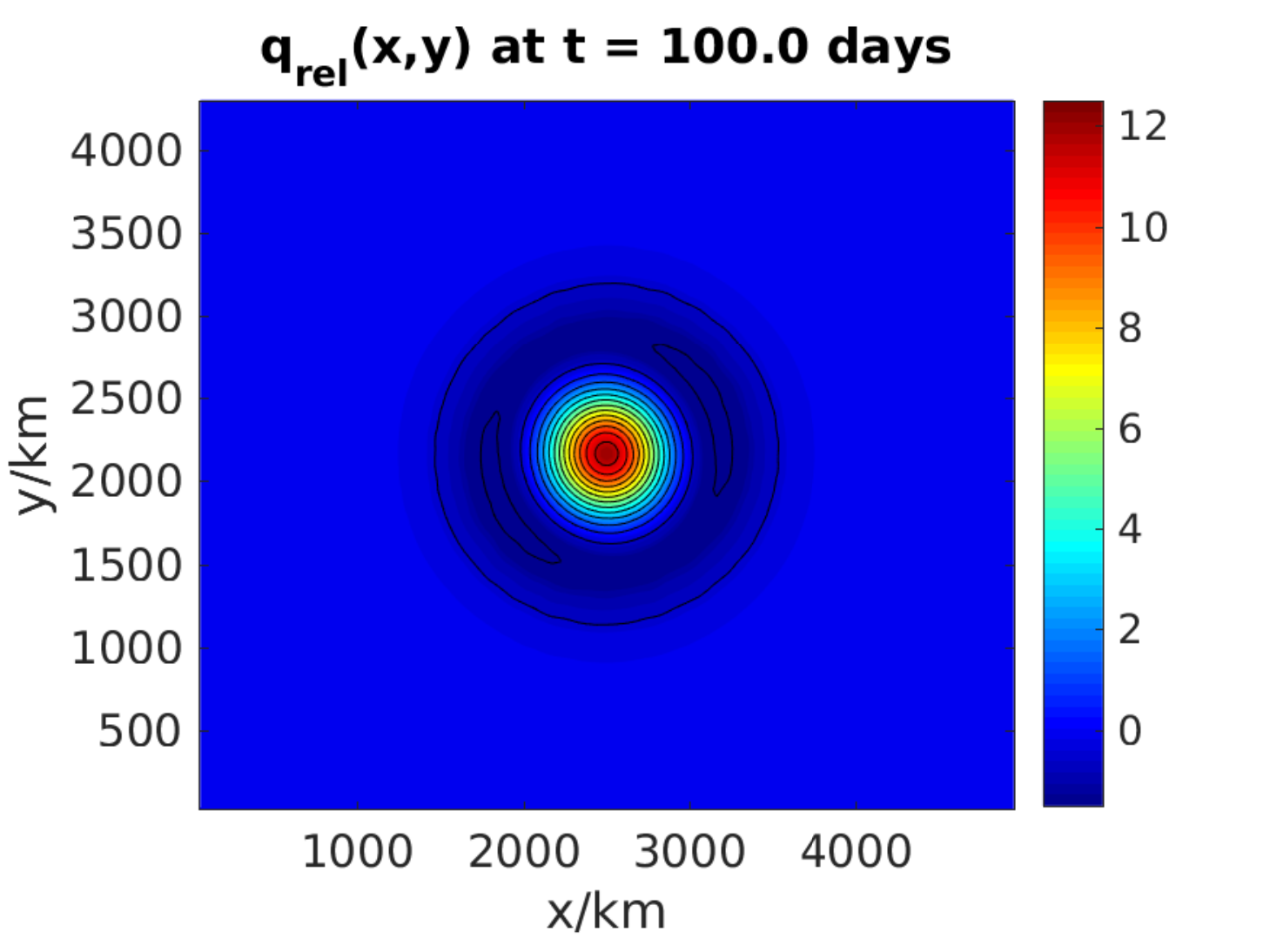}} 
 \end{tabular}
 \caption{Isolated vortex test case: relative potential vorticity $q_{\rm rel}(x,y)$ at initial time $t=0$ (left) and 
 at $t=100\,$days on a regular (center) and an irregular (right) mesh with $2 \cdot 64^2$ triangular cells.
 Contours between $-1.5\,{\rm days^{-1}km^{-1}}$ and $12.5\,{\rm days^{-1}km^{-1}}$ with interval of 
 $1\,{\rm days^{-1}km^{-1}}$.
 }                                                                                             
 \label{fig_dyn_Z_singlevortex}
 \end{figure}


 \begin{figure}[t!]\centering
 \begin{tabular}{cccccc}  
  \hspace*{-.075cm}{\includegraphics[scale=0.38]{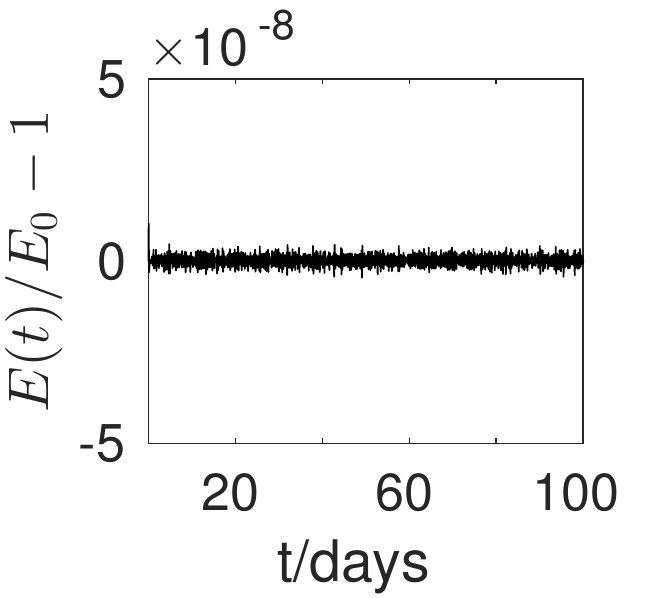}}  & 
  \hspace*{-.075cm}{\includegraphics[scale=0.38]{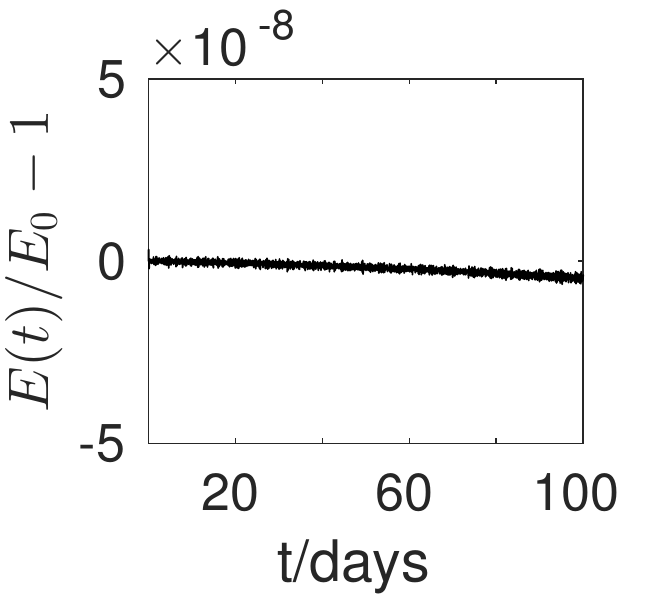}}  &
  \hspace*{-.075cm}{\includegraphics[scale=0.38]{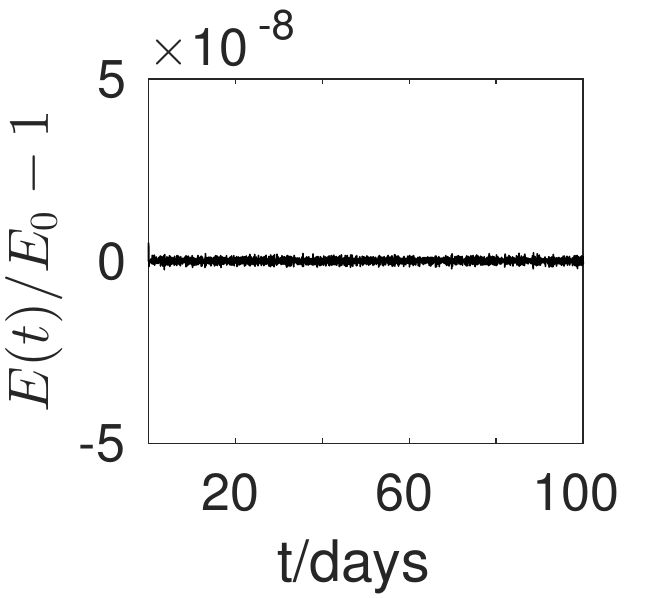}}  & 
  \hspace*{-.075cm}{\includegraphics[scale=0.38]{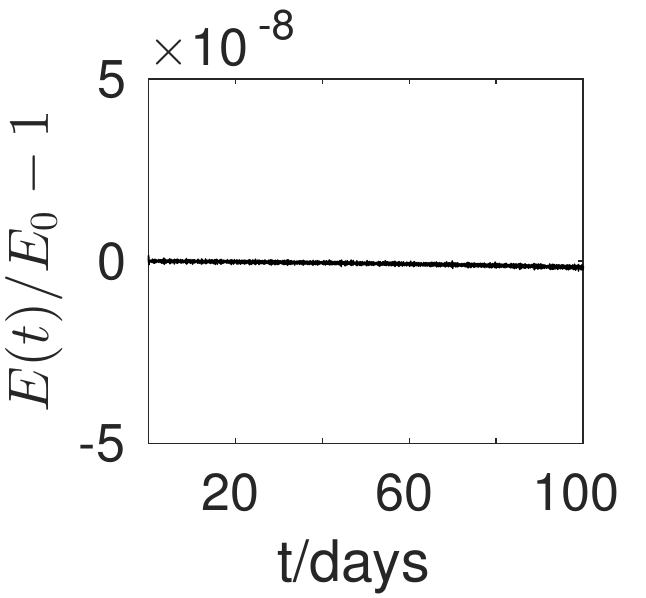}}  &
  \hspace*{-.075cm}{\includegraphics[scale=0.38]{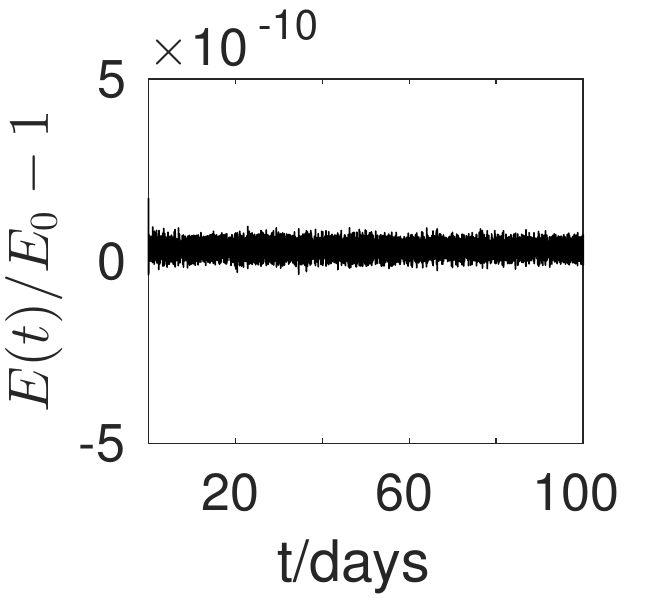}}  & 
  \hspace*{-.075cm}{\includegraphics[scale=0.38]{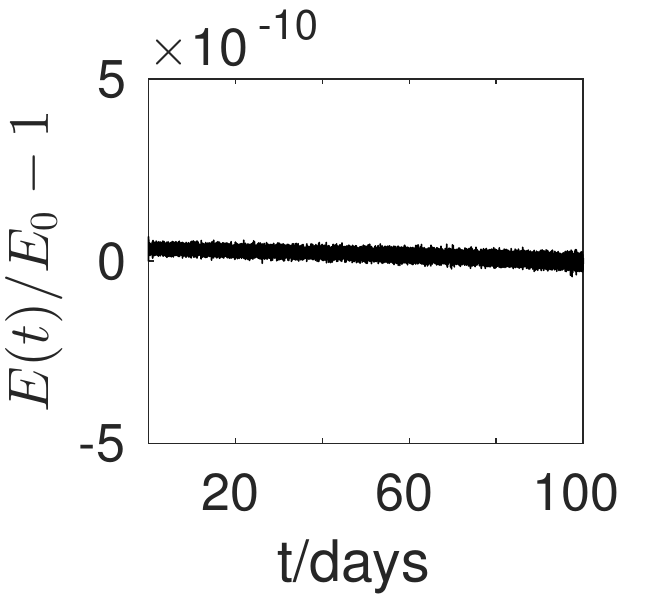}}  \\
  \hspace*{-.075cm}{\includegraphics[scale=0.38]{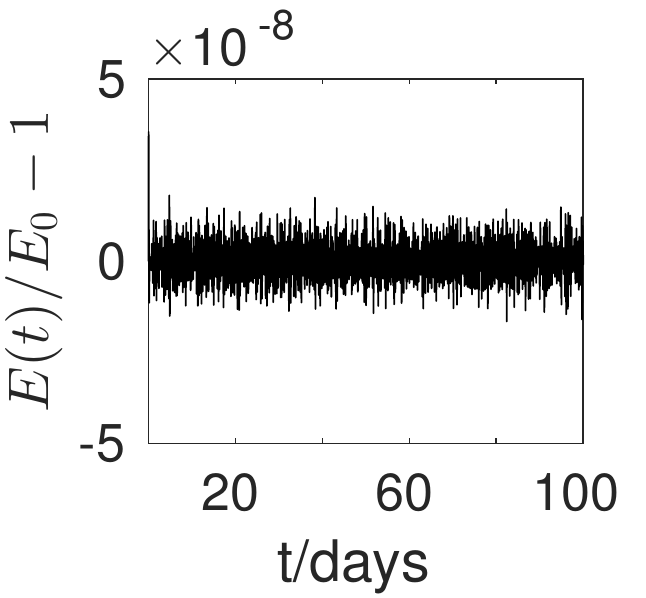}}  & 
  \hspace*{-.075cm}{\includegraphics[scale=0.38]{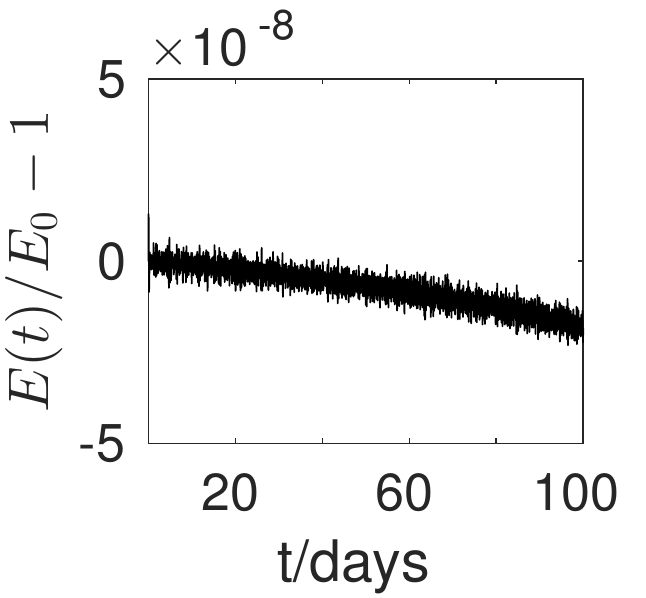}}  &
  \hspace*{-.075cm}{\includegraphics[scale=0.38]{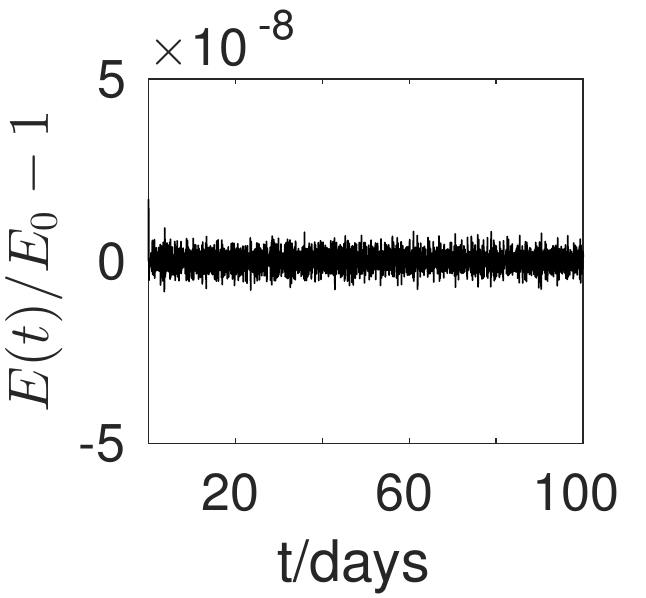}}  & 
  \hspace*{-.075cm}{\includegraphics[scale=0.38]{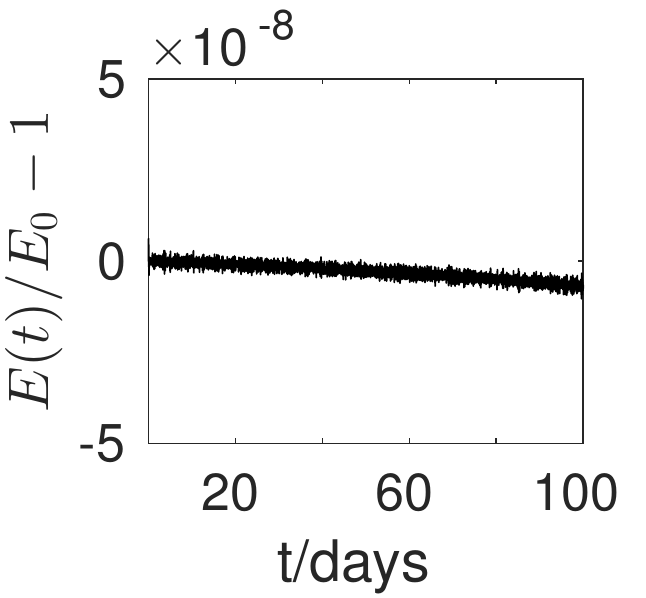}}  &
  \hspace*{-.075cm}{\includegraphics[scale=0.38]{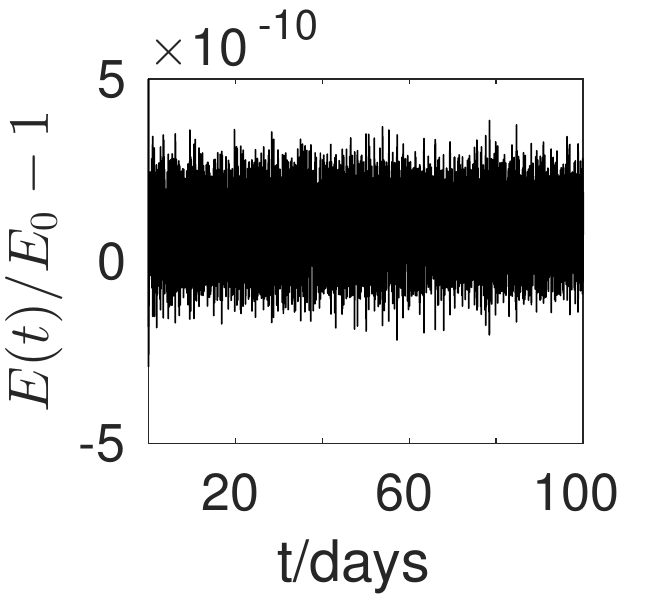}}  & 
  \hspace*{-.075cm}{\includegraphics[scale=0.38]{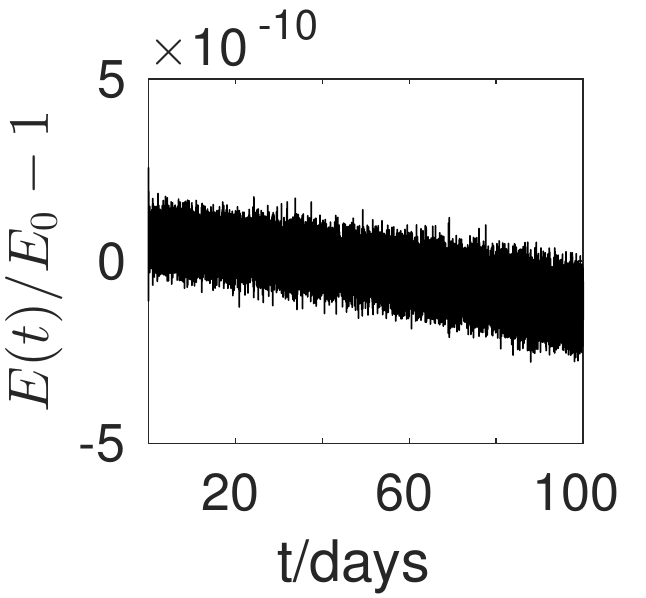}}  
 \end{tabular}
  \caption{Isolated vortex test case: relative errors of total energy $E(t)$ 
  on meshes with $2\cdot 64^2$ cells (upper row) and with $2\cdot 32^2$ cells (lower row) for a fluid in semi-geostrophic ($1^{st},2^{nd}$ column), 
  in quasi-geostrophic ($3^{rd},4^{th}$ column), and in incompressible ($5^{th},6^{th}$ column) regime
  for regular ($1^{st},3^{rd},5^{th}$ column) and irregular ($2^{nd},4^{th},6^{th}$ column) meshes. 
  }                                                                                             
  \label{fig_Z_single_Hall_qoi}
 \end{figure}

 \medskip

 \paragraph{Results of convergence study.}
 We compute $L_2$ and $L_\infty$ error norms of the numerical solutions for fluid depth $D$ and 
 relative potential vorticity $q_{\rm rel}$ to study the spatial convergence behavior 
 of solutions of the variational RSW integrator.  
 For the steady state case, any deviation of the numerical solutions from the initial 
 fields is considered as numerical error. Hence, we define the $L_2$ and $L_\infty$ error measures 
 for a discrete function $f_i(t)$ with respect to its initial values $f_i(0)$ over all triangles $T_i$ by
 \begin{equation}
  L_2 [f]      = \frac{\sqrt{ \sum_i (f_i(t)\Omega_{ii} - f_i(0)\Omega_{ii})^2} }{\sqrt{\sum_i (f_i(0)\Omega_{ii})^2}}, \qquad
  L_\infty [f] = \frac{\max | f_i(t)\Omega_{ii} - f_i(0)\Omega_{ii} |  }{\max | f_i(0)\Omega_{ii}| } .
 \end{equation}
 \textcolor{black}{These errors are determined on regular and irregular meshes with resolutions of
 $2 \cdot 32^2$, $2 \cdot 64^2$, $2 \cdot 128^2$, and $2 \cdot 256^2$ triangles. 
 Moreover, we use for all simulations one fixed time step of $\Delta t = 12\,$s to consider only the spatial convergence behavior.}
  
 Fig.~\ref{fig_isolated_vtx_convergence} illustrates the 
 error values of the numerical solutions for $D$ and $q_{\rm rel}$ 
 after $1\,$day with respect to the corresponding initial states. 
 As $q_{\rm rel}$ differs from $q$ only by $f/h$, which is almost constant here, 
 the convergence rates for $q$ are very similar to those of $q_{\rm rel}$ (hence not shown). 
 Also not shown are error values for the absolute vorticity $\omega_a$ 
 which are, too, very similar to those of $q_{\rm rel}$.
 Using numerical solutions for later times provided qualitatively the 
 same results, only the absolute error values would be larger.
 We compare a fluid in semi-geostrophic (left), quasi-geostrophic (middle),
 and incompressible (right) regimes. Similarly to the conservation properties of the QOI, 
 the absolute error values are the smallest in the incompressible regime while we realize 
 a slightly increase of the errors for semi-geostrophic and quasi-geostrophic flows. 
 Independently from the regime, all solutions for $D$ and $q_{\rm rel}$ show in both error norms 
 convergence rates between $2^{nd}$- and $1^{st}$-order. 
 
 Considering the absolute error values for $D$, one realizes that for all regimes both 
 $L_2$ and $L_\infty$ errors on irregular meshes are significantly smaller than on regular meshes with 
 same resolution $N$, i.e. with the same number of cells. In particular for $L_2[D]$, 
 the error values on an irregular mesh with $N = 2 ({\frac{1}{2}N_{\rm 1D}})^2$ cells are close to those with 
 $N = 2 (N_{\rm 1D})^2$ cells on a regular mesh. This agrees well with the fact that on irregular meshes the central region has 
 cells with halved triangle edge length providing effectively a doubled resolution.
 
 In contrast to the improvements for $D$, the error values for $q_{\rm rel}$ are higher on irregular 
 meshes. 
 Here, the irregular cells might trigger numerical noise in form of additional 
 small scale eddies earlier than on regular meshes leading to the increased error values.
 Nevertheless, the numerical solutions for $q_{\rm rel}$ 
 converge also here with, at least, $1^{st}$-order to the correct solution.

 \begin{figure}[t!]\centering
 \begin{tabular}{c}
  \hspace*{-1.6cm}{\includegraphics[scale=0.35]{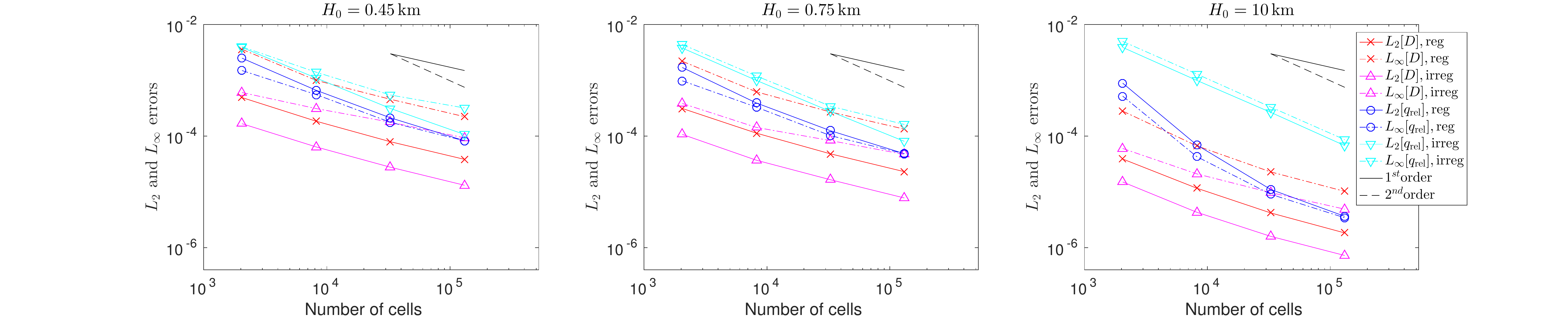}} 
 \end{tabular}
 \caption{Isolated vortex test case: $L_2$ and $L_\infty$ error values of numerical solutions for $D$ and $q_{\rm rel}$ 
 after 1 day as a function of grid resolution for a fluid in semi-geostrophic (left), quasi-geostrophic (middle), 
 and incompressible (right) regime for regular and irregular meshes. 
 }                                                                                             
 \label{fig_isolated_vtx_convergence}
 \end{figure}

 \color{black}

 \subsection{Nonlinear dynamics}
  
 Let us focus next on flows that are dominated by nonlinear processes. 
 By means of two test cases we study whether our variational integrator
 is able to correctly represent the general dynamical behavior while
 conserving the quantities of interest discussed above. 
 
 In the first test case we study the evolution of two interacting corotating vortices in different 
 regimes, i.e. we study semi-geostrophic, quasi-geostrophic, and incompressible 
 flows. This allows us to examine how accurate the scheme represents flows that are 
 in advection and/or divergence dominated regimes. In the second test case, 
 the time evolution of a shear flow in quasi-geostropic regime is studied. 
 We examine the general flow pattern, such as position and magnitude of the vortex 
 cores in the first test case or the growth rate of the instability in the second test case, and compare 
 the solutions with literature, in particular with \cite{BauerPHD2013,RestelliHundermark2009,Reich2006}. 
 As above, we apply also here regular and irregular computational meshes for the simulations.

 \subsubsection{Vortex pair interaction}
 \label{sec_tc_dual}
  
 In this test case, the flow evolution of two interacting corotating vortices in the inviscid 
 case is studied. This vortex pair problem is described in 
 \cite{RestelliHundermark2009,Reich2006,Reich2007,Staniford2007} for instance. Here, we give a brief
 description of the test case and its initialization according to \cite{RestelliHundermark2009}.

 \bigskip  
 \paragraph{Initialization.} 
 We choose the initial conditions in geostrophic equilibrium 
 by prescribing the fluid depth $h$ by an analytic solution while determining the 
 velocity by the constraint of geostrophic balance, i.e. $f {\bf k} \times \mathbf{u} =- g \nabla h$ with 
 ${\bf k} = (0,0,1)^\mathsf{T}$, cf. \cite{RestelliHundermark2009} for more details.
 We use the initial fields
 \begin{equation}\label{equ_init_vortex_pair}
  h(x,y,0) =  H_0 - H' \left[   e^{-\frac{1}{2}({x_1'}^2 + {y_1'}^2  )} + e^{-\frac{1}{2}({x_2'}^2 + {y_2'}^2  )} - \frac{4\pi \sigma_x \sigma_y}{L_x L_y}     \right] \ , 
 \end{equation}
 \begin{equation}\label{equ_init_vortex_pair_velocity}
 \begin{split}
   u(x,y,0) =   - \frac{g H'}{f \sigma_y}  \left[ y_1^{''}  e^{-\frac{1}{2}({x_1'}^2 + {y_1'}^2  )} + y_2^{''} e^{-\frac{1}{2}({x_2'}^2 + {y_2'}^2  )}     \right] \ , \\
   v(x,y,0) =   + \frac{g H'}{f \sigma_x}  \left[ x_1^{''}  e^{-\frac{1}{2}({x_1'}^2 + {y_1'}^2  )} + x_2^{''} e^{-\frac{1}{2}({x_2'}^2 + {y_2'}^2  )}     \right] \ ,
 \end{split}
 \end{equation}
 where we apply for $i=1,2$ the periodic extensions
 \begin{equation}\label{equ_init_single_vortex_defn}
 \begin{split}
  x'_{i} = \frac{L_x}{\pi \sigma_x}\sin \left( \frac{\pi}{L_x}(x - x_{c_{i}}) \right)\, , \quad 
     y'_{i} = \frac{L_y}{\pi \sigma_y}\sin \left( \frac{\pi}{L_y}(y - y_{c_{i}}) \right)\, , \\
  x^{''}_{i} = \frac{L_x}{2\pi \sigma_x}\sin \left( \frac{2\pi}{L_x}(x - x_{c_{i}}) \right)\, , \quad 
     y^{''}_{i} = \frac{L_y}{2\pi \sigma_y}\sin \left( \frac{2\pi}{L_y}(y - y_{c_{i}}) \right)\, . \\
 \end{split}
 \end{equation}
 The centers of the vortices and $\sigma_x,\sigma_y$ are given by
 \begin{equation}\label{equ_dualvortexoffset}
  \begin{split}
  &  x_{c_{1}} = \left(\frac{1}{2} - o\right)L_x \, , \quad  x_{c_{2}} = \left(\frac{1}{2} + o\right)L_x \, , \quad   \sigma_x = \frac{3}{40}L_x  \, ,     \\
  &  y_{c_{1}} = \left(\frac{1}{2} - o\right)L_y \, ,  \quad  y_{c_{2}} = \left(\frac{1}{2} + o\right)L_y \, , \quad \sigma_y = \frac{3}{40}L_y  \, ,
  \end{split} 
 \end{equation}
 using $o = 0.1$ and $H' = 75\,$m.

As in Sect.~\ref{sec_tc_single}, we map the analytical function \eqref{equ_init_vortex_pair} onto 
 the mesh by sampling it at each cell center to obtain $D_i$. Considering 
 $h$ as stream function, we initialize the normal velocity values at the cell faces by using the discrete 
 tangential gradient operator~\eqref{curl_grad}, i.e. 
 $V_{ij} = -\frac{g}{f} \operatorname{G}^\perp(h)_{ij}$, rather than initializing the velocity directly 
 via \eqref{equ_init_vortex_pair_velocity}. This approach leads to discrete
 fields for fluid depth and velocity that are in geostrophic balance.
 The initial fields are shown in Fig.~\ref{fig_tc2_dual_DandVorticity}. Again, 
 we examine the variational scheme with respect to the three flow regimes:
 (i) $H_0 = 450\,$m for semi-geostrophic, (ii) $H_0 = 750\,$m for quasi-geostropic,
 and (iii) $H_0 = 10\,$km for incompressible flows. 
To allow for comparisons to \cite{RestelliHundermark2009}, we apply 
 also here a test case with trivial bottom topography.

 \textcolor{black}{For all dual vortex simulations, we apply a time step of $\Delta t = 12\,\rm s$ 
 and use regular and irregular meshes with $2\cdot 256^2$ cells giving a minimal edge length of 
 $\Delta x_{min} = 1.313\,$km. Similarly to the convergence test, we obtain hence
 for the largest water depth of $H_0= 10\,$km a Courant number $C = 2.86$. 
 }

 \begin{figure}[t!]\centering
 \begin{tabular}{cc}
  \hspace*{-0.0cm}{\includegraphics[scale=0.5]{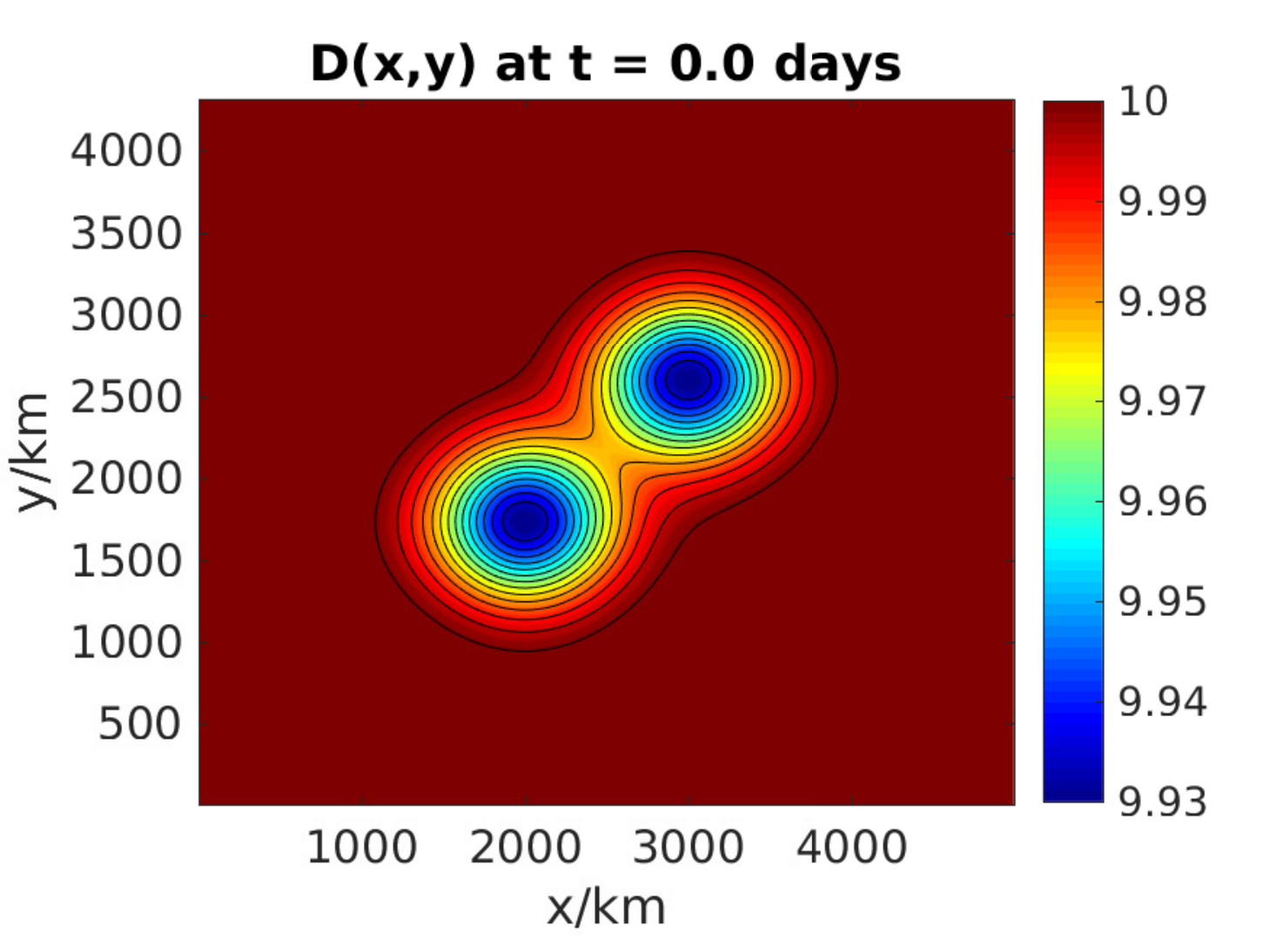}} &  
  \hspace*{-0.0cm}{\includegraphics[scale=0.5]{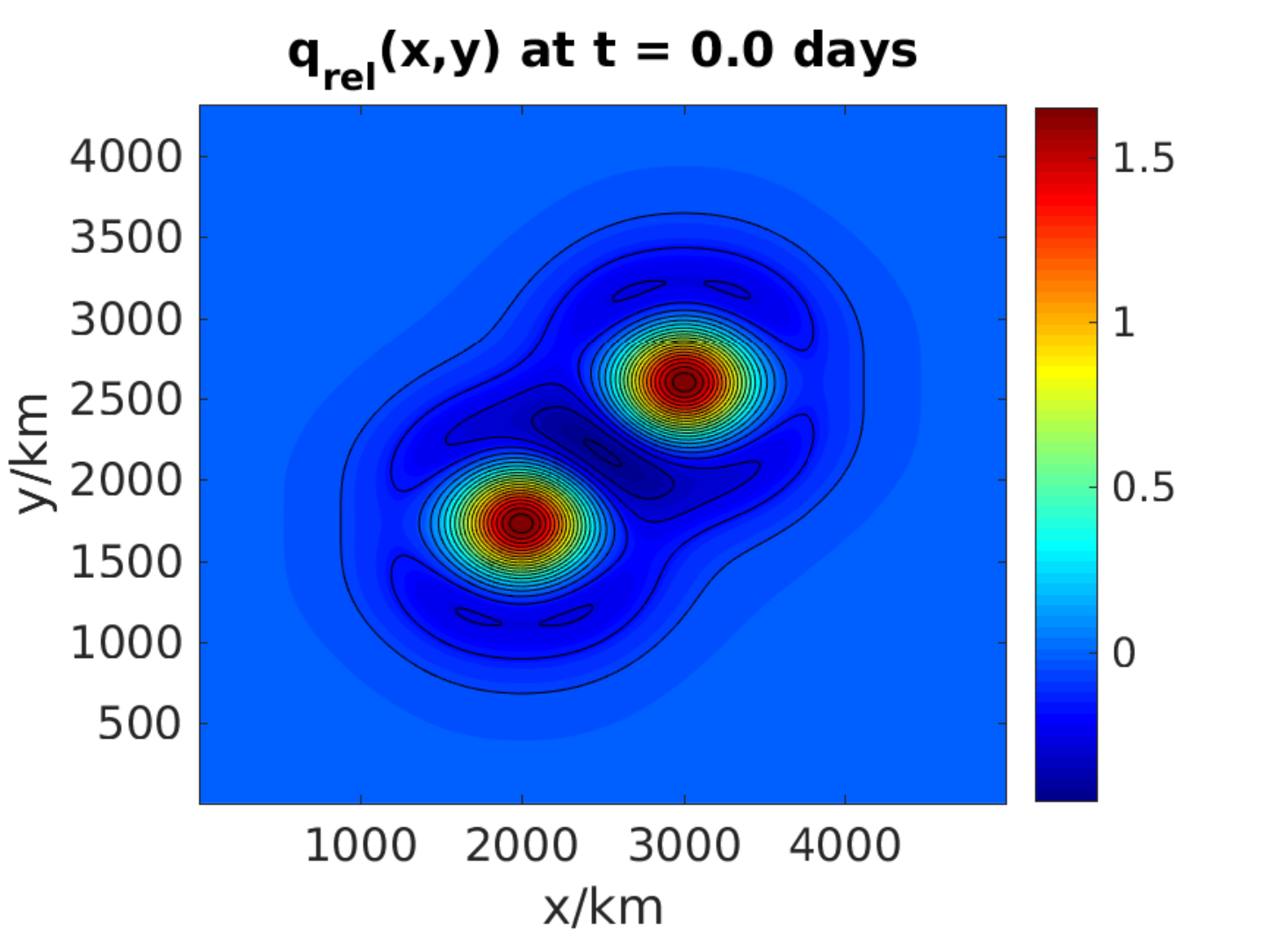}} 
 \end{tabular}
  \caption{Initial fluid depth $D$ and relative potential vorticity $q_{\rm rel}$
  in geostrophic balance. 
  Contours for $D$ between $9.93\,{\rm km}$ and $10\,{\rm km}$ with interval of 
 $0.005\,{\rm km}$ and for $q_{\rm rel}$ between 
 $-0.45\,{\rm days^{-1}km^{-1}}$ and $1.7\,{\rm days^{-1}km^{-1}}$ with interval of 
 $0.1\,{\rm days^{-1}km^{-1}}$.
 }                                                                                             
  \label{fig_tc2_dual_DandVorticity}
 \end{figure}

 \medskip

 \paragraph{Results.} 
 Let us first discuss the flow evolution of the two interacting corotating vortices. 
 For the chosen initial distance between the vortex cores exhibiting an area of negative 
 vorticity in between them (cf. Fig.~\ref{fig_tc2_dual_DandVorticity}), the vortex cores 
 are too far apart to allow for a merger (see e.g. \cite{Baueretal2014} for more details). 
 Instead, the two cores are mutually repelled due to nonlinear effects. 
 The corresponding time evolution of the relative potential vorticity field is shown in 
 Fig.~\ref{fig_Z_dual_vortices} for the incompressible case, but the flow evolves very 
 similar for all flow regimes studied (cf. Fig.~\ref{fig_dual_qoi}).
 
 Comparing in Fig.~\ref{fig_Z_dual_vortices} simulations performed on either a regular 
 (upper row) or an irregular mesh (lower row), we notice that both relative potential vorticity fields agree very 
 well, in particular the speed of the mutual repulsion of the two vortex cores and the thin 
 filaments between them. Also both fields of fluid depth agree very well (not shown). 
 This very good match between the corresponding fields is also given in case of 
 semi-geostrophic and quasi-geostrophic flows (also not shown).

 The snapshots of relative potential vorticity and fluid depth at day 10, presented in 
 Fig.~\ref{fig_Z_dual_Hall_vortices} for regular meshes, allows us further to compare our 
 simulations with those performed in \cite{BauerPHD2013,RestelliHundermark2009}. For both mesh 
 types and for all flow regimes, i.e. semi-geostrophic (right column), quasi-geostrophic 
 (middle column), and incompressible (right column), we obtain solutions that are very close to 
 those determined in \cite{RestelliHundermark2009} with a conventional triangular C-grid discretization 
 of the shallow water equations. In particular the fields agree very well when considering the 
 magnitude of $q_{\rm rel}$ and $D$ fields and the position of the two vortex cores. 
 This good match is also given when comparing our results with those in \cite{BauerPHD2013} 
 obtained by a corresponding hexagonal C-grid scheme. 
 
 The relative errors of the QOI are shown in Fig.~\ref{fig_dual_qoi}. 
 The first and second 
 column correspond to the semi-geostrophic case, the third and forth to the quasi-geostrophic case, 
 and the fifth and sixth to the incompressible case. The relative errors in total energy (upper row) 
 for flows in semi- and quasi-geostrophic regimes are at the order of $10^{-7}$ for both mesh types. 
 In the incompressible case, these errors are about two orders of magnitudes smaller, just as in the 
 isolated vortex test case. 
 \textcolor{black}{They are related to the time step $\Delta t = 12\,$s} and decrease further at $1^{st}$-order when using smaller time step sizes. 
 The error values are very much independent from the spatial resolution and they 
 exhibit the expected oscillatory behavior of a symplectic time integrator.

 Here, and for all simulations performed, the mass-weighted potential vorticity $PV$ is conserved 
 at the order of machine precision, just as mass $m$, for both regular and irregular meshes
 and for all temporal and spatial resolutions studied. 
 Considering the relative errors of $PE$ on regular meshes, they show a similar dependency on the flow 
 regime as $E$, with an accuracy at the order of $10^{-4}$ in the semi- and quasi-geostrophic 
 regimes and one order of magnitude smaller in the incompressible case. 
 Our RSW scheme conserves these values well, although the 
 variational discretization method does not treat neither $PV$ nor $PE$ as discrete Casimirs.
 Hence we do not expect them to be strictly conserved.
 In fact, here for the vorticity dominated vortex interaction test case, we notice a growth of $PE$ on irregular meshes 
 at the order of $10^{-3}$ for a simulation of 10 days. This growth rate is however
 rather small and in the shear flow test case studied below, even much smaller. 
 Finally we point out that also the error values for $PE$ are more or less independent 
 from spatial {\em and} temporal resolutions (not shown).

 \begin{figure}[t!]\centering
 \begin{tabular}{cccc}
  \hspace*{-0.25cm}{\includegraphics[scale=0.30]{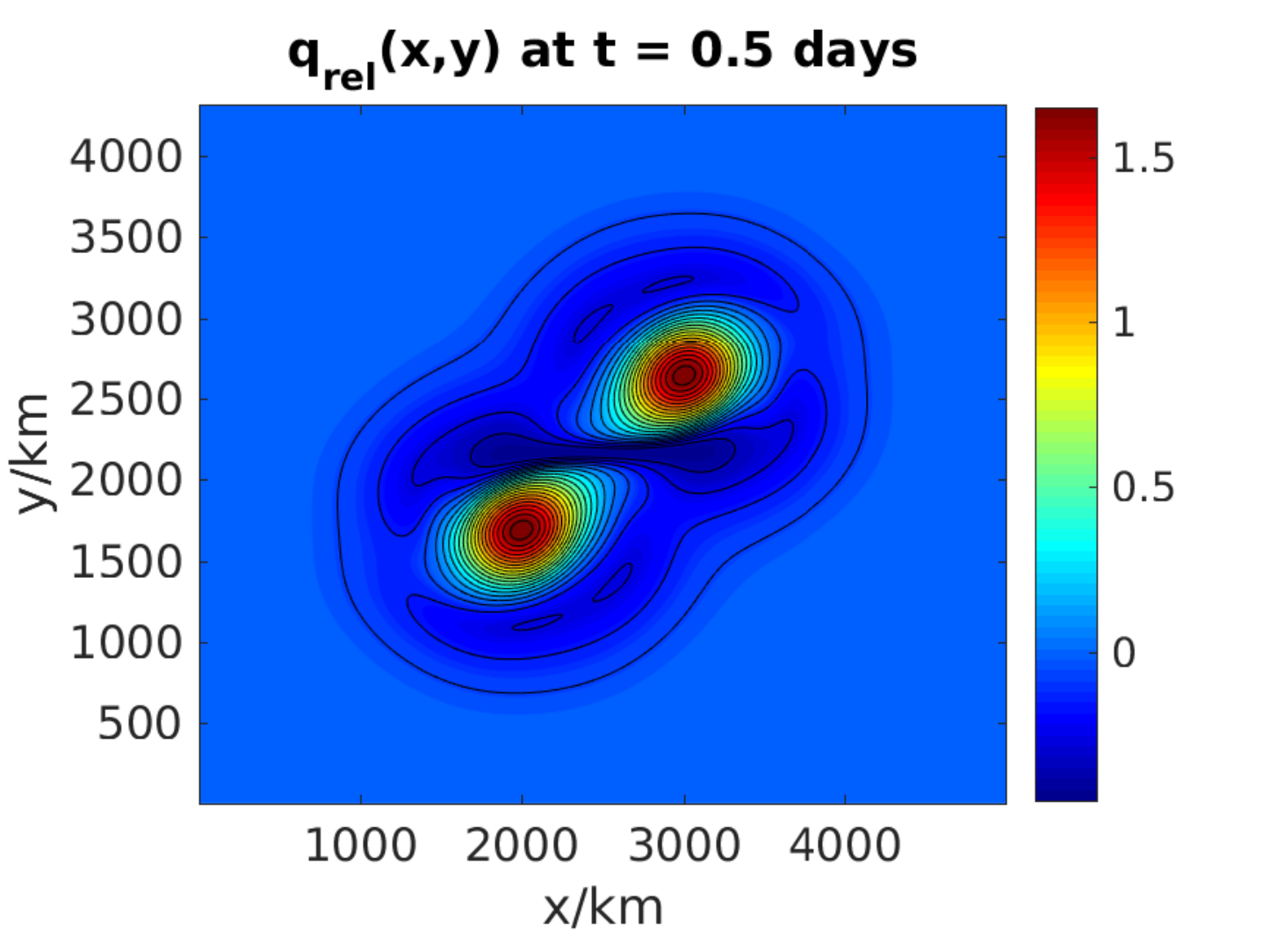}} &  
  \hspace*{-0.5cm}{\includegraphics[scale=0.30]{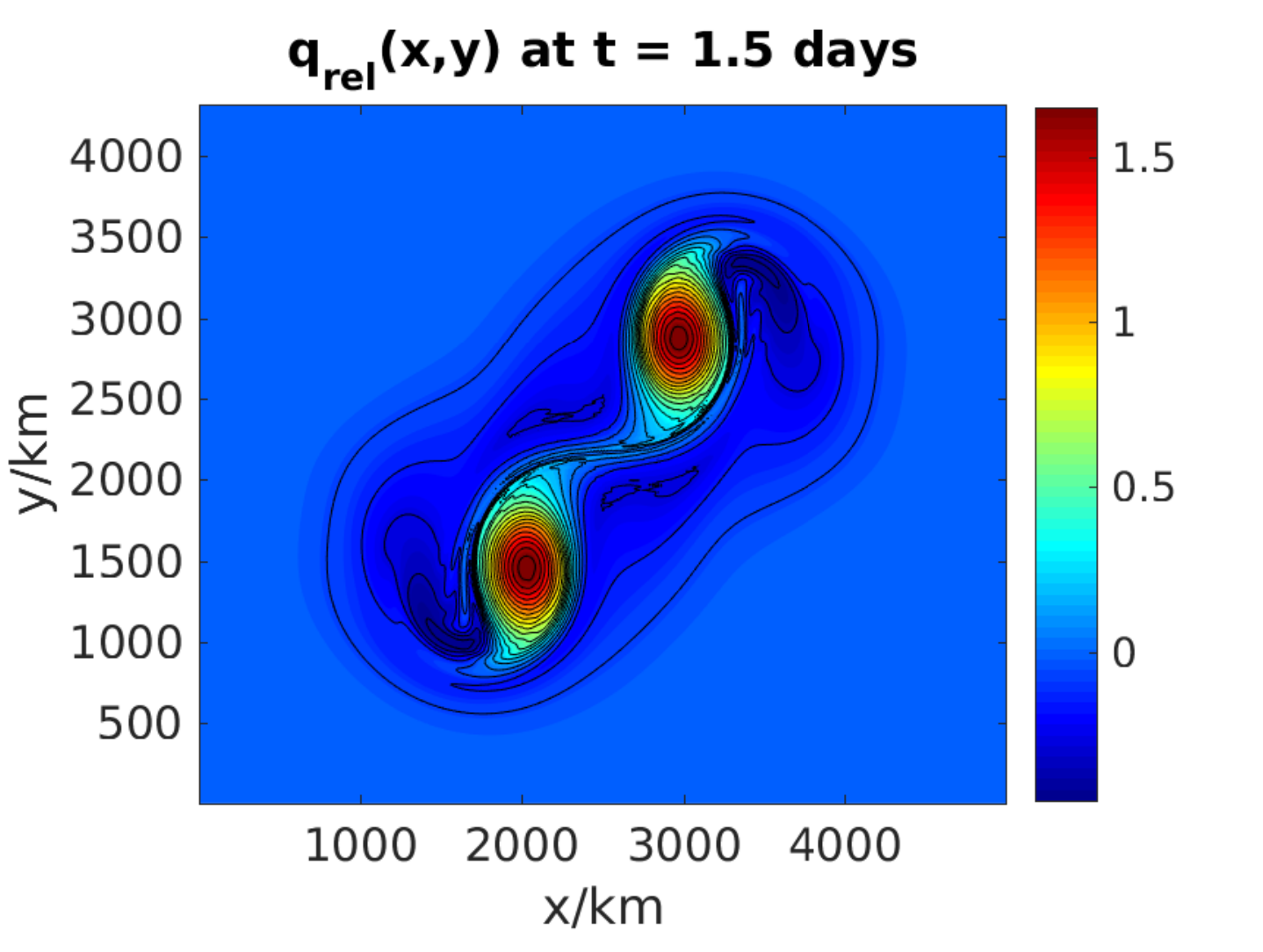}} &
  \hspace*{-0.5cm}{\includegraphics[scale=0.30]{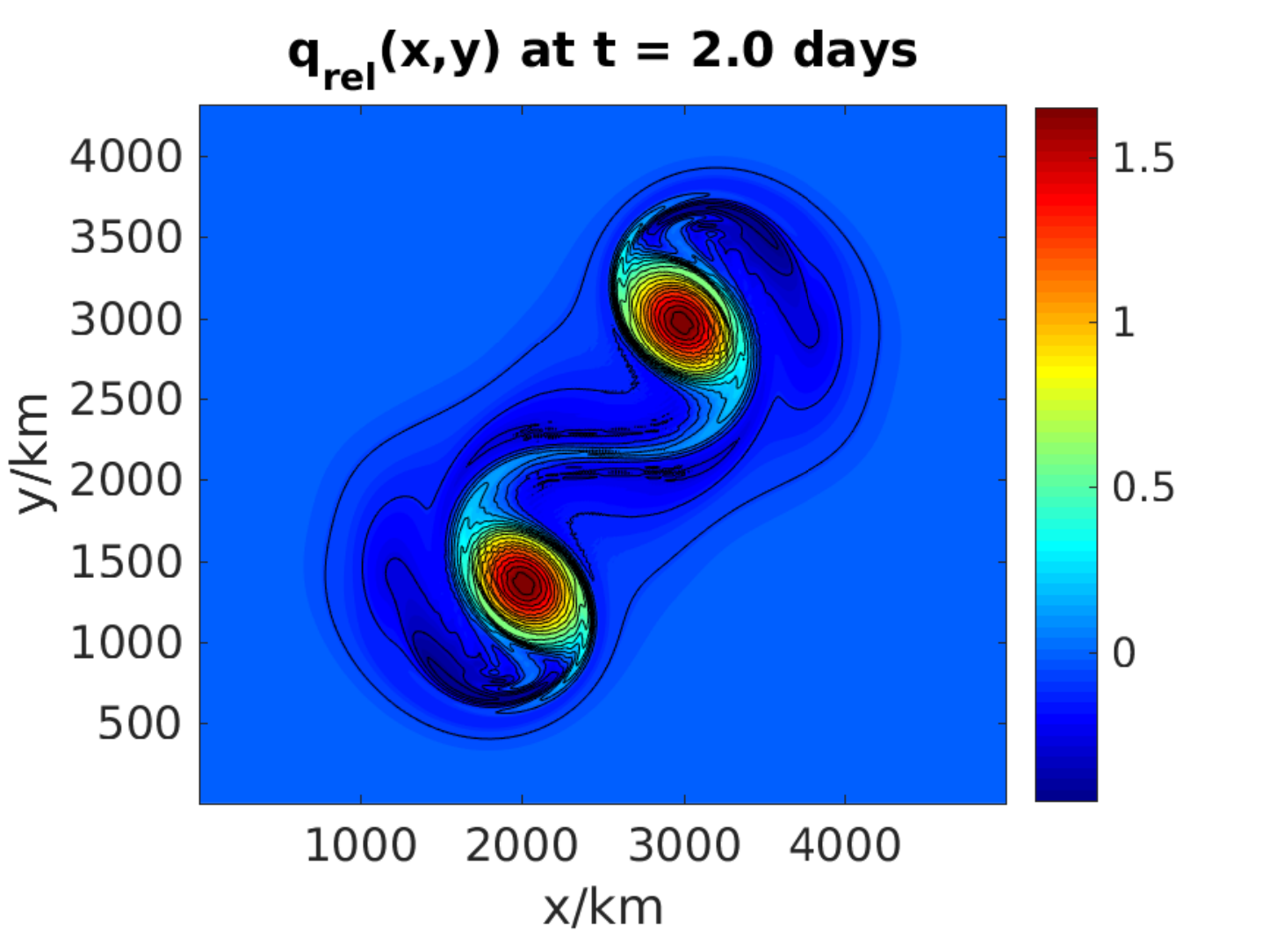}}  &
  \hspace*{-0.5cm}{\includegraphics[scale=0.3]{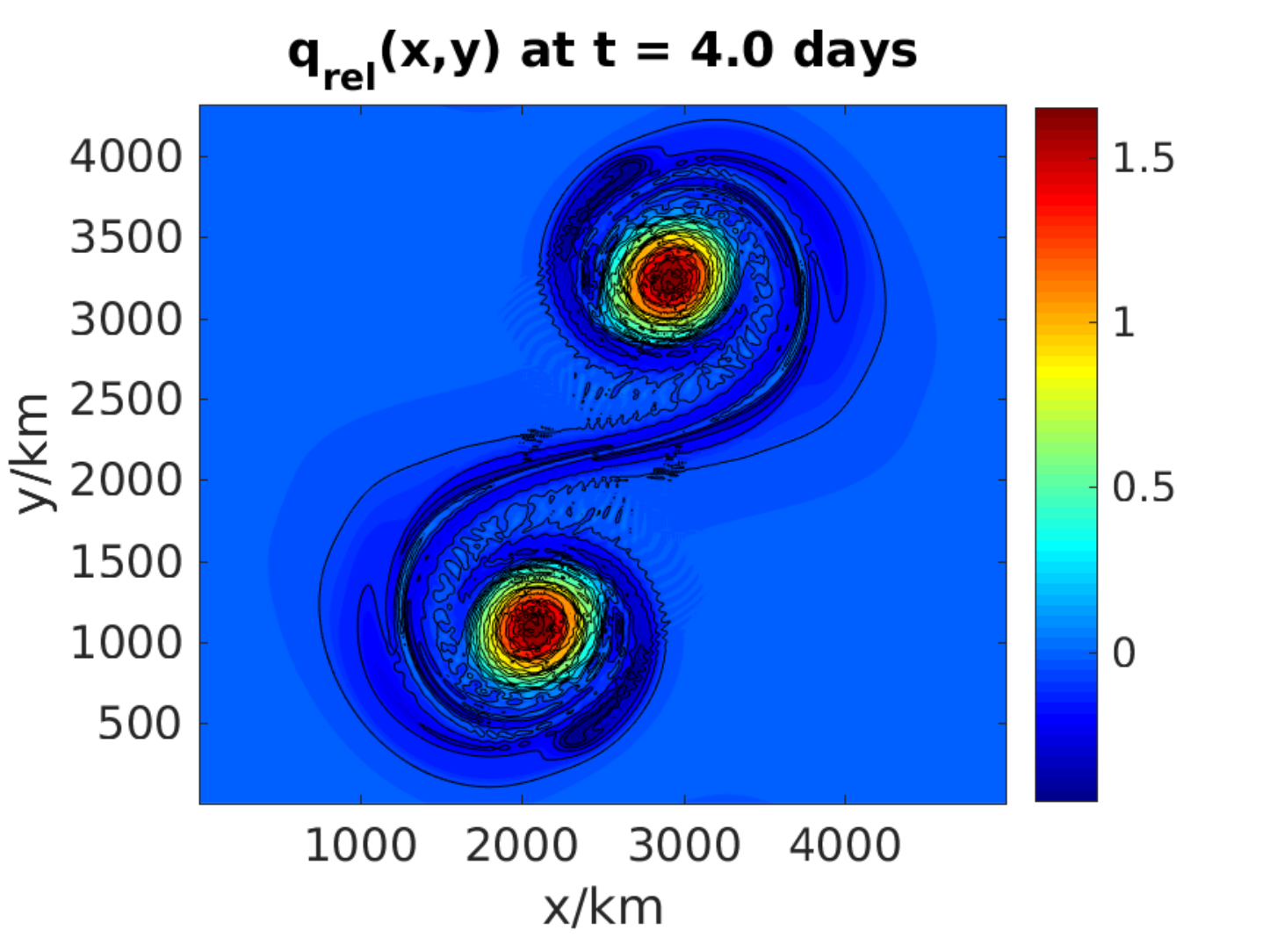}}  \\
  \hspace*{-0.25cm}{\includegraphics[scale=0.3]{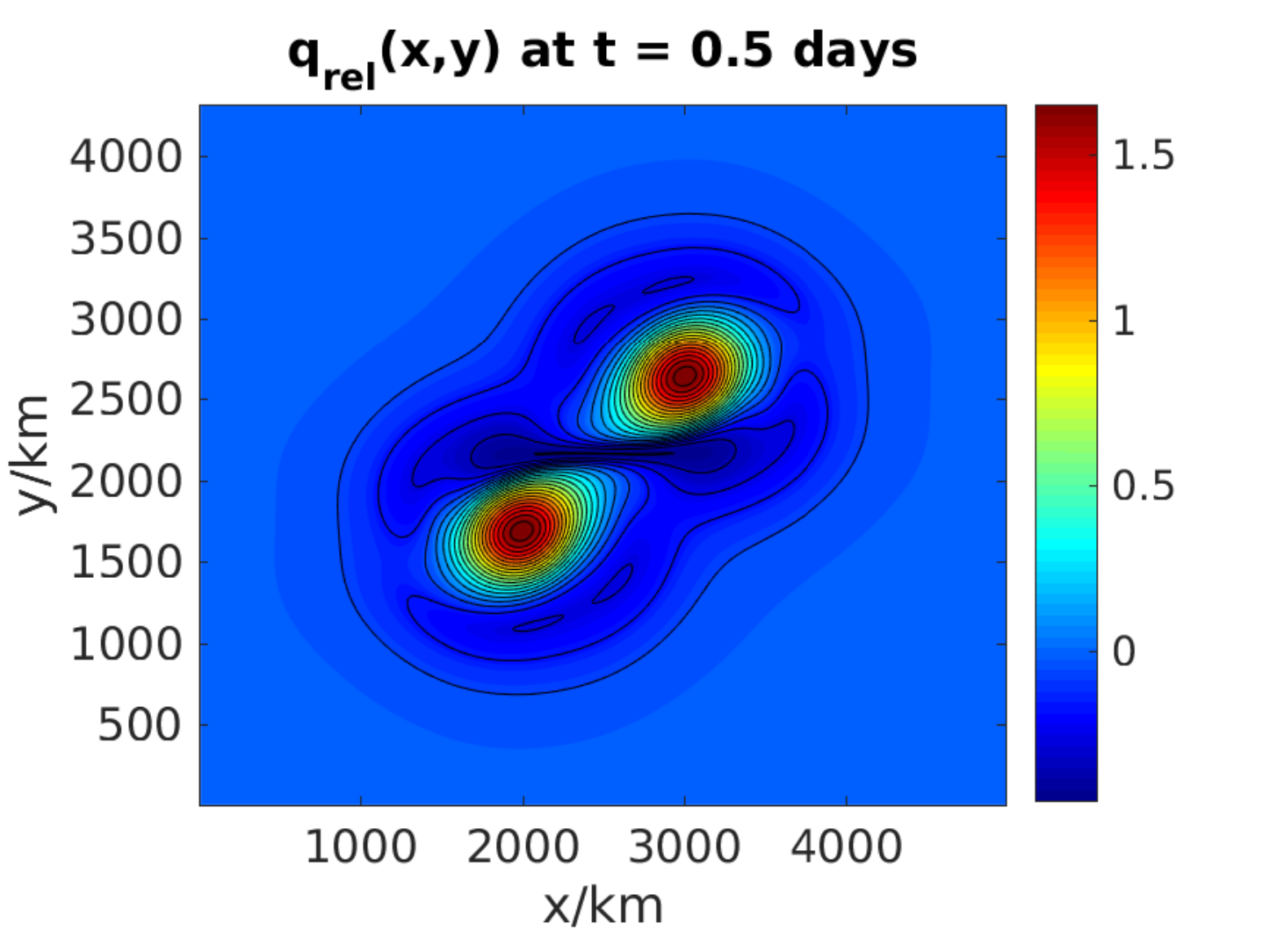}} &  
  \hspace*{-0.5cm}{\includegraphics[scale=0.3]{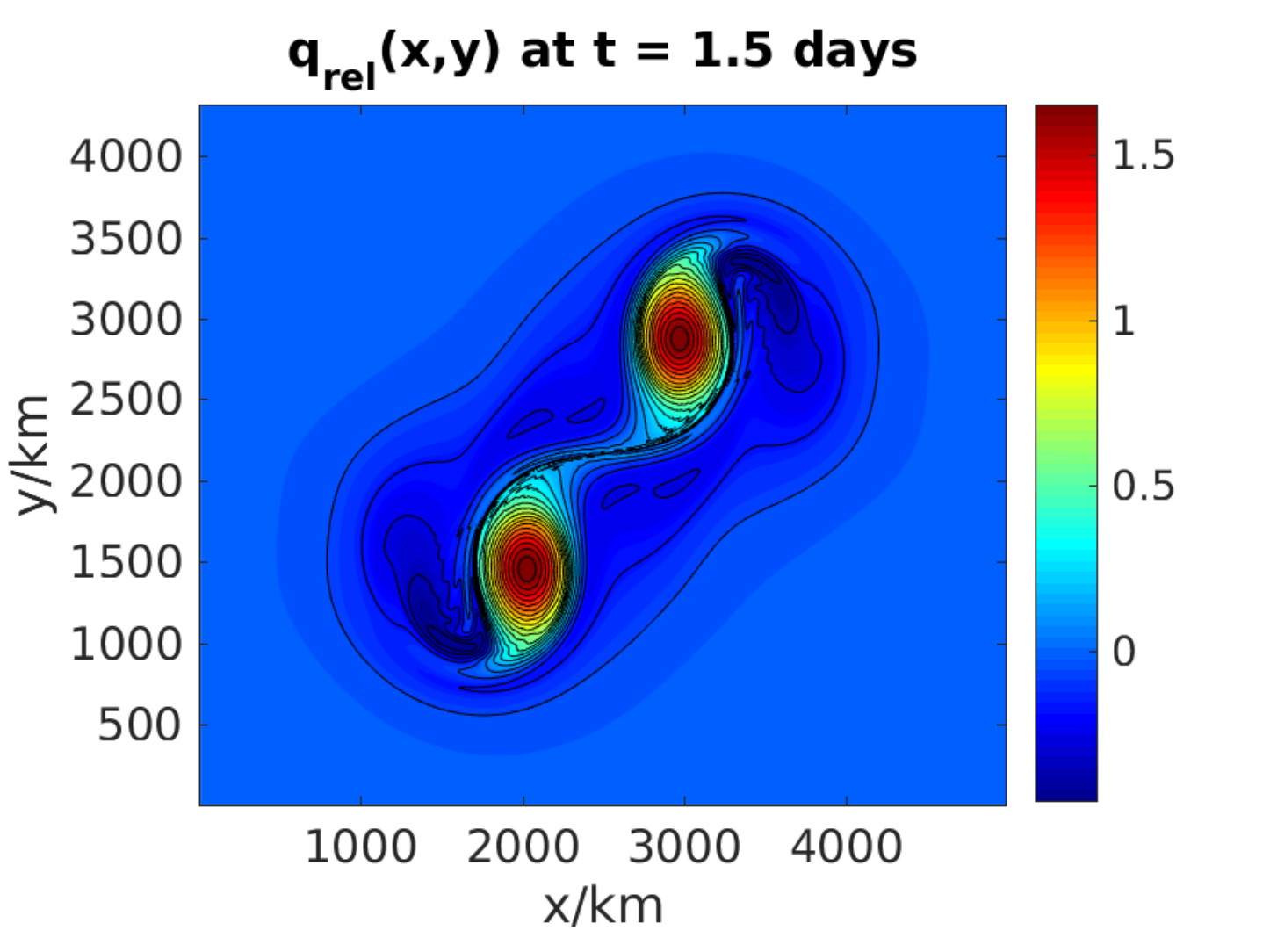}} &
  \hspace*{-0.5cm}{\includegraphics[scale=0.3]{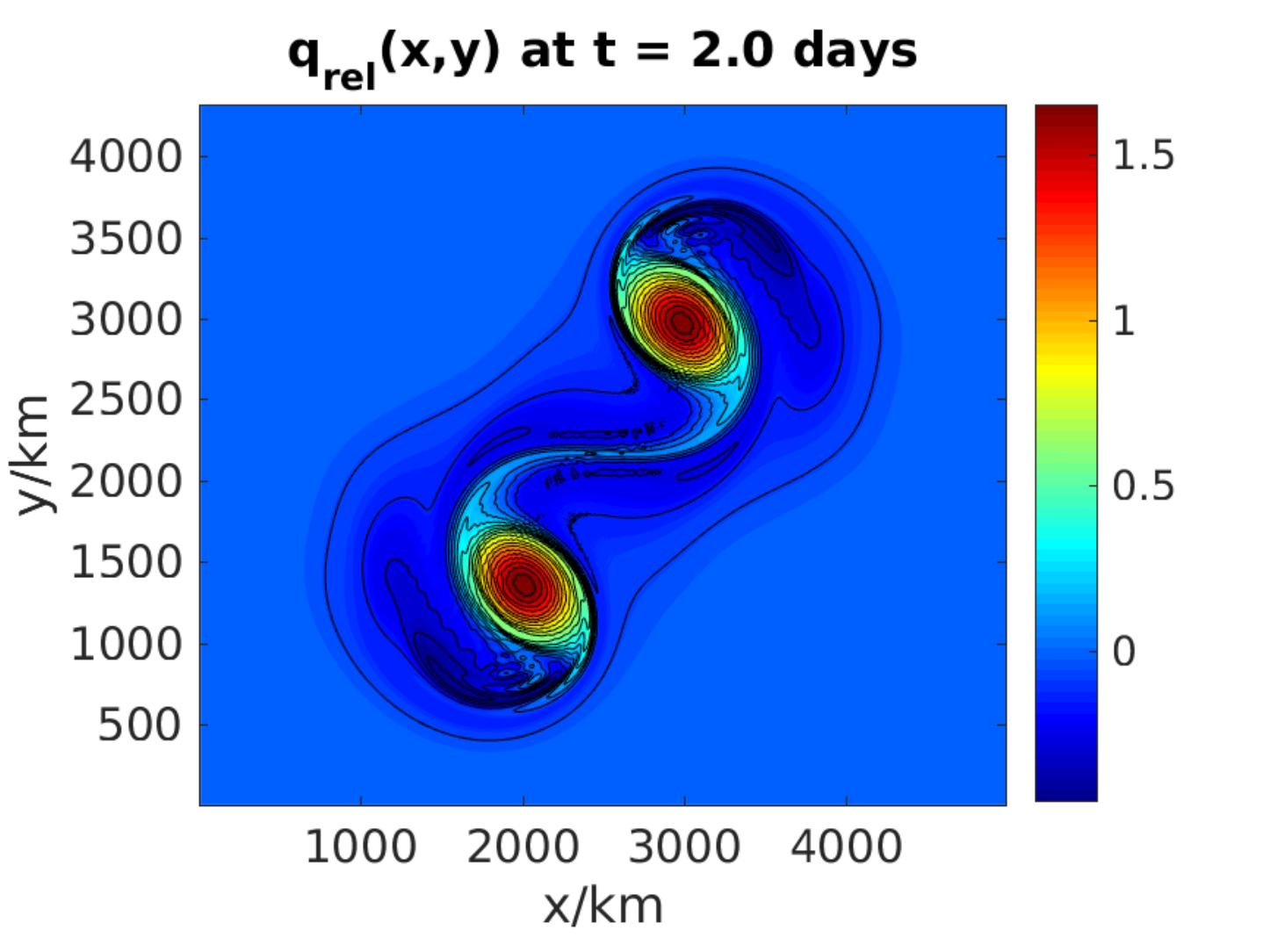}}  &
  \hspace*{-0.5cm}{\includegraphics[scale=0.3]{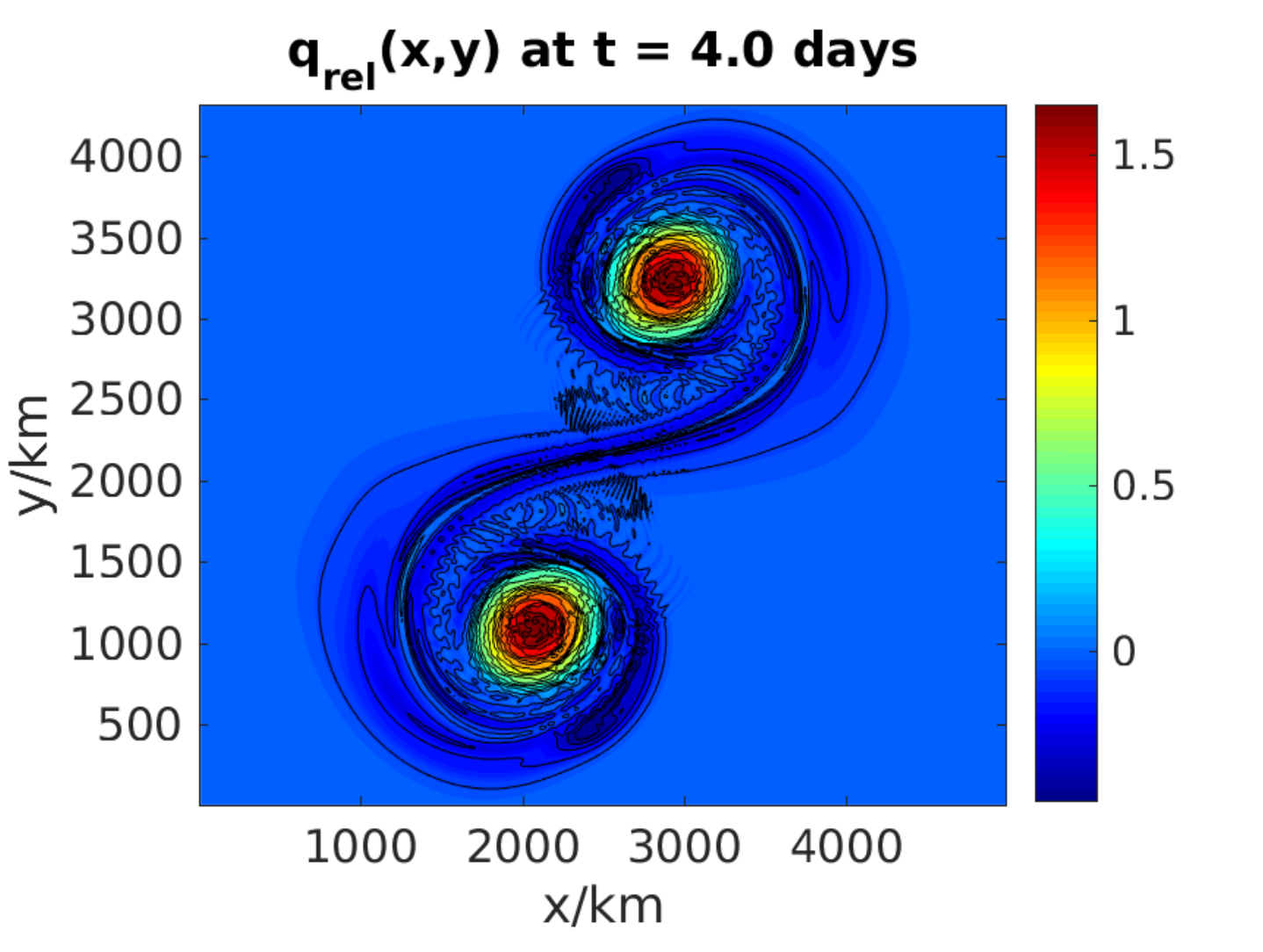}}  \\
 \end{tabular}
  \caption{Snapshots of relative potential vorticity $q_{\rm rel}$ for $H_0 = 10\,$km
  on regular (upper row) and irregular (lower row) meshes with $2 \cdot 256^2$ cells. 
  Contours between $-0.45\,{\rm days^{-1}km^{-1}}$ and $1.7\,{\rm days^{-1}km^{-1}}$ with 
  interval of $0.1\,{\rm days^{-1}km^{-1}}$.
  }                                                                                             
  \label{fig_Z_dual_vortices}
 \end{figure}

 \begin{figure}[h!]\centering
 \begin{tabular}{ccc}  
  \hspace*{-.35cm}{\includegraphics[scale=0.35]{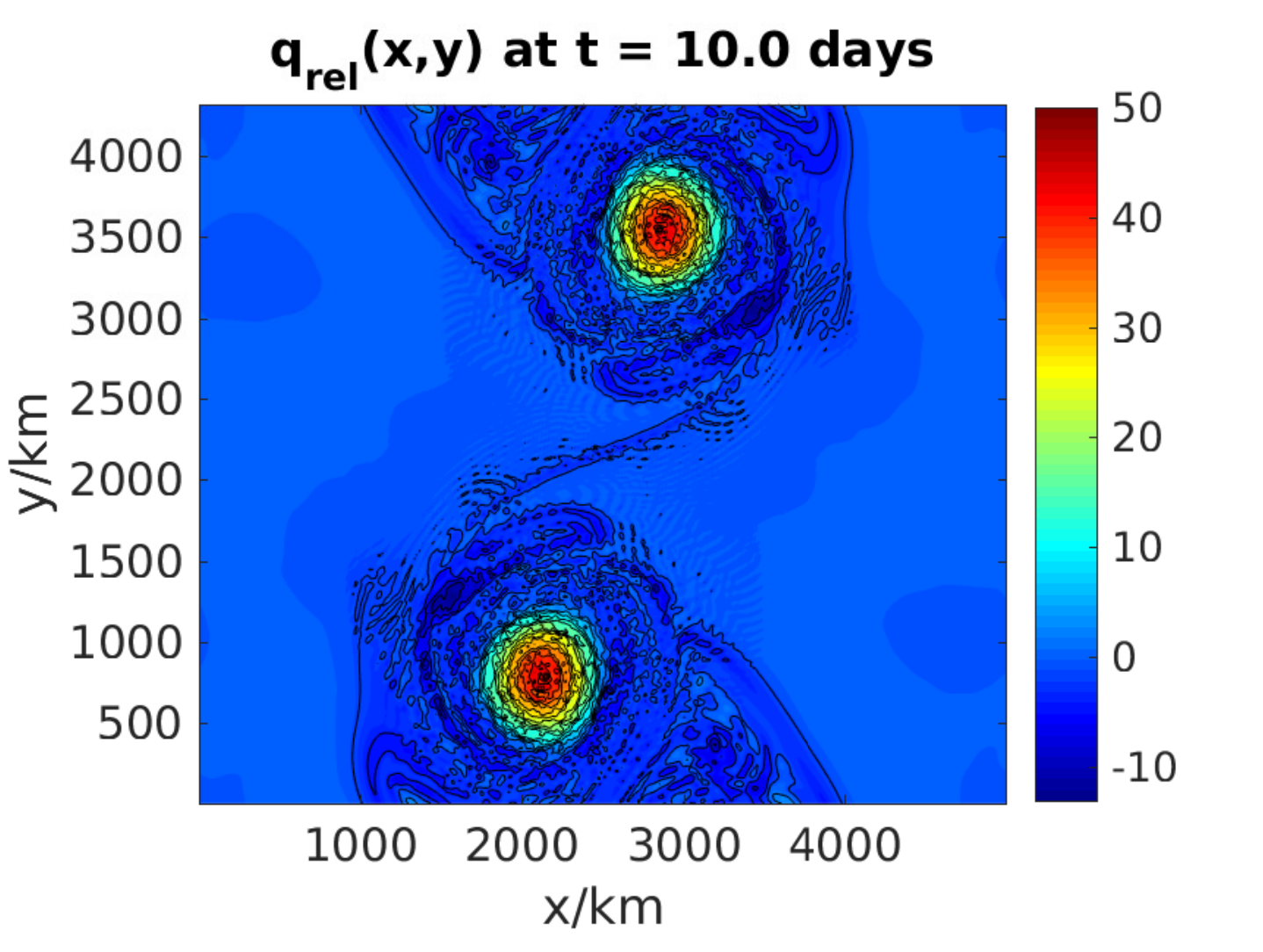}}  &
  \hspace*{-.5cm}{\includegraphics[scale=0.35]{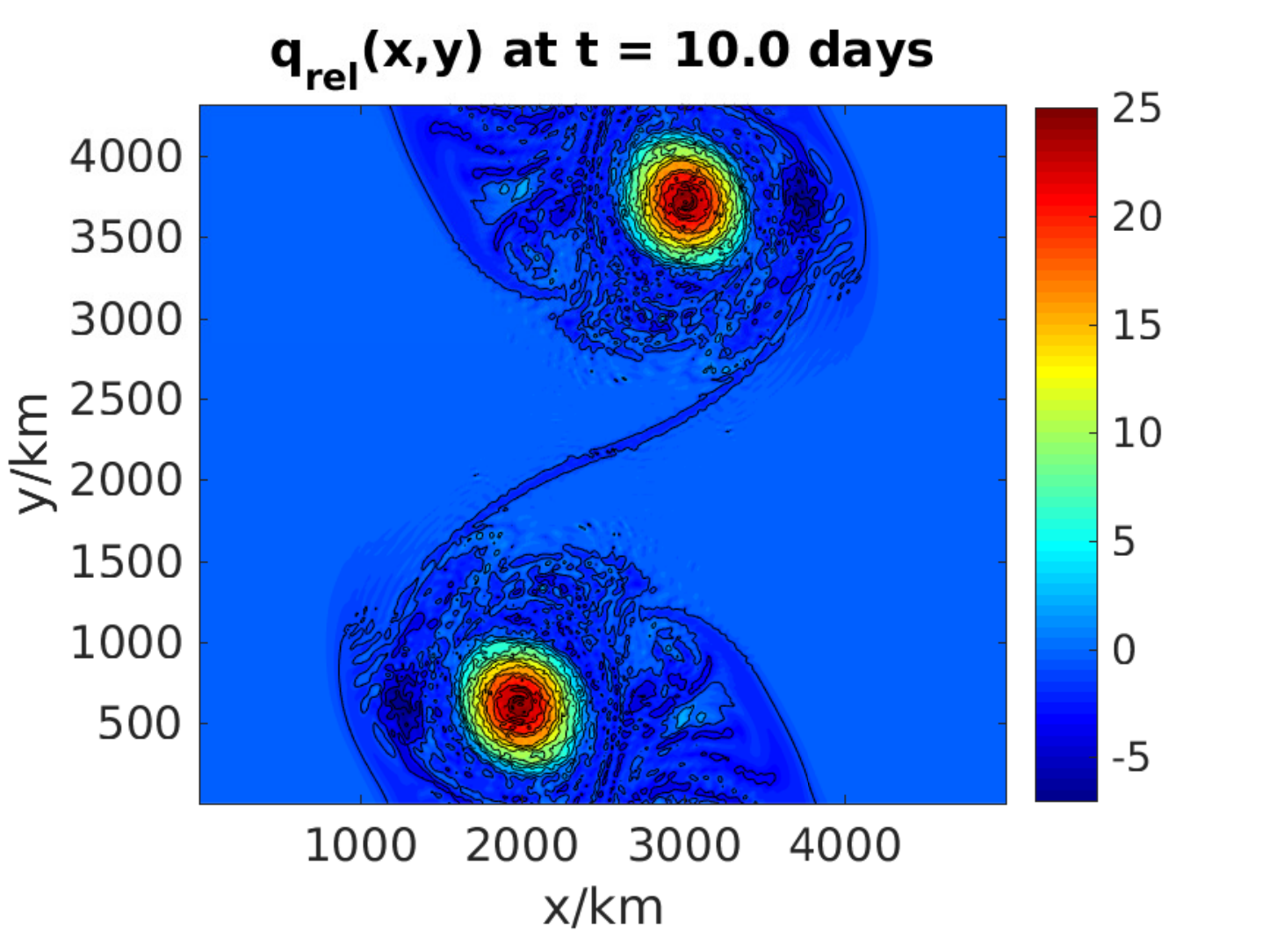}}  &
  \hspace*{-.5cm}{\includegraphics[scale=0.35]{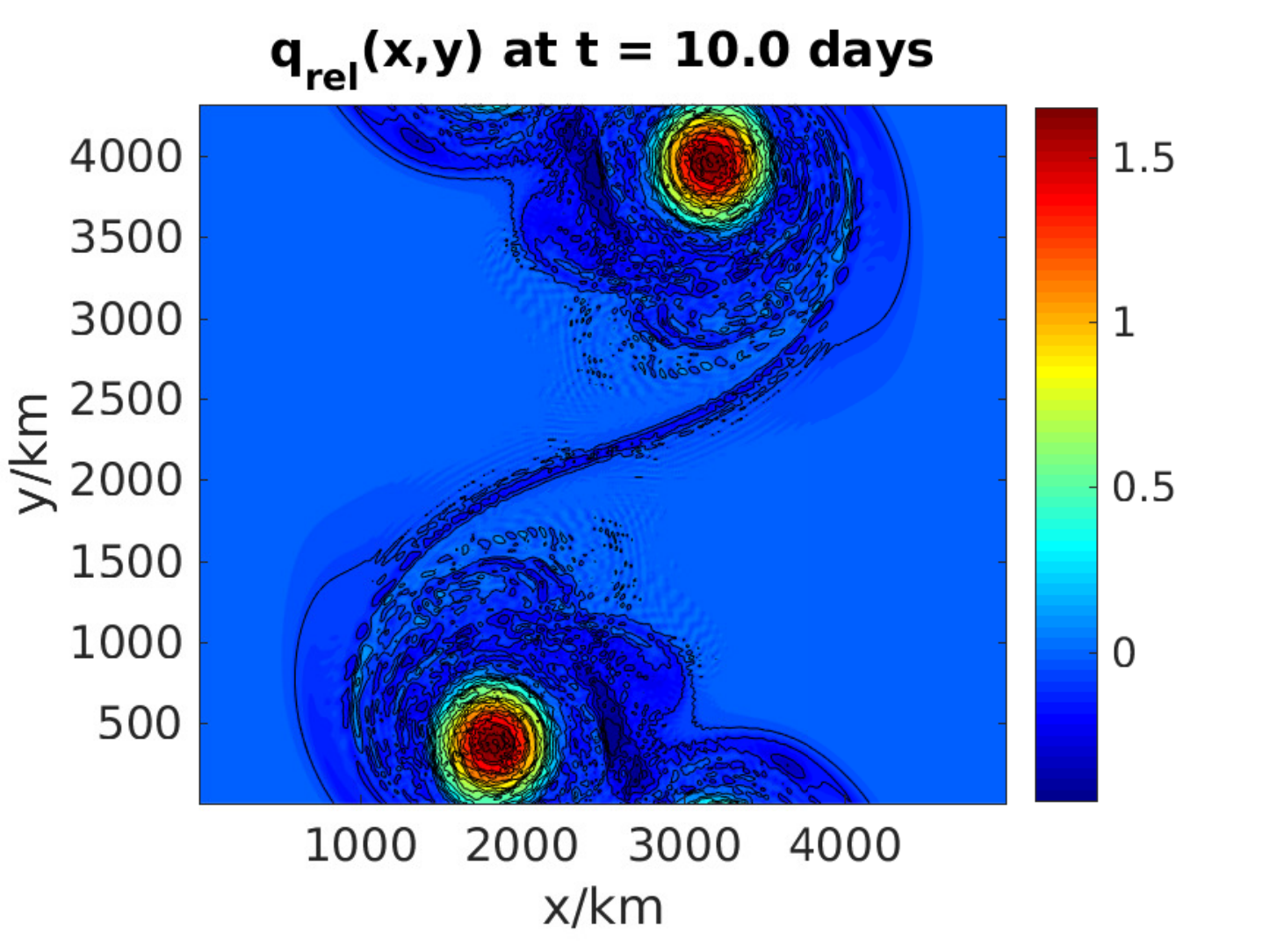}}  \\
  \hspace*{-.35cm}{\includegraphics[scale=0.35]{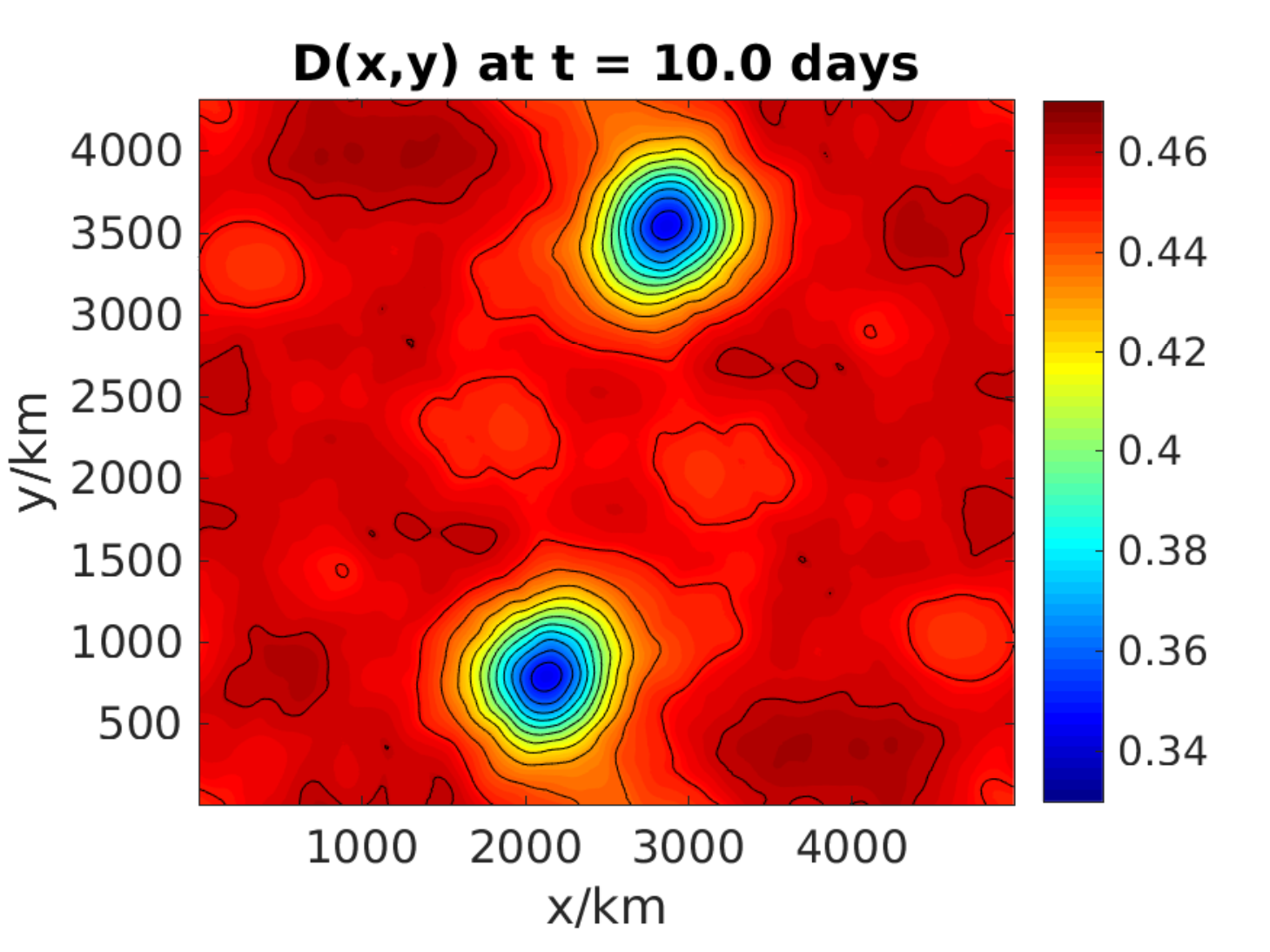}} &  
  \hspace*{-.5cm}{\includegraphics[scale=0.35]{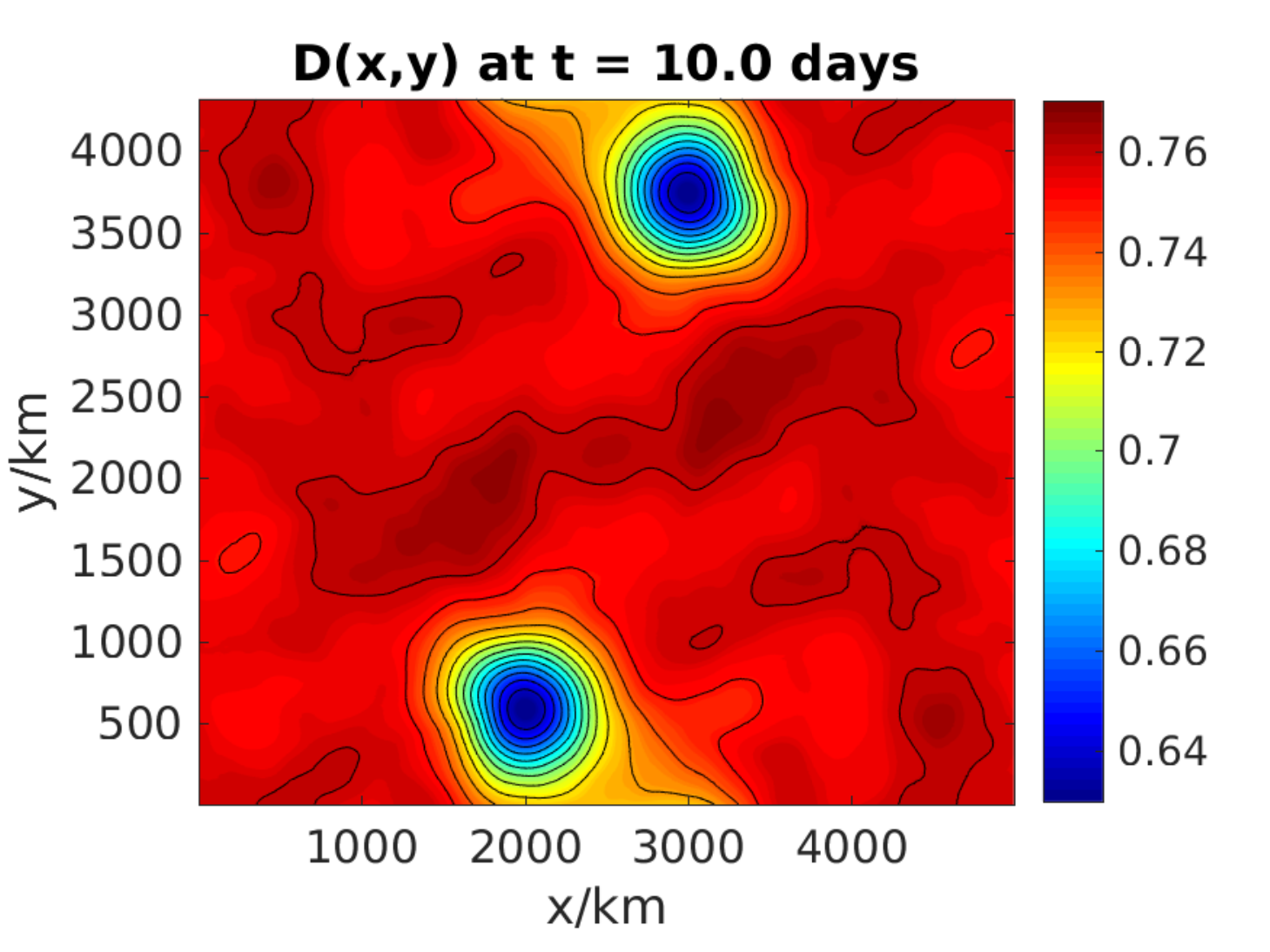}} &   
  \hspace*{-.5cm}{\includegraphics[scale=0.35]{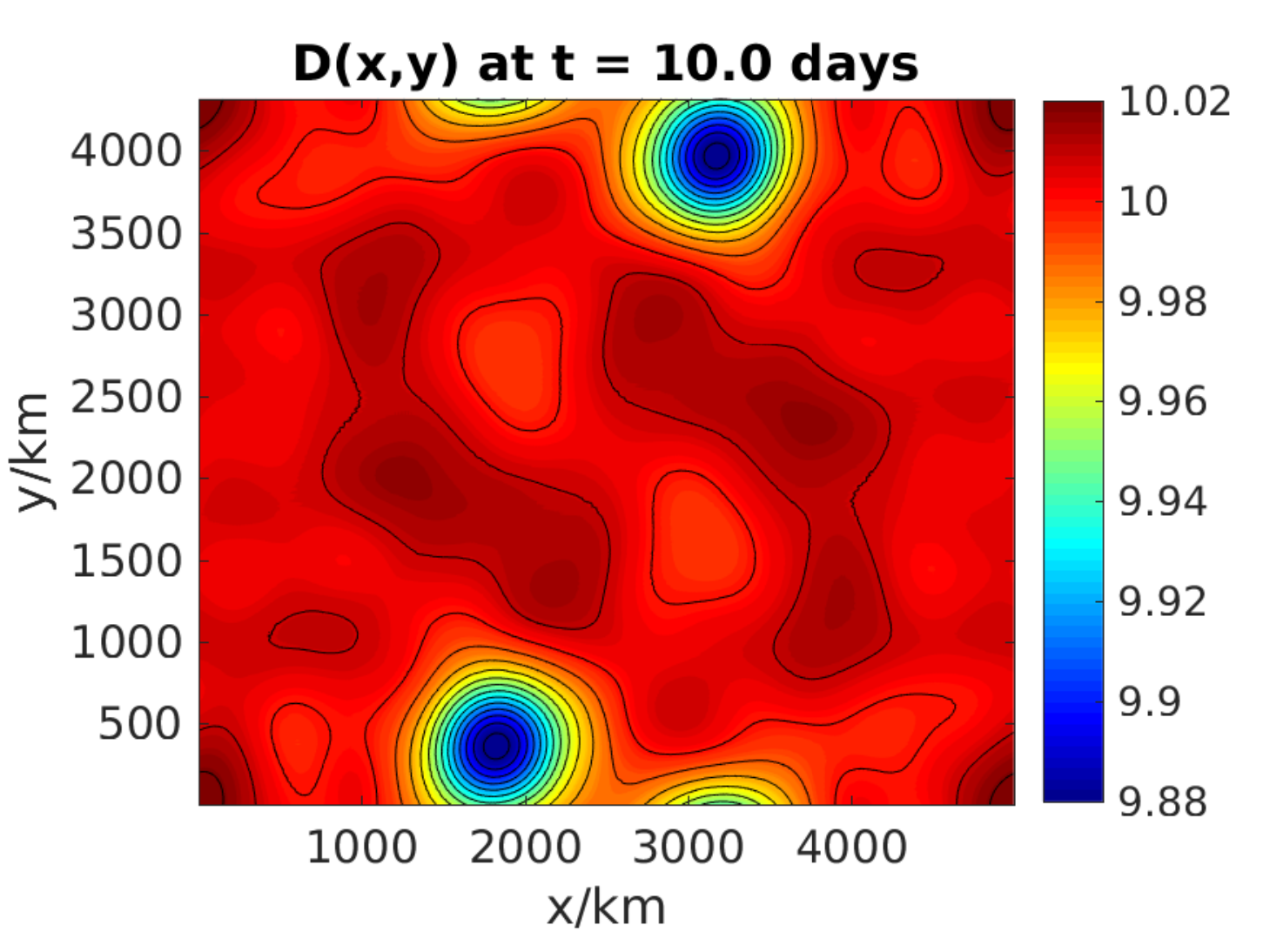}} 
 \end{tabular}
   \caption{Comparison of $q_{\rm rel}$ and $D$ for fluids in semi-geostrophic (left), quasi-geostrophic (middle), 
  and incompressible regimes (right) for a regular mesh with $2 \cdot 256^2$ cells.  
  Contours for $D$ between $-0.12\,{\rm km} +H_0$ and $0.02\,{\rm km} + H_0$ with interval of 
  $0.01\,{\rm km}$. Contours for $q_{\rm rel}$; 
  left: between $-13\,{\rm days^{-1}km^{-1}}$ and $50\,{\rm days^{-1}km^{-1}}$ with interval of $3\,{\rm days^{-1}km^{-1}}$;
  middle: between $-7\,{\rm days^{-1}km^{-1}}$ and $25\,{\rm days^{-1}km^{-1}}$ with interval of $2\,{\rm days^{-1}km^{-1}}$;
  right: between $-0.45\,{\rm days^{-1}km^{-1}}$ and $1.7\,{\rm days^{-1}km^{-1}}$ with interval of $0.1\,{\rm days^{-1}km^{-1}}$.
  }                                                                                             
  \label{fig_Z_dual_Hall_vortices}
 \end{figure}

 \begin{figure}[h!]\centering
 \begin{tabular}{cccccc}  
  \hspace*{-.075cm}{\includegraphics[scale=0.38]{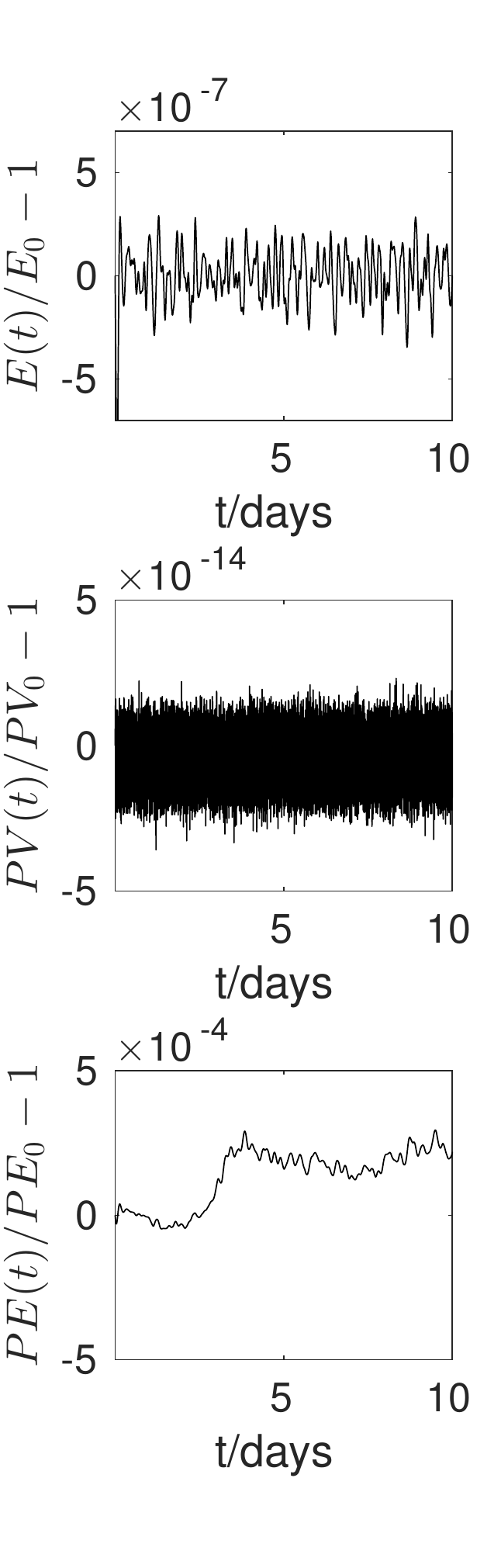}}  & 
  \hspace*{-.075cm}{\includegraphics[scale=0.38]{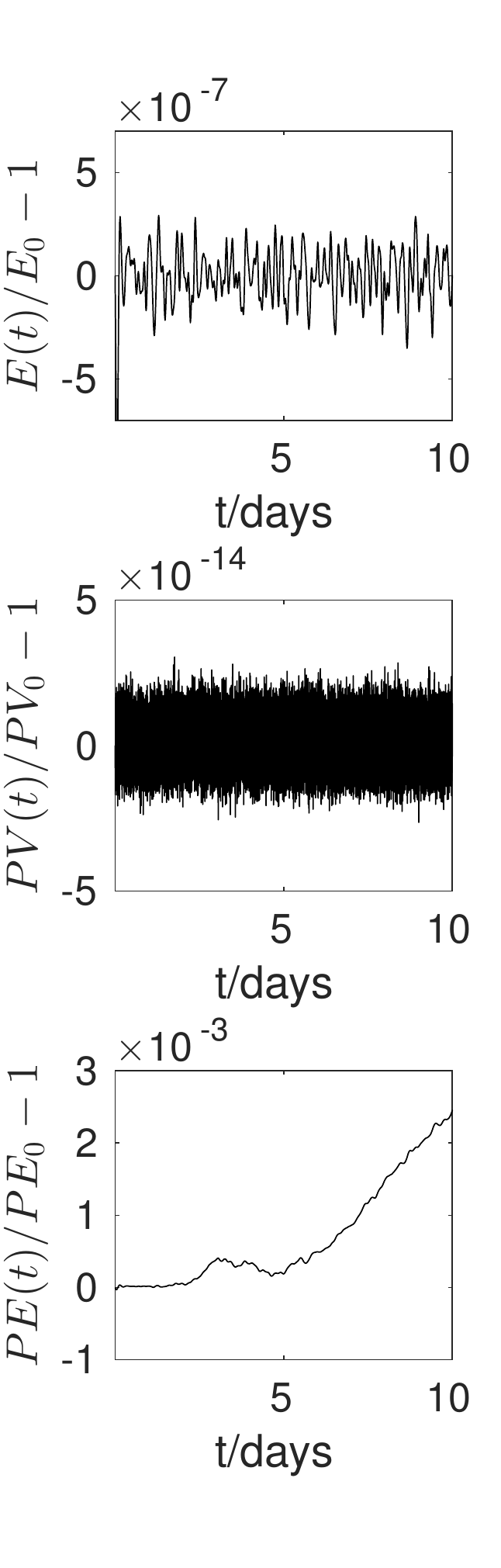}}  &
  \hspace*{-.075cm}{\includegraphics[scale=0.38]{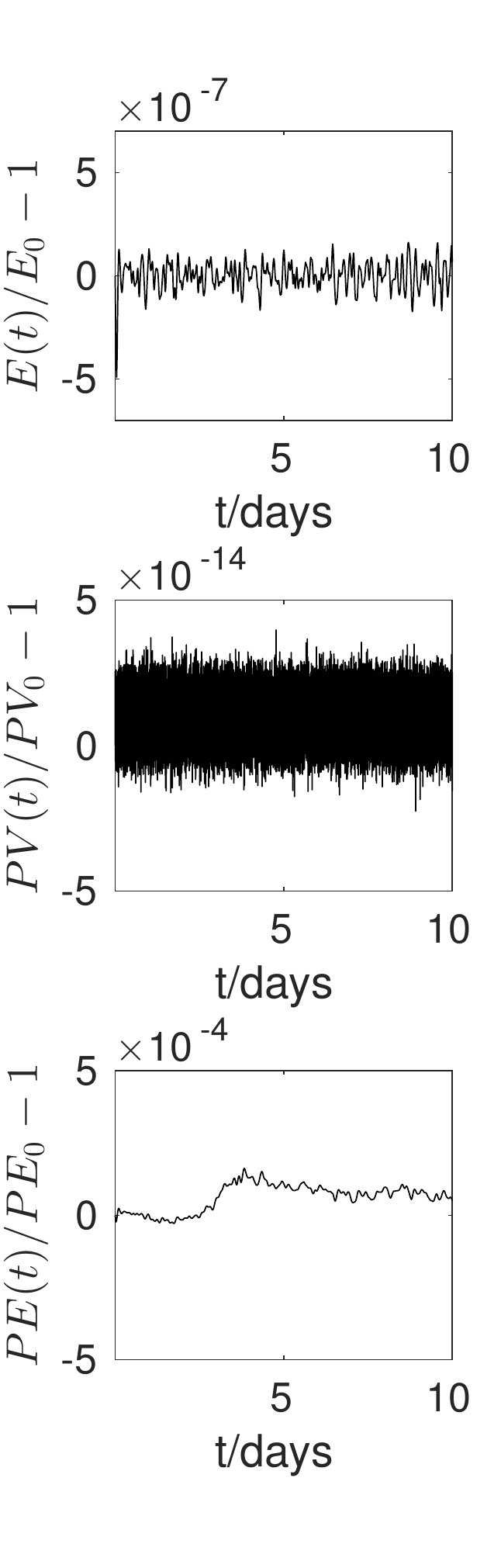}}  & 
  \hspace*{-.075cm}{\includegraphics[scale=0.38]{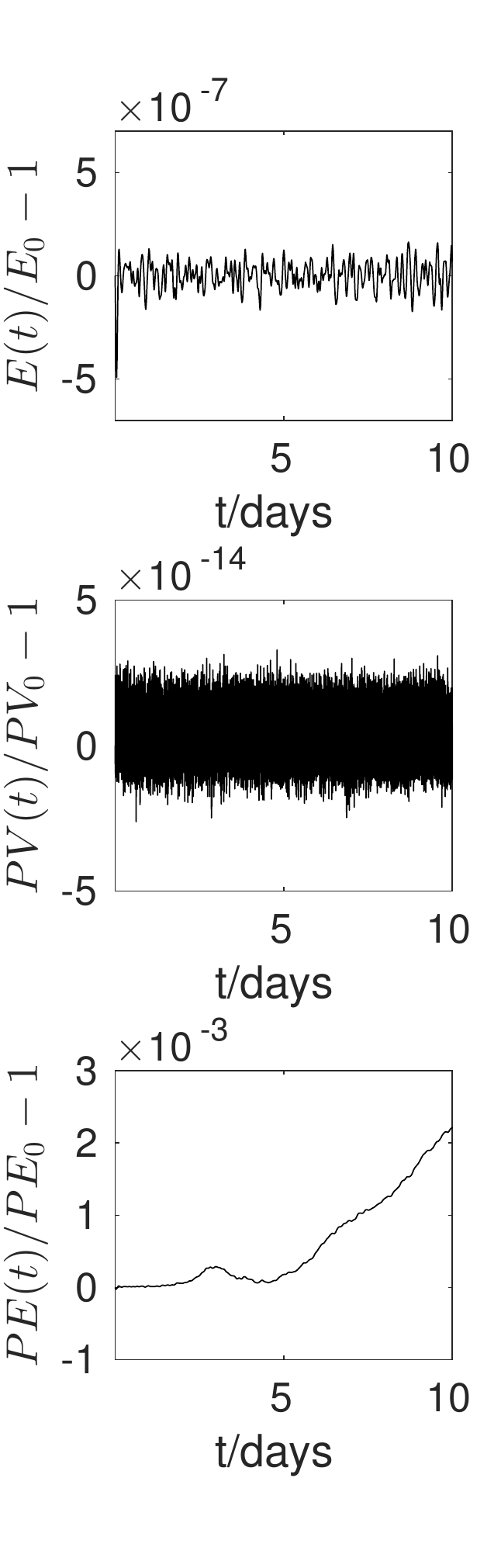}}  &
  \hspace*{-.075cm}{\includegraphics[scale=0.38]{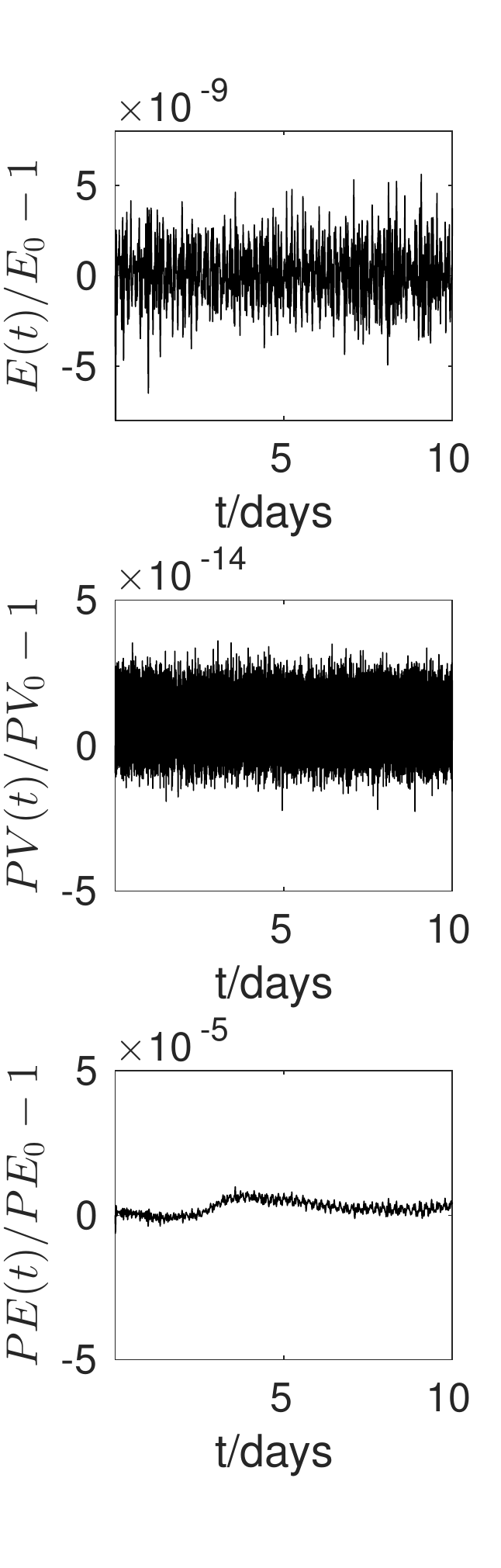}}  & 
  \hspace*{-.075cm}{\includegraphics[scale=0.38]{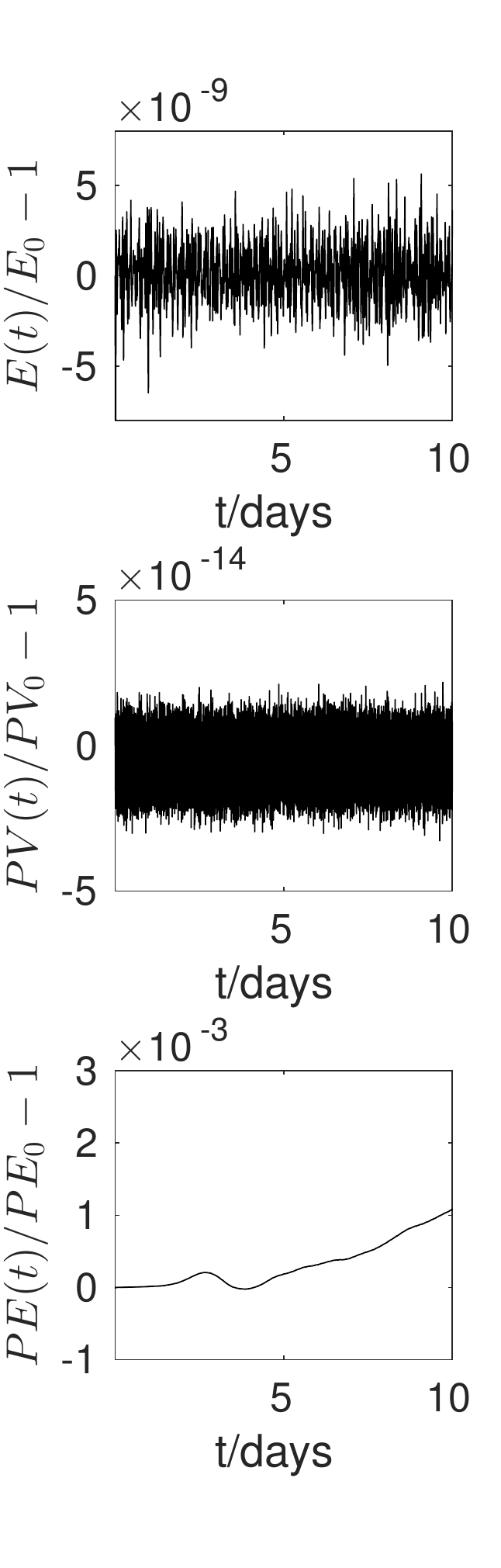}}  
 \end{tabular}
  \caption{Vortex interaction test case: relative errors of $E(t)$ (upper row), of 
  $PV(t)$ (middle row), and $PE(t)$ (lower row) for a fluid in semi-geostrophic ($1^{st},2^{nd}$ column), 
  in quasi-geostrophic ($3^{rd},4^{th}$ column), and in incompressible ($5^{th},6^{th}$ column) regime
  for regular ($1^{st},3^{rd},5^{th}$ column) and irregular ($2^{nd},4^{th},6^{th}$ column) meshes with 
  $2 \cdot 256^2$ cells. 
  }                                                                                             
  \label{fig_dual_qoi}
 \end{figure}

 \subsubsection{Shear flow in semi-geostrophic regime}
 \label{sec_tc_shear}
 
 In the second test case with dominantly nonlinear effects, we study 
 the evolution of a shear flow in quasi-geostrophic regime. The flow is initialized
 so that it is in unstable equilibrium. During the flow evolution, perturbations 
 that are superimposed on an initial zonal jet in $x$-direction (cf. Fig.~\ref{fig_Z_shear_initials}) 
 evolve towards two pairs of counter-rotating vortices (cf. Fig.~\ref{fig_Z_shear_dynamics}). 
 As the development and growth rate of this instabilities essentially depend on nonlinear 
 effects, we thus evaluate the accurate representation of these effects by the nonlinear terms of 
 the variational RSW scheme.
 
 \medskip
 
 \paragraph{Initialization.}
  An initialization according to \cite{RestelliHundermark2009} will allow us 
  also for this test case a direct comparisons of the corresponding numerical results. 
  Hence, we initialize the shear flow in quasi-geostrophic regime by the following 
  fluid depth $h$ while enforcing the geostrophic balance. There follows
  \begin{equation}
  h(x,y,0)  = H_0 - H' \frac{y''}{\sigma_y}e^{-\frac{y'^2}{2\sigma_y^2} +\frac{1}{2}} 
		      \left( 1+ \kappa \sin \left( \frac{2\pi x'}{\lambda_x}  \right)  \right)\ ,
  \end{equation}
  \begin{equation}
  \begin{split}
  & u(x,y,0) =  \frac{ g H'}{f \sigma_y L_y } \left ( c(y) -  \frac{y''^2}{\sigma_y^2} \right )e^{-\frac{y'^2}{2\sigma_y^2} +\frac{1}{2}} 
		      \left( 1+ \kappa \sin \left( \frac{2\pi x'}{\lambda_x}  \right)  \right)\ , \\
  & v(x,y,0) = - \frac{ g H'}{f L_x } \frac{2\pi \kappa}{\lambda_x} \frac{y''}{\sigma_y} e^{-\frac{y'^2}{2\sigma_y^2} +\frac{1}{2}} \cos \left( \frac{2\pi x'}{\lambda_x}  \right)
  \end{split}
  \end{equation}
  using the definitions $ c(y) = \cos\left( \frac{2\pi}{L_y}\left(y - \frac{L_y}{2}\right)  \right)$ 
  and 
  \begin{equation}
  x' = \frac{x}{L_x} , \quad
  y' = \frac{1}{\pi} \sin  \left( \frac{\pi}{L_y } \left(y - \frac{L_y}{2}  \right)  \right) , \quad
  y'' = \frac{1}{2\pi} \sin  \left( \frac{2\pi}{L_y } \left(y - \frac{L_y}{2}  \right)  \right) ,
  \end{equation}
  for the parameters $\lambda_x = \frac{1}{2}$, $\sigma_y = \frac{1}{12}$, $\kappa = 0.1 $.
  To obtain an inviscid flow in quasi-geostrophic (compressible) flow regime with ${\rm Bu}\approx 1$, 
  we choose $H_0 = 1076\,$km and $H'=30\, $m, cf. \cite{RestelliHundermark2009}. 
  We set $B({\bf x}) =0$.
  The analytic fields are mapped to the mesh exactly as in Sect.~\ref{sec_tc_dual}.
  The corresponding initial fields are shown in Fig.~\ref{fig_Z_shear_initials}.

  \textcolor{black}{For the shear flow simulations, we apply a time step of $\Delta t = 36\,\rm s$ 
  and use regular and irregular meshes with $2\cdot 256^2$ cells and $\Delta x_{min} = 1.313\,$km. 
  For the water depth of $H_0= 1.0760\,$km, the Courant number here is $C = 2.82$. 
  }
 
  \begin{figure}[t!]  \centering	
  \begin{tabular}{cc}  
      \hspace*{-0cm}{\includegraphics[scale=0.55]{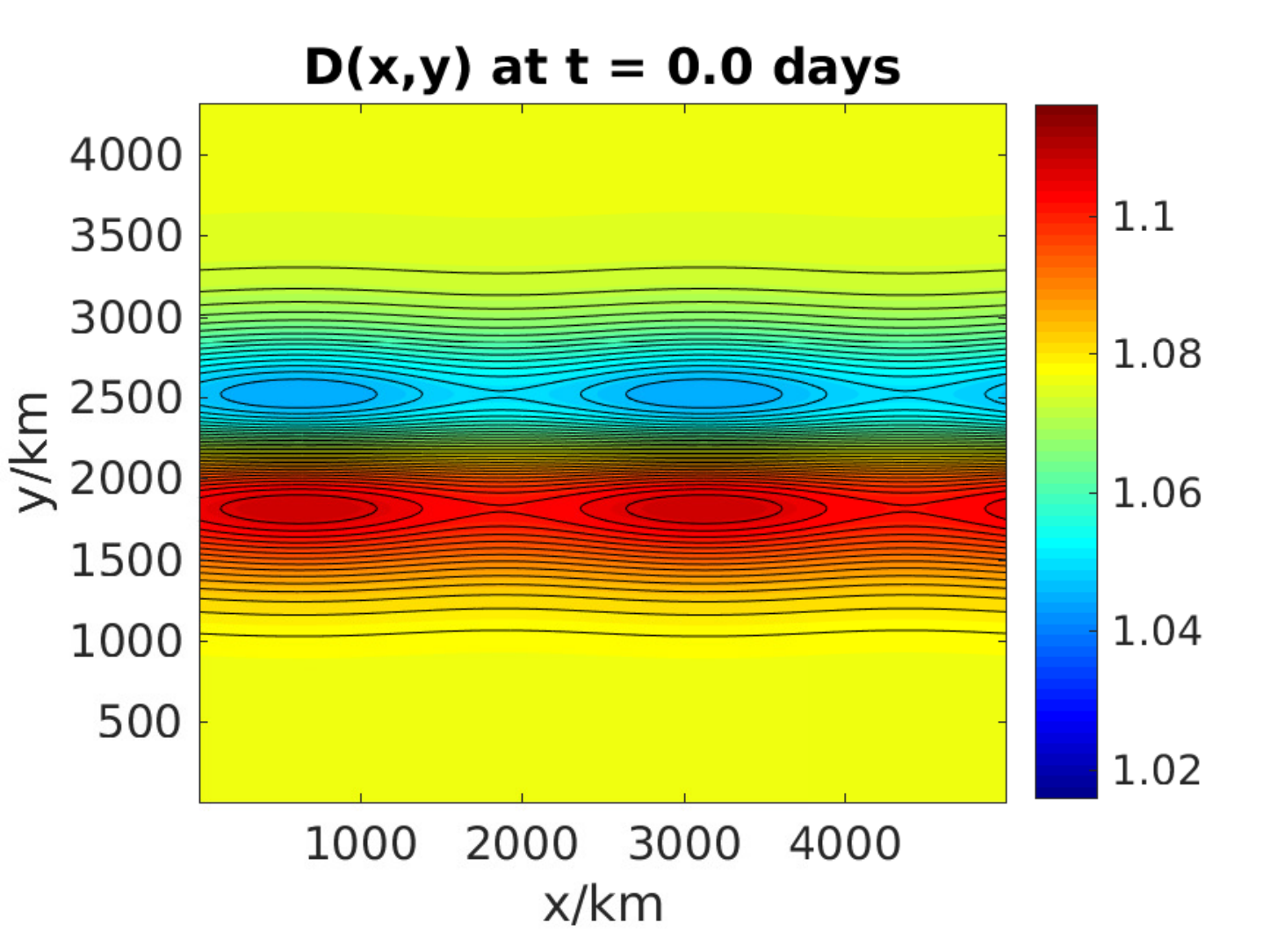}} &
      \hspace*{-0cm}{\includegraphics[scale=0.55]{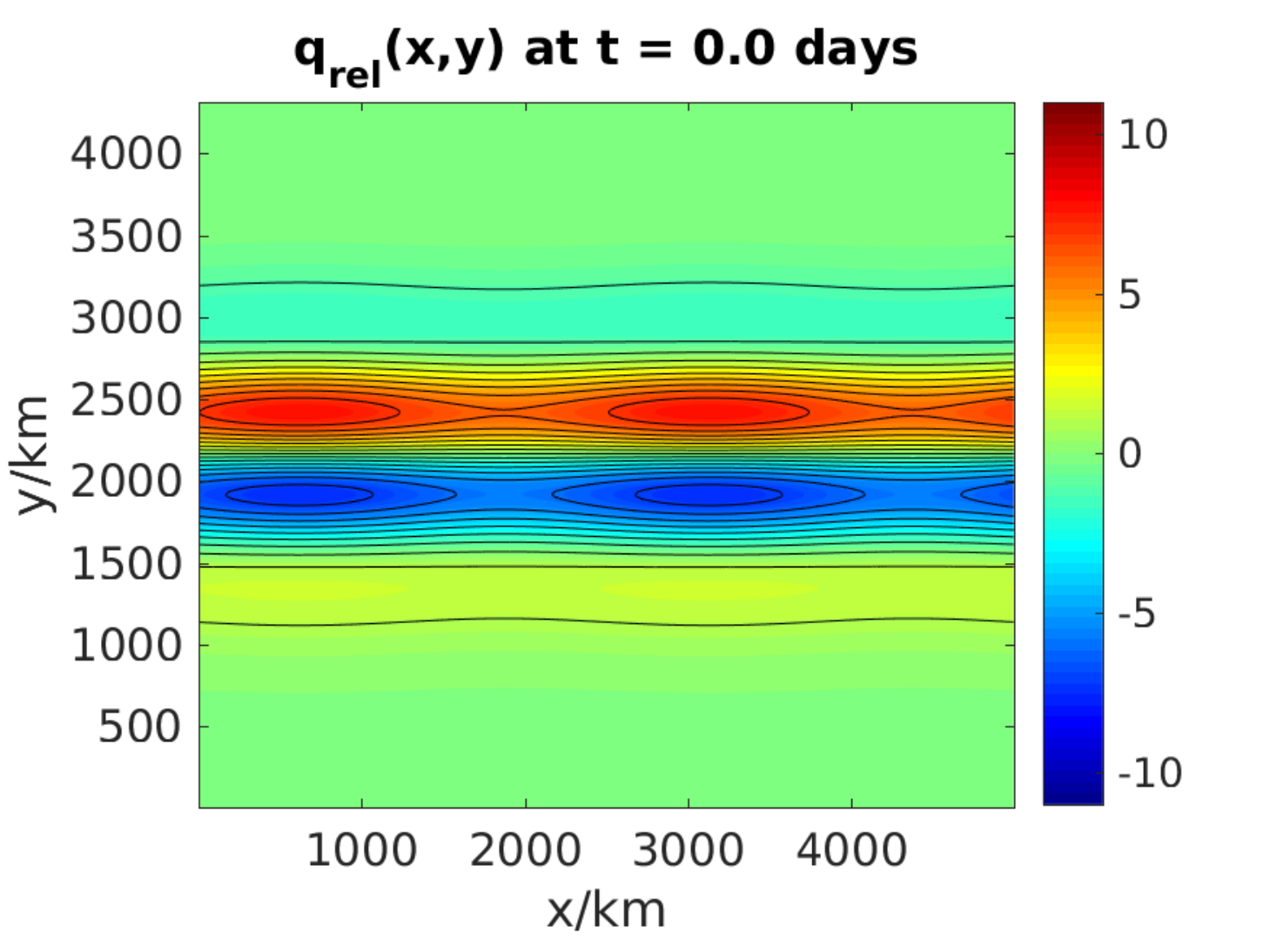}} 
  \end{tabular}
    \caption{Initial fields of fluid depth (left) and relative potential vorticity (right) 
             in geostrophic balance for the shear flow test case on a regular mesh with $N = 2 \cdot 256^2$ cells. 
            Contours for $q_{\rm rel}$ between $-11\,{\rm days^{-1}km^{-1}}$ and $11\,{\rm days^{-1}km^{-1}}$ with interval of $1\,{\rm days^{-1}km^{-1}}$,
            and for $D$ between $-0.06\,{\rm km} +H_0$ and $0.04\,{\rm km} + H_0$ with interval of $0.002\,{\rm km}$.
           }                                                                                             
  \label{fig_Z_shear_initials}
  \end{figure}

  \bigskip
  \paragraph{Results.}  
  Fig.~\ref{fig_Z_shear_dynamics} shows the time evolution of the shear flow 
  in unstable equilibrium. The initial perturbations grow within the 
  first three days to an instability with a dominant wave number of two. This
  instability develops further to two pairs of counter-rotating vortices 
  that are well developed at about day six. The filaments between these vortex 
  cores get thinner for later times and reach scales beyond spatial resolution, 
  which causes the noisy pattern visible at day 10.
  
  A comparison of the snapshots of Fig.~\ref{fig_Z_shear_dynamics} for regular (upper row)
  and irregular meshes (lower row) confirms that (i) the scheme is capable to produce 
  accurate solutions even on very deformed, irregular meshes and that (ii) these solutions
  agree well with literature, i.e. with a triangular \cite{RestelliHundermark2009} and 
  a hexagonal \cite{BauerPHD2013} C-grid discretization of the shallow water equations. 
  In more detail, all schemes' solutions for $q_{\rm rel}$ and $D$ provide very similar 
  general flow pattern -- compared at days 3, 6, and 10 -- 
  consisting in the number of vortex pairs, their magnitudes and positions, 
  and the structure of the thin filaments.

  Fig.~\ref{fig_diag_shearflow} shows the variational integrator's relative errors in total energy (upper row), mass-weighted potential vorticity $PV$ (middle row), and potential enstrophy $PE$ (lower row) for the 
  shear flow simulations in quasi-geostrophic regime performed on a regular (left column) and an irregular (right column) grid.
  For both mesh types, the variational integrator conserves the total energy at the order of $10^{-7}$
  for a time step size of $\Delta t = 36\,$s. As above, the errors decreases at $1^{st}$-order rate for 
  smaller $\Delta t$ and are independent from the spatial resolution. 
  Again, mass (not shown) and $PV$ are preserved at the order of machine precision, 
  independently from the chosen mesh type or from the spatial and temporal resolutions. 
  Also the errors in $PE$ do not significantly depend on this choice and remain preserved at the order of $10^{-3}$. 
  Here in case of the shear flow, we do not observe a growth of $PE$ in case of irregular meshes 
  unlike the vortex interaction test case.

  \begin{figure}[t!]\centering 
  \begin{tabular}{ccc}   
      \hspace*{-.35cm} {\includegraphics[scale=0.38]{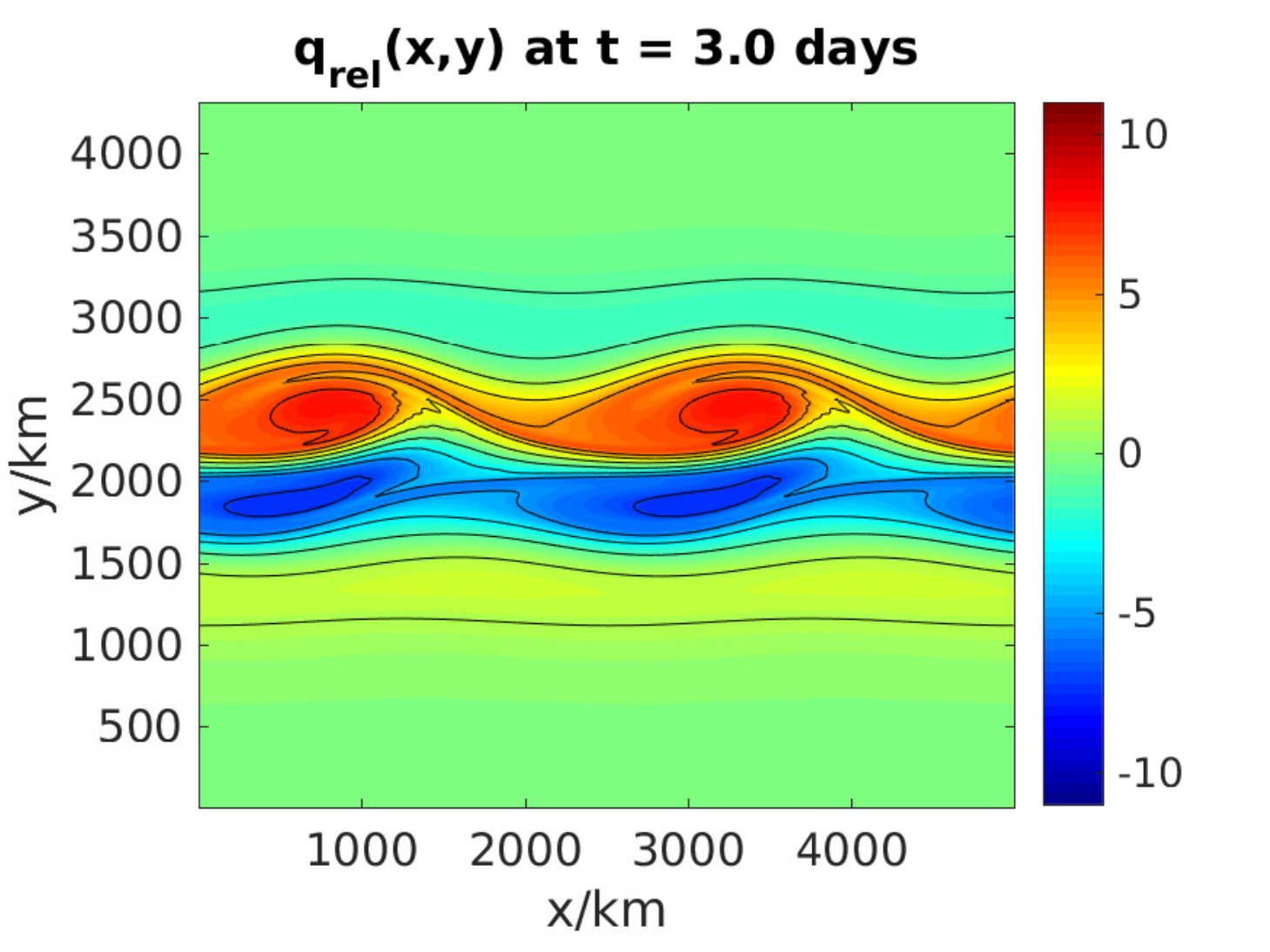}} &
      \hspace*{-.55cm} {\includegraphics[scale=0.38]{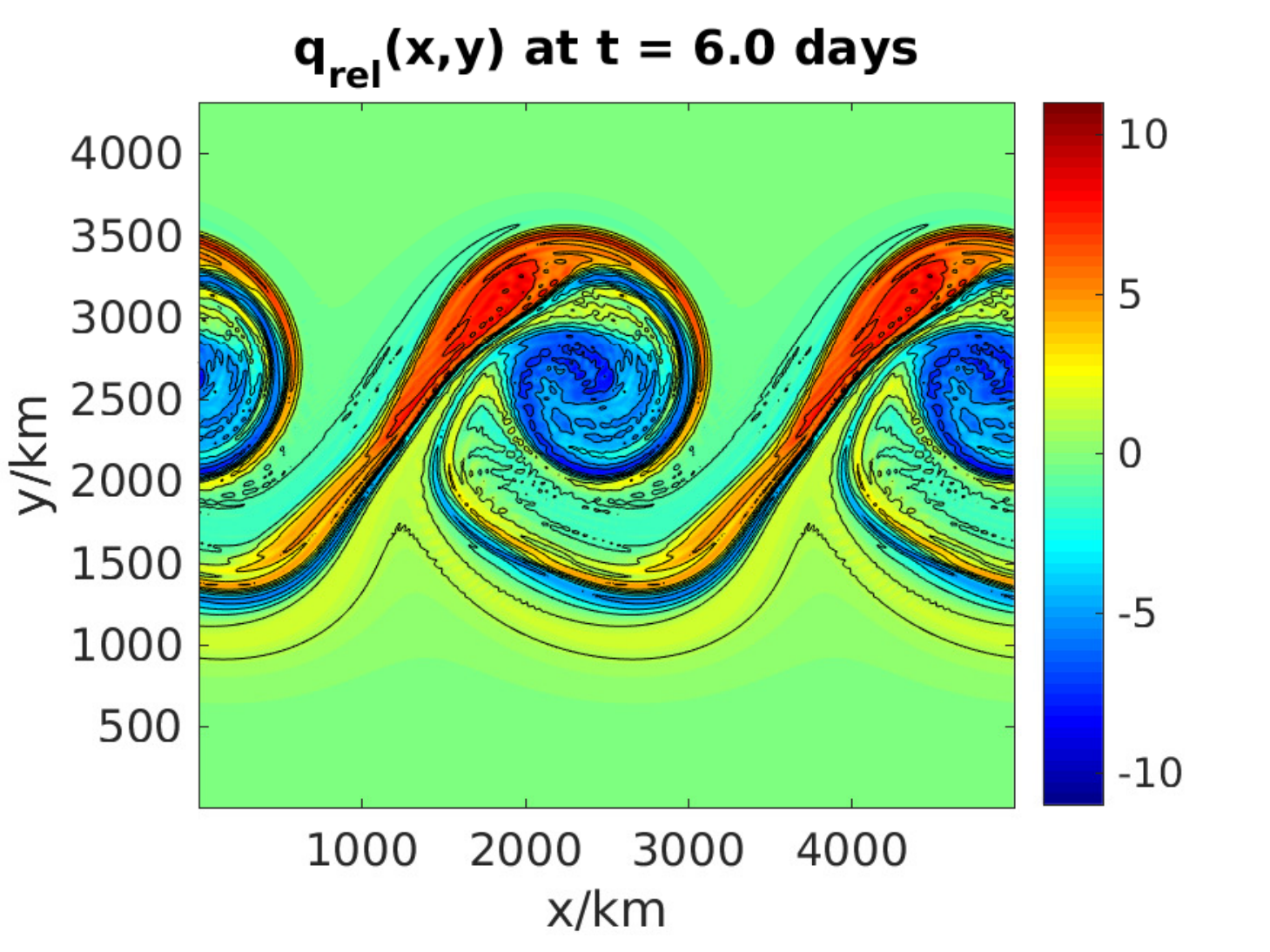}} &
      \hspace*{-.55cm} {\includegraphics[scale=0.38]{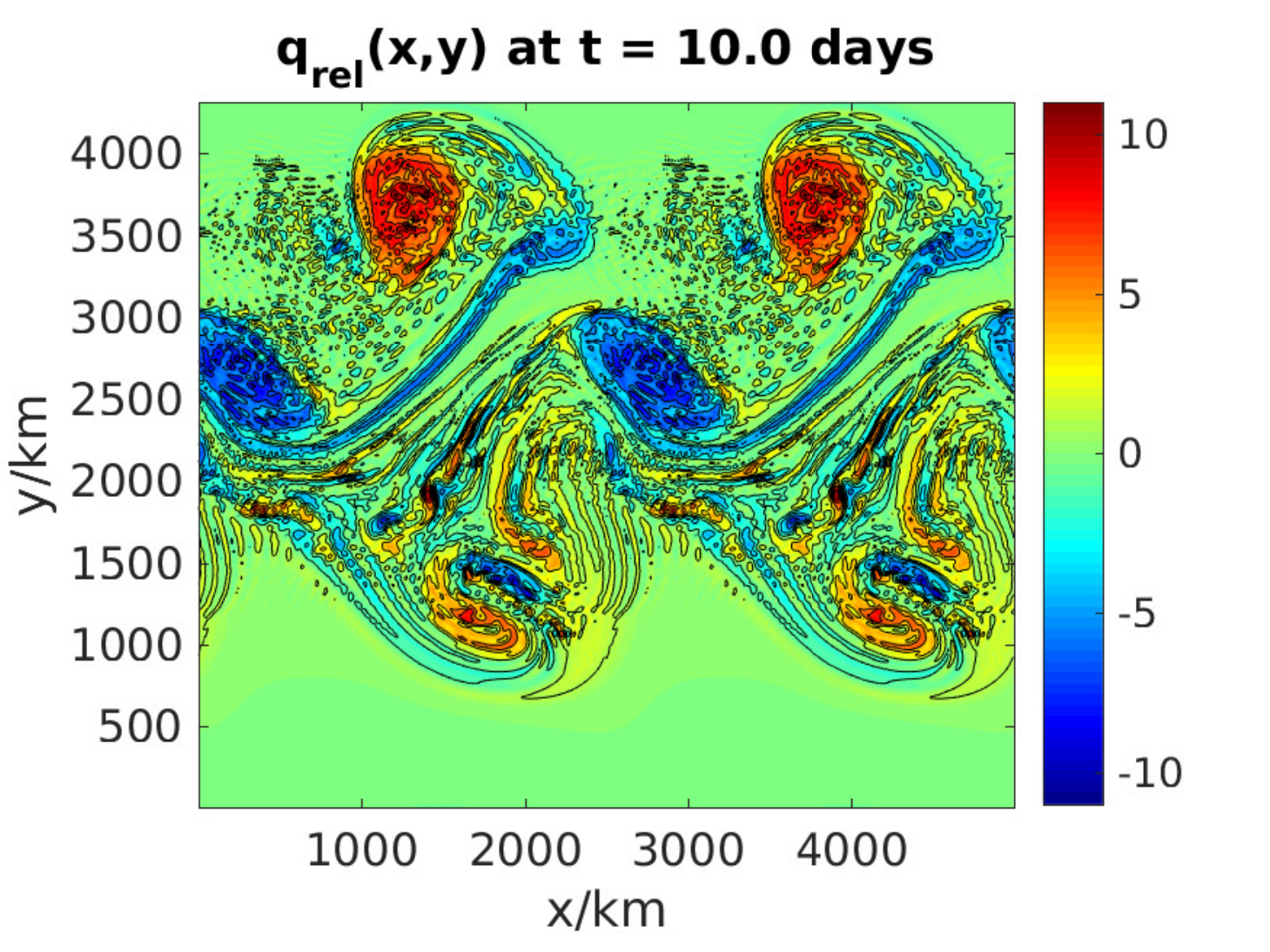}}  \\
      \hspace*{-.35cm} {\includegraphics[scale=0.38]{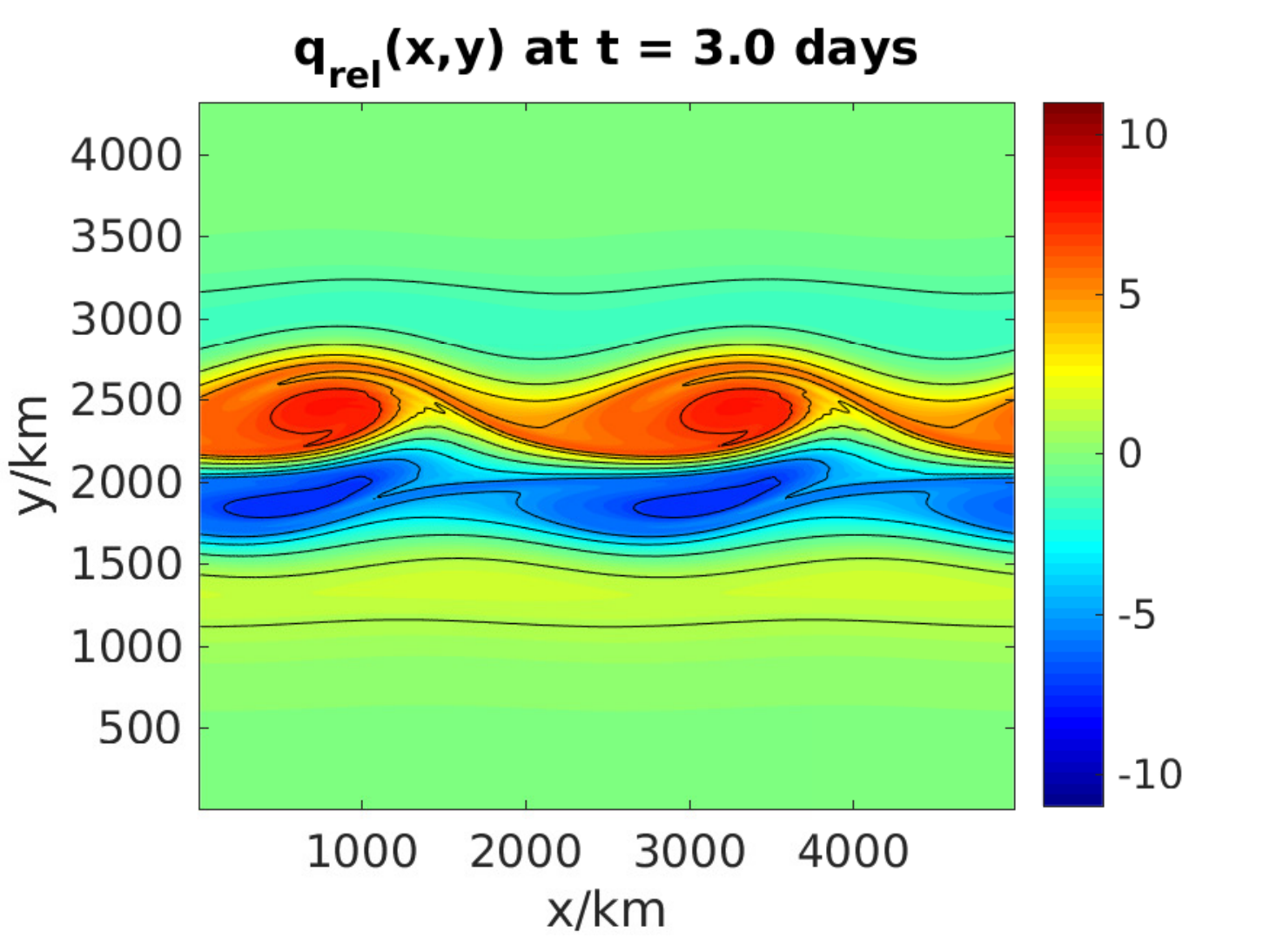}}  &
      \hspace*{-.55cm} {\includegraphics[scale=0.38]{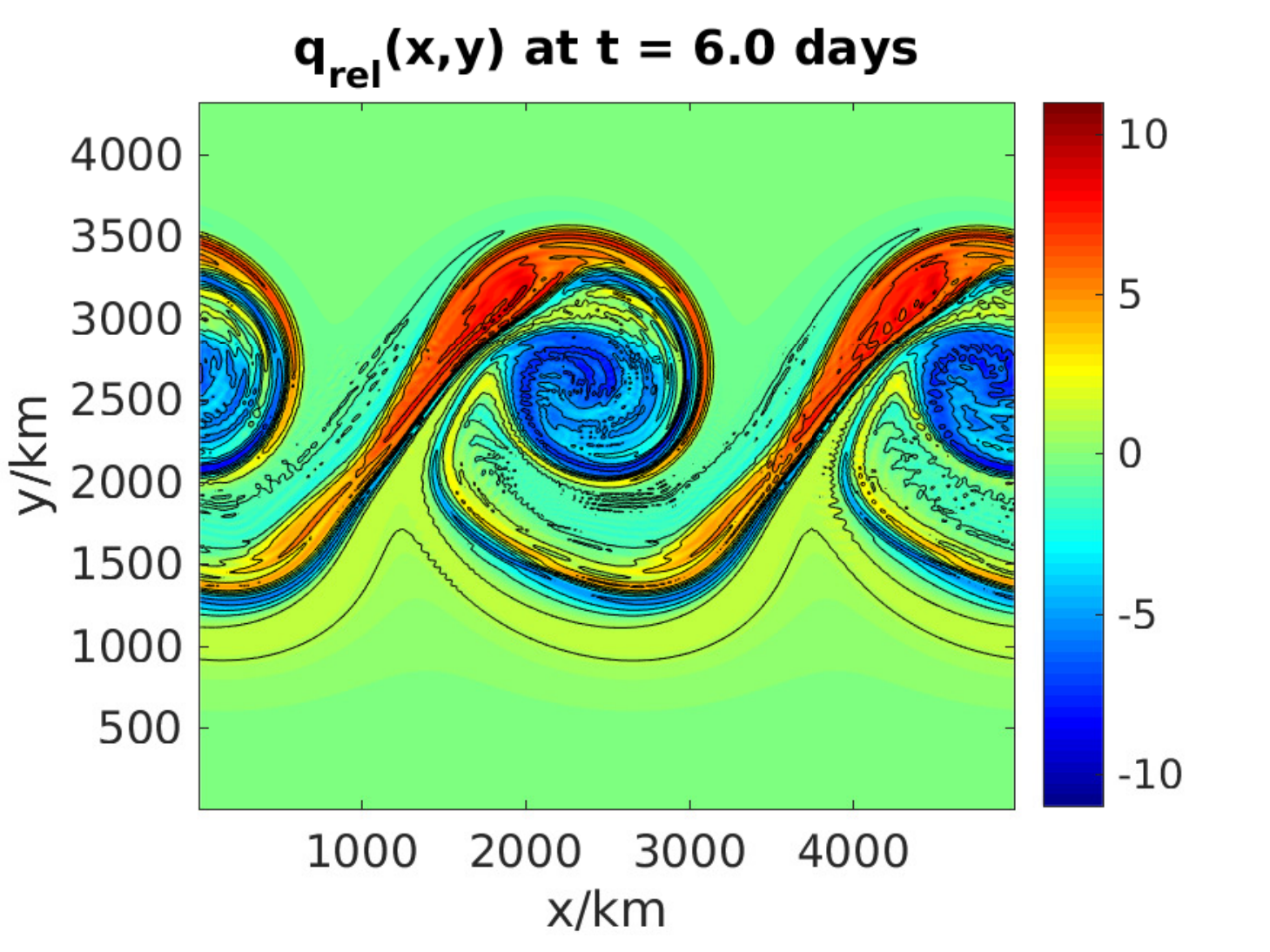}} &
      \hspace*{-.55cm} {\includegraphics[scale=0.38]{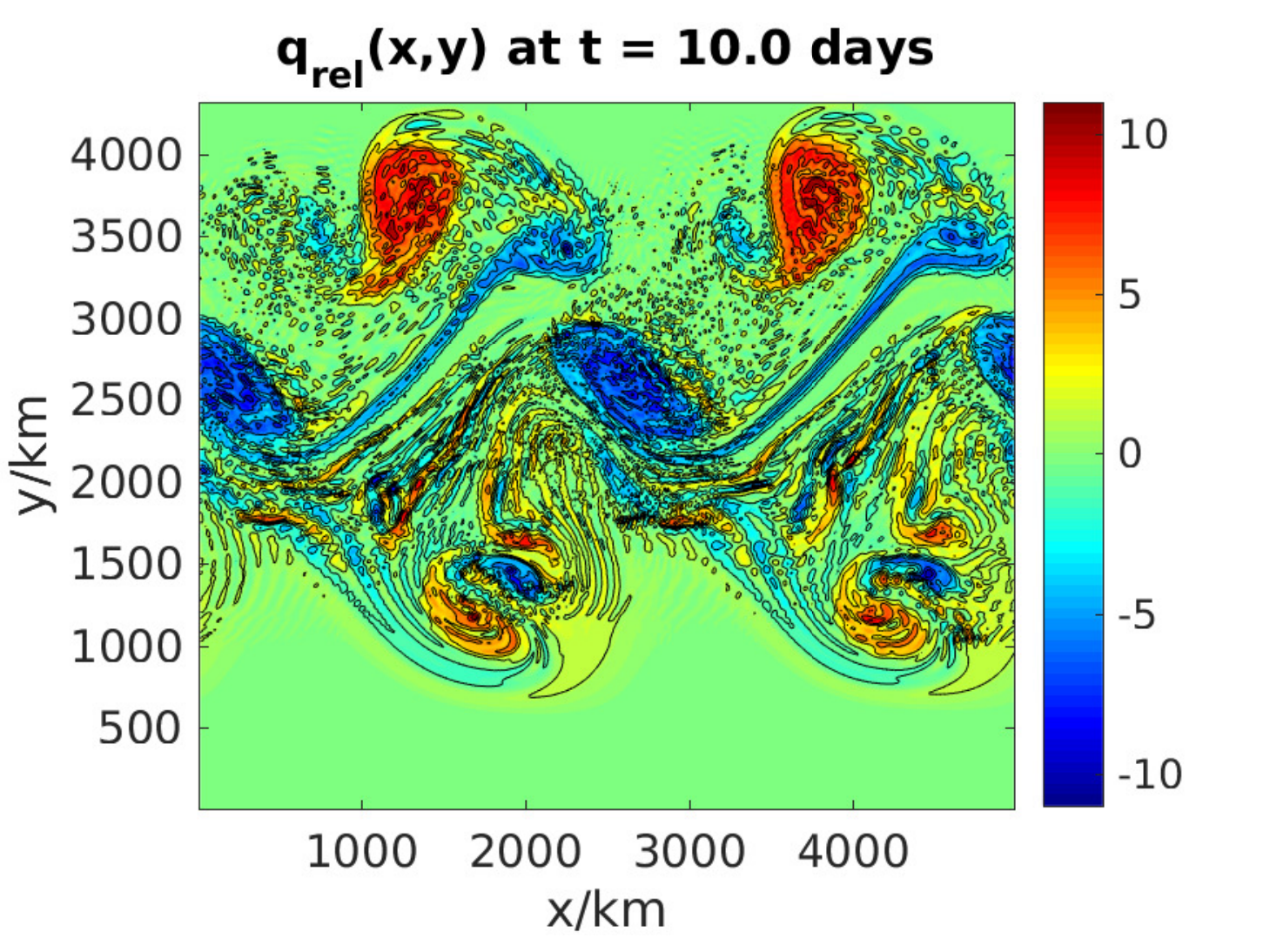}}  
      
  \end{tabular}
    \caption{Shear flow test case: snapshots of relative potential vorticity $q_{\rm rel}$
    on regular (upper row) and irregular (lower row)
    mesh with $2\cdot 256^2$ cells. Contours between $-11\,{\rm days^{-1}km^{-1}}$ and $11\,{\rm days^{-1}km^{-1}}$ 
    with interval of $2\,{\rm days^{-1}km^{-1}}$.}                                                                                             
  \label{fig_Z_shear_dynamics}
  \end{figure}
  
   \begin{figure}[t!]\centering
  \begin{tabular}{ccc} 
      \hspace*{-.25cm}{\includegraphics[scale=0.38]{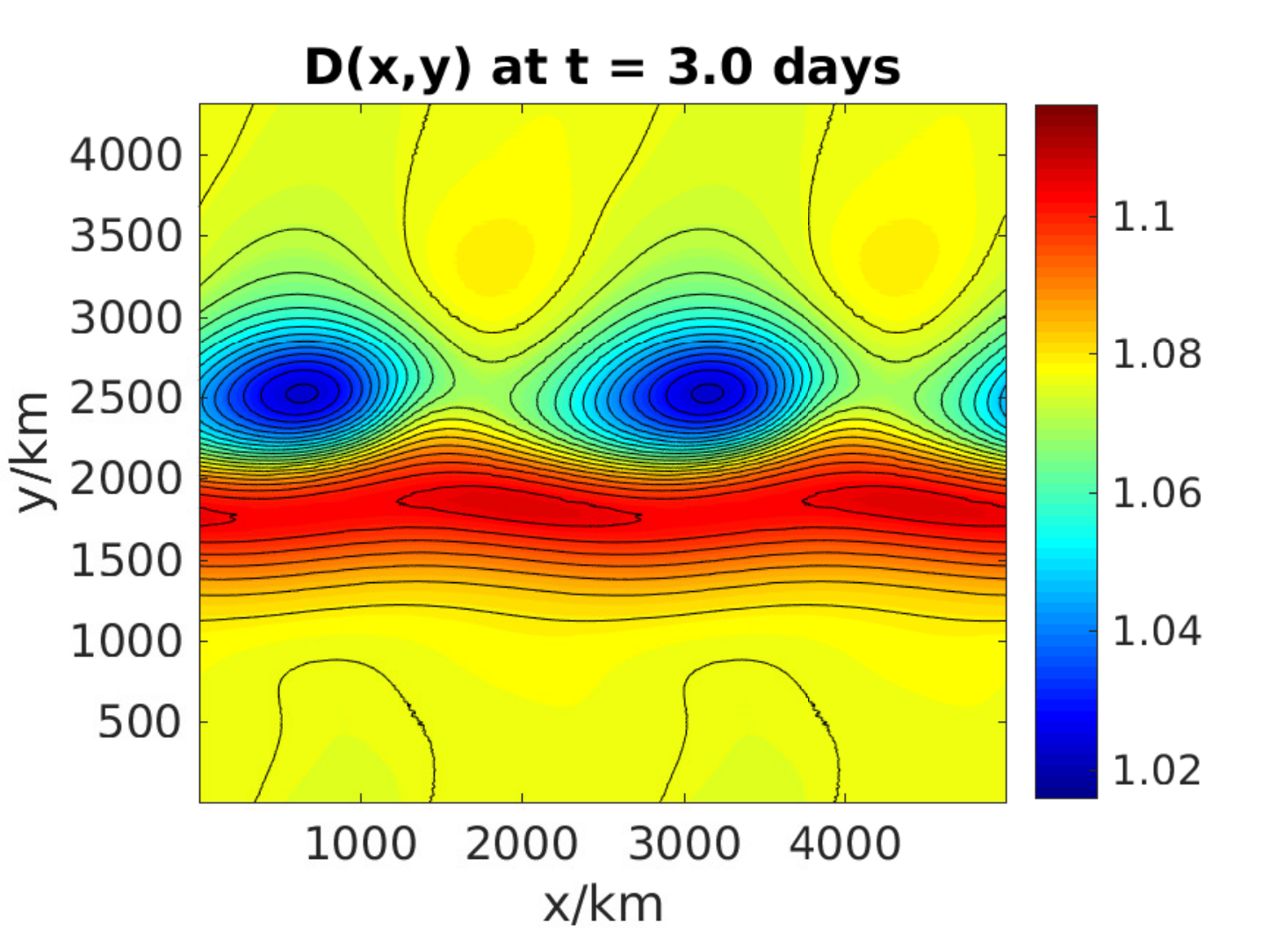}} &
      \hspace*{-.35cm}{\includegraphics[scale=0.38]{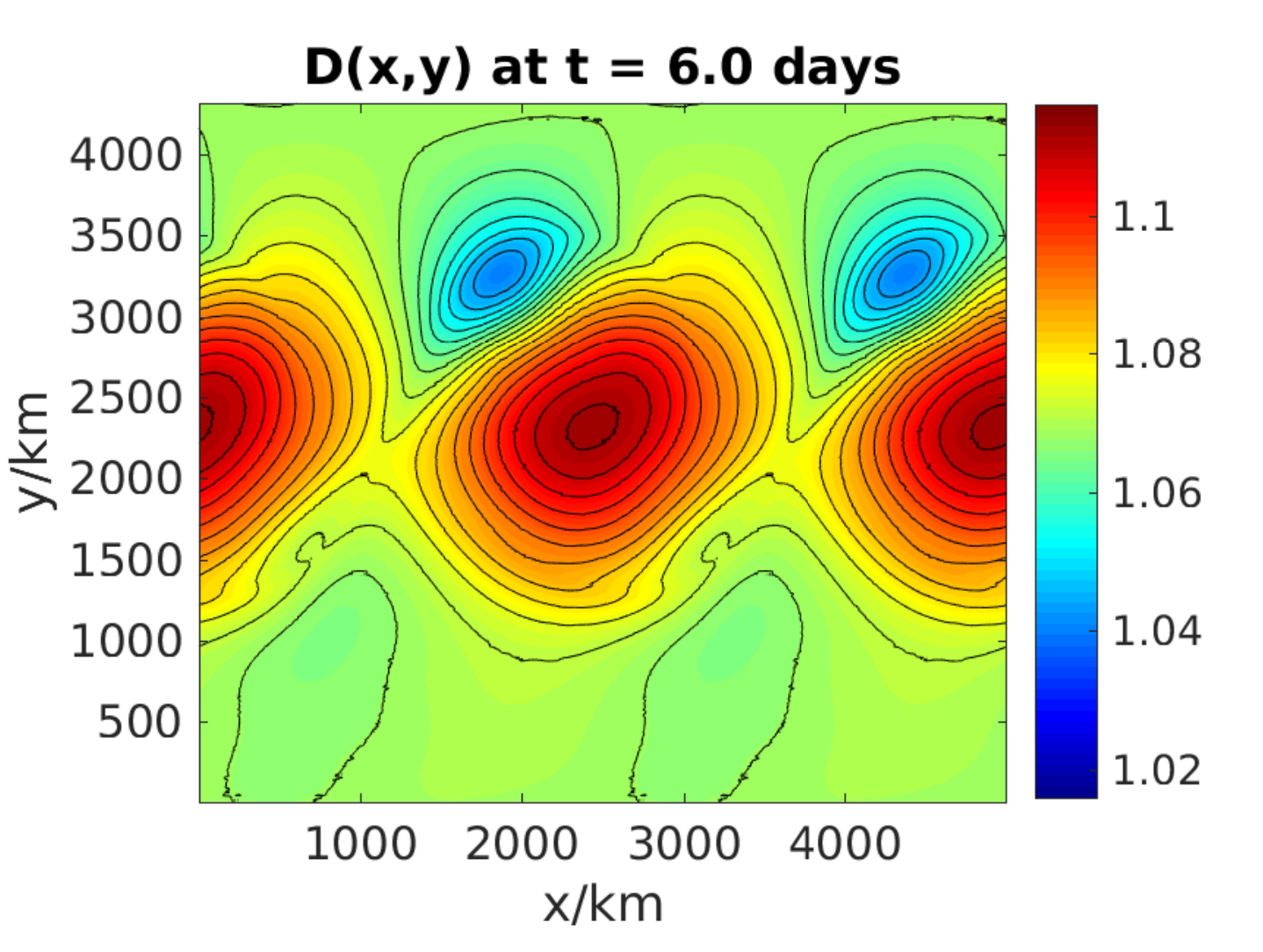}} &
      \hspace*{-.35cm}{\includegraphics[scale=0.38]{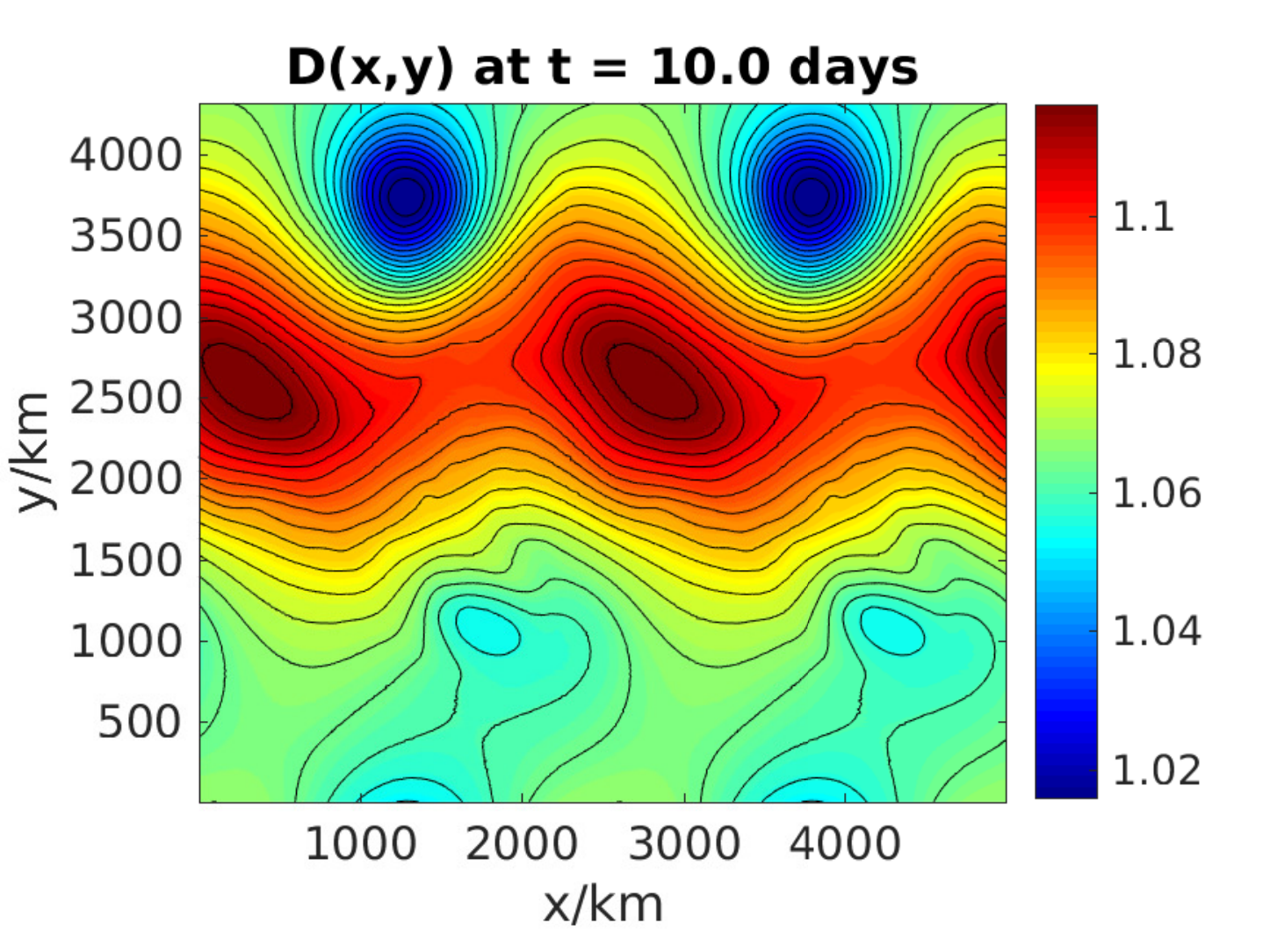}} \\
      \hspace*{-.25cm}{\includegraphics[scale=0.38]{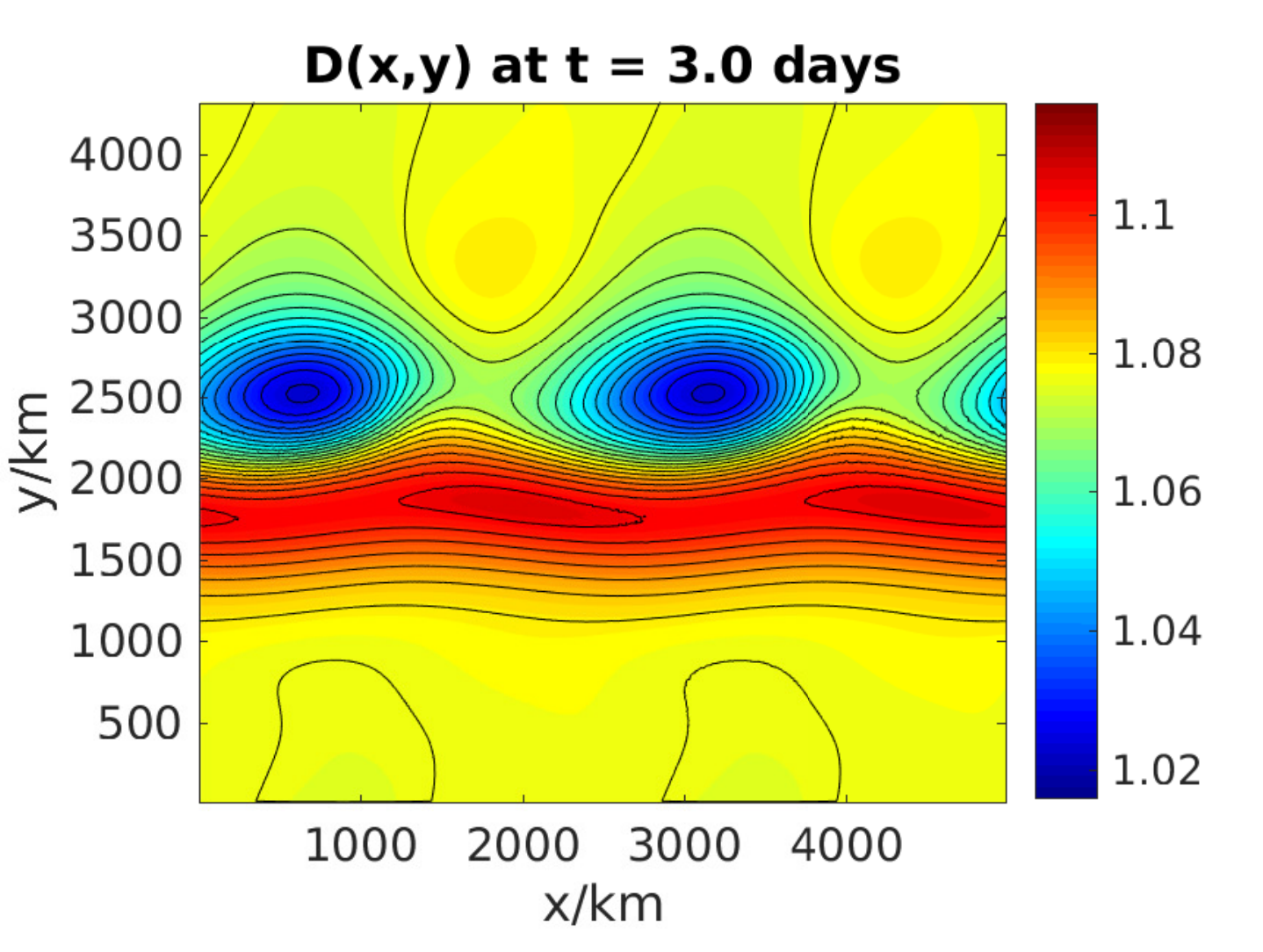}}    &
      \hspace*{-.35cm}{\includegraphics[scale=0.38]{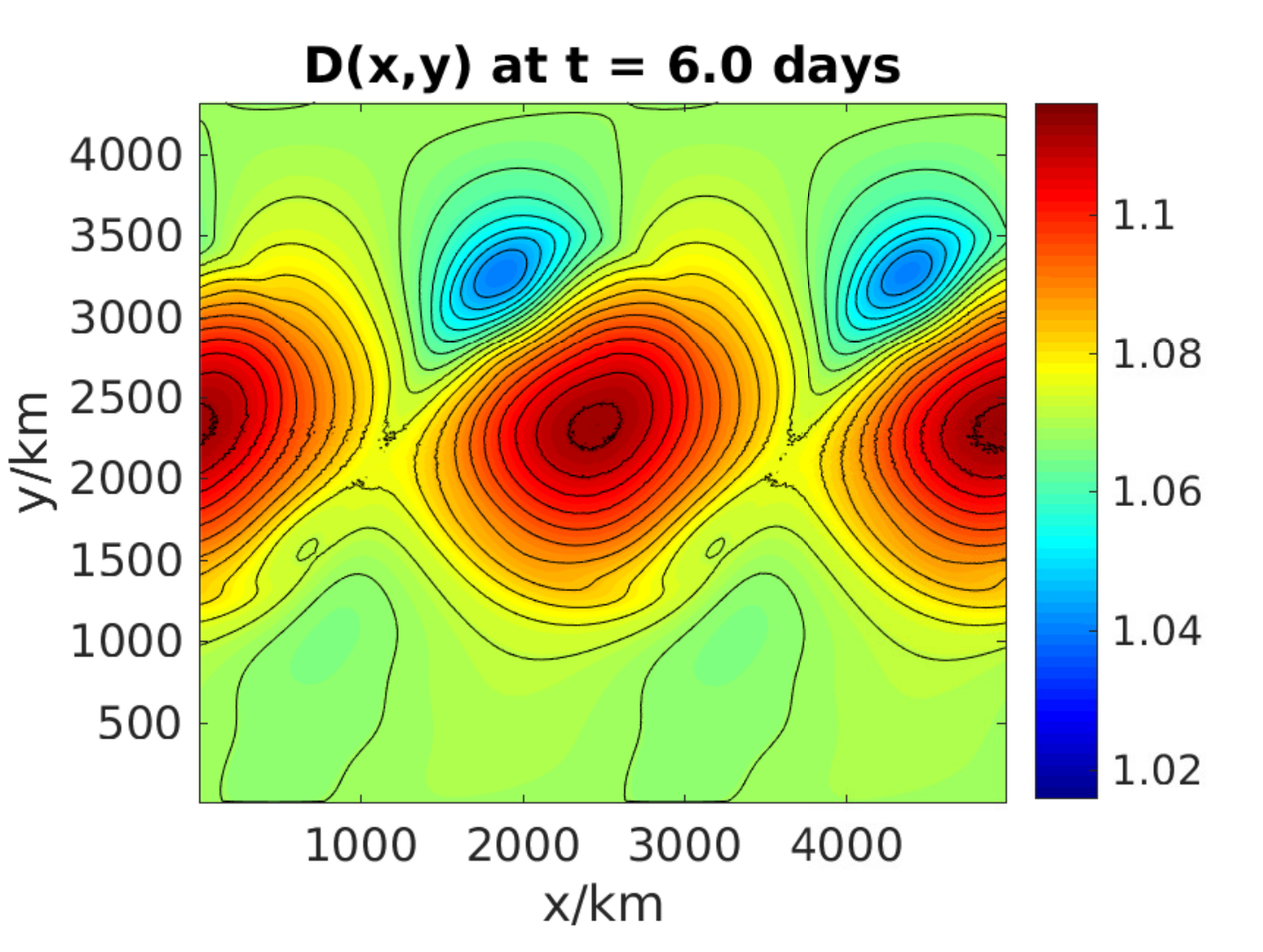}} &
      \hspace*{-.35cm}{\includegraphics[scale=0.38]{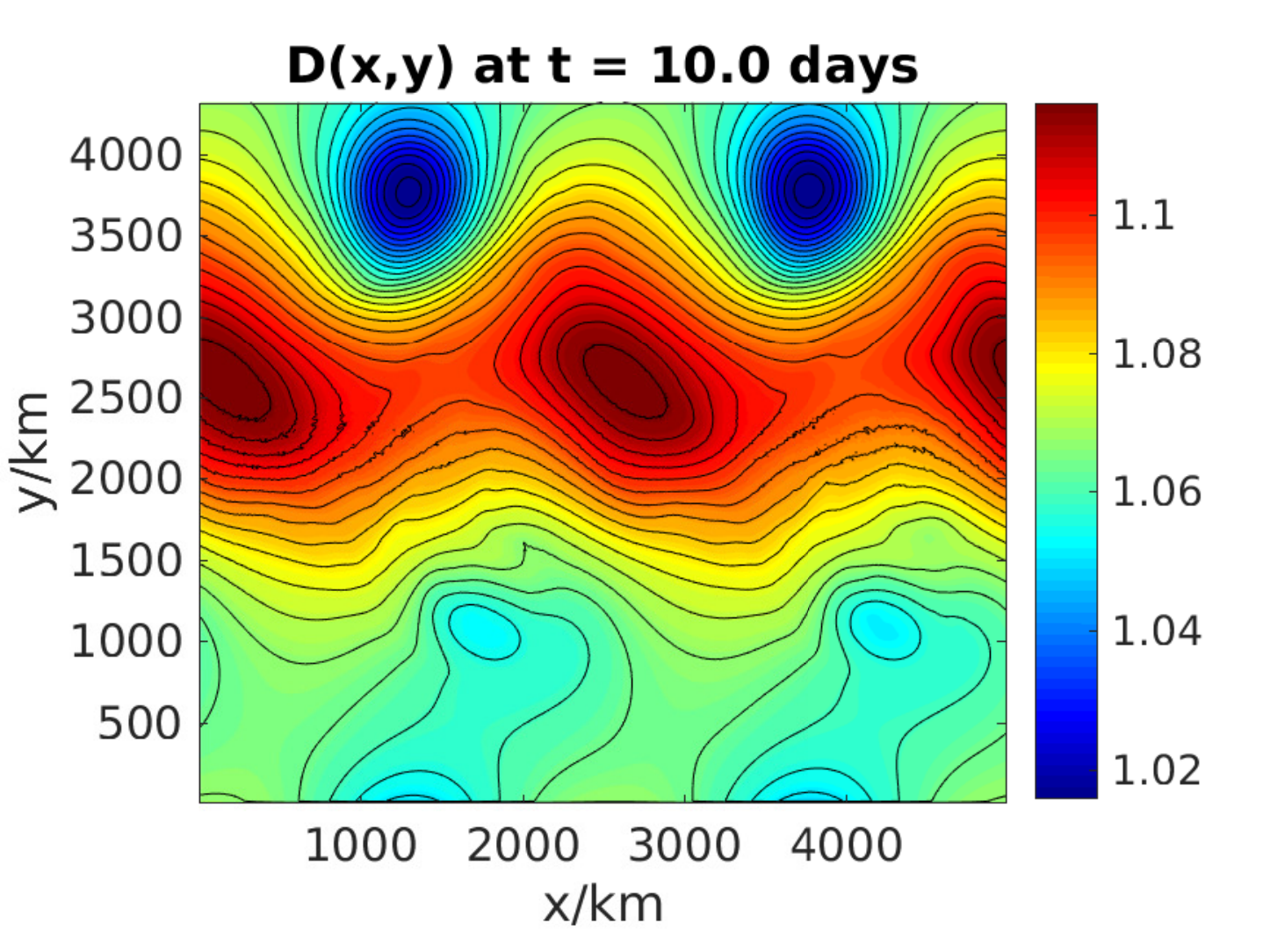}}  
  \end{tabular}
    \caption{Shear flow test case: snapshots of $D$ on regular (upper row) and irregular (lower row)
    mesh with $2\cdot 256^2$ cells. Contours between $-0.06\,{\rm km} +H_0$ and $0.04\,{\rm km} + H_0$ 
    with interval of $0.004\,{\rm km}$.}                                                                                             
  \label{fig_D_shear_dynamics}
  \end{figure}

  \begin{figure}[t!]\centering
  \begin{tabular}{cc} 
   \hspace*{-0cm}{\includegraphics[scale=0.45]{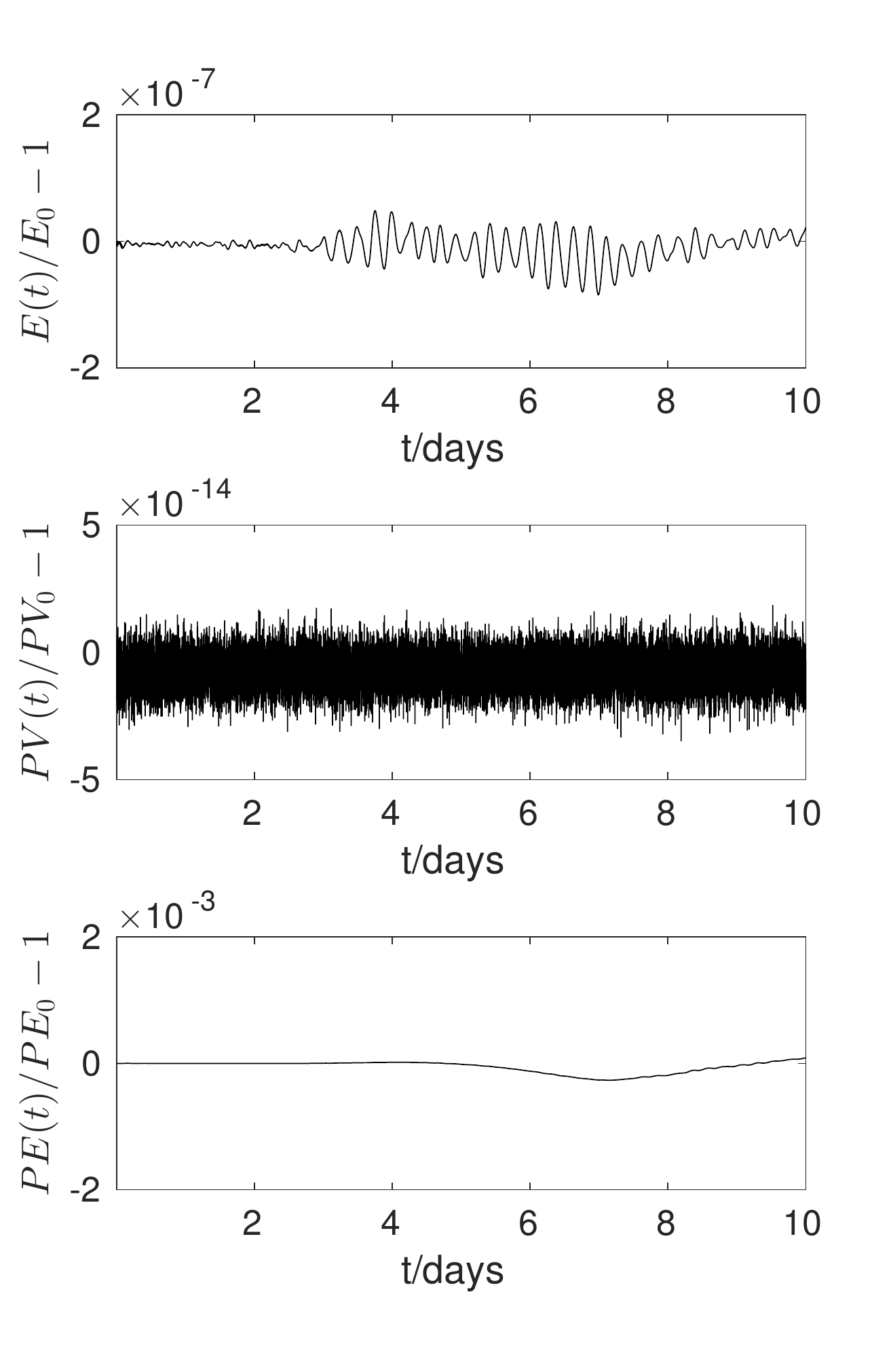}} &   
   \hspace*{-0cm}{\includegraphics[scale=0.45]{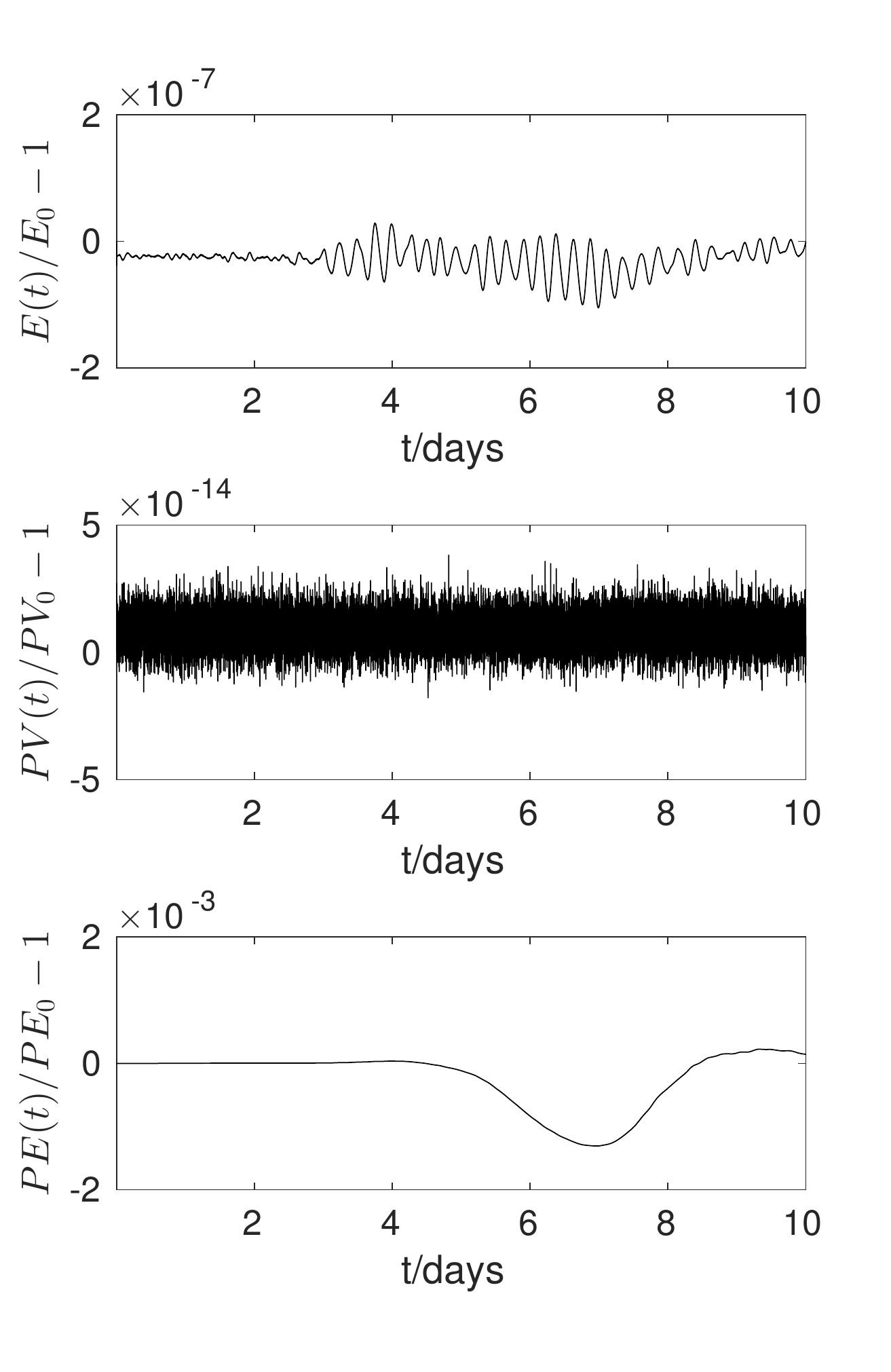}} 
  \end{tabular}
    \caption{Shear flow test case: relative errors of total energy $E(t)$ (upper row), 
    of mass-weighted potential vorticity $PV(t)$ (middle row), and potential enstrophy $PE(t)$ (lower row) for a 
    fluid in quasi-geostrophic regime for regular (left) and irregular (right) meshes with 
    $2 \cdot 256^2$ cells. }                                                                                             
  \label{fig_diag_shearflow}
  \end{figure}

 \section{Conclusions and outlook}

 This study provides a first step in the development of geometry preserving numerical integrators for 
 compressible fluids. We derived a variational discretization for compressible fluids 
 by extending the variational discretization framework developed in \cite{PaMuToKaMaDe2011} for incompressible 
 ideal fluids, based on a Lie group approximation of the group of volume preserving diffeomorphisms. 
 This extension was achieved by relaxing the volume preserving condition on the discrete diffeomorphisms and 
 imposing appropriate nonholonomic constraints that naturally follow from the relation between the 
 discrete and continuous velocities.
 Given a semidiscrete Lagrangian, the semidiscrete equations followed by applying the Lagrange-d'Alembert 
 principle in reduced Eulerian form, i.e., the Euler-Poincar\'e-d'Alembert equations. The resulting 
 semidiscrete equations are valid on any mesh discretization of the fluid domain, in 3D and 2D. We derived 
 them explicitly for 2D irregular simplicial meshes. In particular, a discrete Lie derivative operator 
 was obtained on such meshes. We then specialized our study by focusing on the case of the rotating shallow water equations.

 For this case, we numerically verified that our variational discretization
  (i) preserves stationary solutions of the shallow water equations, 
  (ii) conserves very accurately quantities of interest such as mass, total energy, 
  mass-weighted potential vorticity, and potential enstrophy, (iii) correctly represents nonlinear dynamics,
  and (iv) correctly represents (inertia-gravity) waves. 
  In more detail, simulating the time evolution of a lake at rest and of a steady isolated vortex, 
  we showed that the shallow water scheme conserves these stationary solutions to a high degree even 
  for long-term simulations. In addition, the numerical solutions of a stationary isolated vortex converge
  with, at least $1^{st}$-order, towards the exact solutions for both 
  regular or irregular computational meshes. For all test cases studied, mass-weighted potential vorticity and mass were conserved at machine precision
  while the total energy shows excellent long terms conservation properties 
  with a maximal error at the order of $10^{-7}$ that decreases further at 
  $1^{st}$-order rate for smaller time step sizes. 
  Moreover, the scheme correctly represents (inertia-gravity) waves as a comparison between
  numerically determined and theoretically predicted wave spectra confirmed. 
  The study of the nonlinear dynamics with respect to a dual vortex interaction and 
  a shear flow test case showed further that the scheme presents very accurately 
  the dynamics triggered by nonlinear interaction. The quality of the results is 
  similarly good for regular and irregular computational meshes. 
  The correctness of our simulations are underpinned by a comparison 
  to literature.

  Providing here a variational integrator for the 
  two dimensional shallow water equations, object of current and future 
  work is the extension of the variational discretization framework 
  to derive structure-preserving discretizations for fully three 
  dimensional compressible flows.

\section*{Acknowledgements}
The authors thank D. Cugnet, M. Desbrun, E. Gawlik, F. de Goes, D. Pavlov, P. Mullen, and V. Zeitlin for extremely useful 
discussions during the course of this work. WB and FGB were partially supported by the ANR project 
GEOMFLUID, ANR-14-CE23-0002-01; WB has received funding from the European Union's Horizon 2020
research and innovation programme under the Marie Sk\l odowska-Curie grant agreement No 657016.


\end{document}